\documentclass[reqno,a4paper,11pt]{amsart}
\usepackage[utf8]{inputenc}

\usepackage[english]{babel}

\usepackage[utf8]{inputenc}
\usepackage[T1]{fontenc}
\usepackage[normalem]{ulem}

\usepackage{amssymb}
\usepackage{amsmath,amsfonts}
\usepackage[tikz]{bclogo}
\usetikzlibrary{arrows, arrows.meta}
\usepackage[top=3cm,bottom=3cm,left=2.5cm,right=2.5cm]{geometry}

\usepackage{lmodern}
\usepackage{textcomp}
\usepackage{mathtools, bm}
\usepackage{amssymb, bm}
\usepackage[unq]{unique}
\usepackage{enumerate}
\usepackage{comment}
\usepackage[breaklinks]{hyperref}

\usepackage[numbers]{natbib}  

\usepackage{graphicx} 

\usepackage{amsthm}

\usepackage{tikz}
\usepackage{caption}
\usepackage{subcaption}

\usepackage{cleveref}

\usepackage{bbm}

\usepackage{enumitem}
\usepackage{thmtools}
\usepackage{thm-restate}

\newtheorem{theorem}{Theorem}[section]
\newtheorem{lemma}[theorem]{Lemma}
\newtheorem{proposition}[theorem]{Proposition}

\newtheorem{claim}[theorem]{Claim}
\newtheorem{Corollary}[theorem]{Corollary}

\theoremstyle{definition}
\newtheorem{remark}[theorem]{Remark}
\newtheorem{conjecture}[theorem]{Conjecture}

\newtheorem{definition}[theorem]{Definition}

\DeclareMathOperator{\Var}{Var}
\DeclareMathOperator{\Surj}{Surj}

\DeclareMathOperator{\Cov}{Cov}
\DeclareMathOperator{\surp}{sp}
\newcommand{\vx}{\mathbf{x}}

\newcommand{\eps}{\varepsilon}
\DeclareMathOperator{\ind}{\mathbf{1}}

\title{Nearly tight bounds for MaxCut in hypergraphs}
\date{}

\author{Oliver Janzer$^*$}
\author{Julien Portier$^\dagger$}

\begin{document}

\maketitle

\begin{abstract}
    An $r$-cut of a $k$-uniform hypergraph is a partition of its vertex set into $r$ parts, and the size of the cut is the number of edges which have at least one vertex in each part. The study of the possible size of the largest $r$-cut in a $k$-uniform hypergraph was initiated by Erd\H os and Kleitman in 1968.
    For graphs, a celebrated result of Edwards states that every $m$-edge graph has a $2$-cut of size $m/2+\Omega(m^{1/2})$, which is sharp. 
    In other words, there exists a cut which exceeds the expected size of a random cut by the order of $m^{1/2}$.
    Conlon, Fox, Kwan and Sudakov proved that any $k$-uniform hypergraph with $m$ edges has an $r$-cut whose size is $\Omega(m^{5/9})$ larger than the expected size of a random $r$-cut, provided that $k \geq 4$ or $r \geq 3$.
    They further conjectured that this can be improved to $\Omega(m^{2/3})$, which would be sharp.
    Recently, R{\"a}ty and Tomon improved the bound $m^{5/9}$ to $m^{3/5-o(1)}$ when $r \in \{ k-1,k\}$.
    Using a novel approach, we prove the following approximate version of the Conlon--Fox--Kwan--Sudakov conjecture: for each $\varepsilon>0$, there is some $k_0=k_0(\varepsilon)$ such that for all $k>k_0$ and $2\leq r\leq k$, in every $k$-uniform hypergraph with $m$ edges there exists an $r$-cut exceeding the random one by $\Omega(m^{2/3-\varepsilon})$. Moreover, we show that (if $k\geq 4$ or $r\geq 3$) every $k$-uniform linear hypergraph has an $r$-cut exceeding the random one by $\Omega(m^{3/4})$, which is tight and proves a conjecture of R\"aty and Tomon.
    
\end{abstract}

\renewcommand{\thefootnote}{\fnsymbol{footnote}}
\footnotetext[1]{Institute of Mathematics, EPFL, Lausanne, Switzerland. \texttt{oliver.janzer@epfl.ch}}
\footnotetext[2]{Institute of Mathematics, EPFL, Lausanne, Switzerland. \texttt{julien.portier@epfl.ch}}
\renewcommand{\thefootnote}{\arabic{footnote}}

\section{Introduction}

Given a graph $G$, a \emph{cut} refers to a partition of the vertex set into two subsets, with the \emph{size} of the cut defined as the number of edges that have one vertex in each of these subsets.
The MaxCut of a graph $G$ is the maximum size of a cut over all possible bipartitions of the vertex set of $G$, and is denoted by $\mathrm{mc}(G)$.
The MaxCut problem has been central in both combinatorics and theoretical computer science, and appears in Karp's famous list of 21 NP-complete problems \cite{karp1972reducibility}. 
As a simple probabilistic argument shows that $\mathrm{mc}(G) \geq m/2$ for any graph $G$ with $m$ edges, the focus has been on proving bounds on the \emph{surplus} $\surp(G)$ of a graph, defined as $\surp(G)= \mathrm{mc}(G)-m/2$.
A celebrated result by Edwards \cite{edwards1973some,edwards1975improved}
shows that $\surp(G) \geq (\sqrt{8m+1}-1)/8$, which is tight for a complete graph with an odd number of vertices. Very recently, an inverse theorem for MaxCut has been established by Balla, Hambardzumyan and Tomon \cite{balla2025factorization}, stating that if a graph $G$ with $m$ edges has $\surp(G)=O(\sqrt{m})$, then $G$ contains a clique of order $\Omega(\sqrt{m})$.
Much work has been devoted to establishing stronger lower bounds for $\surp(G)$ under the assumption that $G$ is $H$-free, for some fixed graph $H$. 
Those efforts were initiated by Erd\H{o}s and Lov\'{a}sz in the 70s (see \cite{erdHos1975problems}), and a substantial amount of research has since focused in this direction, see for instance \cite{alon1996bipartite,alon2003maximum,alon2005maxcut,carlson2021lower,glock2023new,balla2024maxcut,jin2025small}. Despite remarkable recent breakthroughs on this topic with far-reaching applications beyond Graph Theory (see, e.g. \cite{jin2025small}), tight bounds seem to remain out of reach in most cases. \\

A central problem in the area concerns the extension of Edwards' result to hypergraphs. An \emph{r-cut} of a $k$-uniform hypergraph $G$ is a partition of the vertex set of $G$ into $r$ parts, and the \emph{size} of the cut is defined as the number of hyperedges having at least one vertex in each part of the cut. 
We are then interested in the maximum size of an $r$-cut over all possible partitions of the vertex set of $G$, which is denoted by $\mathrm{mc}_r(G)$.
As shown by Erd\H{o}s and Kleitman \cite{erdos1968coloring}, by assigning each vertex to one of the $r$ parts independently and uniformly at random, it follows that $\mathrm{mc}_r(G) \geq \frac{S(k,r)r!}{r^k}m$ for every graph $G$ with $m$ edges, where $S(k,r)$ is the Stirling number of second kind, denoting the number of unlabelled partitions of $\{1, \dots, k\}$ into $r$ nonempty sets.
The \emph{r-surplus} (or \emph{r-excess}) of $G$ is then defined as $\surp_r(G)=\mathrm{mc}_r(G)-\frac{S(k,r)r!}{r^k}m$. Similarly, we say that the surplus of an $r$-cut in a $k$-graph with $m$ edges is the size of the cut minus $\frac{S(k,r)r!}{r^k}m$.
Conlon, Fox, Kwan and Sudakov \cite{conlon2019hypergraph} showed that for hypergraphs, Edwards' bound can be significantly improved. 

\begin{theorem}[Conlon--Fox--Kwan--Sudakov \cite{conlon2019hypergraph}]
    For every $2 \leq r \leq k$ with $(r,k) \neq (2,2)$ and $(r,k) \neq (2,3)$, every $k$-uniform hypergraph with $m$ edges has an $r$-cut of surplus $\Omega(m^{5/9})$.
\end{theorem}

They noted that for $(r,k)=(2,3)$, Steiner triple systems show that there are $3$-uniform hypergraphs on $m$ edges with $2$-surplus $\Theta(\sqrt{m})$. They furthermore observed that with high probability the binomial random $k$-graph $G_k(n,n^{3-k})$ has $r$-surplus $O(m^{2/3})$, where $m$ is the number of its edges. This disproved a conjecture of Scott \cite{scott2005judicious}, which had predicted that the complete $k$-graph has the smallest $2$-surplus.

Conlon et al. \cite{conlon2019hypergraph} conjectured that this random hypergraph is asymptotically optimal whenever $(r,k)\neq (2,2),(2,3)$.

\begin{conjecture}[Conlon--Fox--Kwan--Sudakov \cite{conlon2019hypergraph}]
\label{conj:CFKS}
For any $2 \leq r \leq k$ with either $k\geq 4$ or  $r \geq 3$, every $k$-uniform hypergraph with $m$ edges has an $r$-cut of surplus $\Omega(m^{2/3})$.
\end{conjecture}

In the particular case where $r \in \{k -1, k \}$, R{\"a}ty and Tomon \cite{raty2024large} recently improved Conlon, Fox, Kwan and Sudakov's bound to $\Omega(m^{3/5-o(1)})$ by using spectral techniques.

\begin{theorem}[R\"aty--Tomon \cite{raty2024large}] \label{thm:3/5}
    If $r\in \{k-1,k\}$ and $r\geq 3$, then every $k$-uniform hypergraph with $m$ edges has an $r$-cut of surplus $\Omega(m^{3/5-o(1)})$.
\end{theorem}

R\"aty and Tomon pointed out that their methods do not extend to the range $2\leq r\leq k-2$.
Obtaining strong results in the case $r=2$ is of particular interest for several reasons. Firstly, it is this case that has been the most extensively studied in Theoretical Computer Science (see, e.g. \cite{haastad2001some,guruswami2004inapproximability}). This problem is known as \emph{Max $E_k$-Set Splitting}, or, in the case the hypegraph is not uniform, as \emph{Max Set Splitting}. Secondly, this case has close connections to other well-studied problems such as negative (or positive) discrepancy \cite{bollobas2006discrepancy}, bisection width \cite{raty2024bisection} and hypergraph colourings \cite{lovasz1973coverings}. In the concluding remarks we will mention our upcoming work which extends the techniques in this paper to the setting of positive discrepancy and bisection width.

Using a novel approach, we prove the following approximate version of the Conlon--Fox--Kwan--Sudakov conjecture for all values of $r$.

\begin{theorem}
\label{thm:Hypergraph-Excess}
    For any $\eps>0$, there exists some $k_0=k_0(\eps)$ such that for any $k>k_0$ and $2\leq r\leq k$, every $k$-uniform hypergraph with $m$ edges has an $r$-cut of surplus $\Omega(m^{2/3-\eps})$.
\end{theorem}

In the case of linear hypergraphs, still for $r \in \{k -1, k \}$, R{\"a}ty and Tomon \cite{raty2024large} proved a stronger bound $\Omega(m^{3/4-o(1)})$. They noted that this is asymptotically optimal, by constructing $k$-graphs with $r$-surplus $O(m^{3/4})$ in which each pair of vertices is in at most $O(\log m)$ edges (i.e., the hypergraphs are nearly linear). They conjectured that the $o(1)$ term can be removed and that the result can be extended to all $2\leq r\leq k$.

We completely resolve this conjecture.

\begin{theorem} \label{thm:linear}
    For any $2 \leq r \leq k$ with either $k \geq 4$ or $r\geq 3$, every $k$-uniform linear hypergraph with $m$ edges has an $r$-cut of surplus $\Omega(m^{3/4})$.
\end{theorem}

We also show that this is tight up to a constant factor by constructing linear (not just nearly linear) $k$-graphs with $r$-surplus $O(m^{3/4})$.

\begin{proposition}
\label{thm:Upper-Bound-Linear-Hypergraphs}
    Fix some integers $2\leq r \leq k$. There is a constant $\alpha=\alpha(k)$ such that, for every $m \in \mathbb{N}$, there exists an $m$-edge $k$-uniform linear hypergraph with $r$-surplus at most $\alpha m^{3/4}$.
\end{proposition}

Finally, while Conlon, Fox, Kwan and Sudakov's construction of a $4$-uniform hypergraph with $2$-surplus $O(m^{2/3})$ is probabilistic, we show that highly structured constructions can achieve the same bound, which perhaps demonstrates the difficulty of proving tight lower bounds.
Recall that an $(n,s,2)$-Steiner system is an $n$-element set $V$ together with a set of $s$-element subsets of $V$ (called \emph{blocks}) with the property that each pair of elements in $V$ is contained in exactly one block. Note that finite projective planes are examples of such Steiner systems with $s=\Theta(\sqrt{n})$.

\begin{restatable}{proposition}{blockconstruction}
\label{prop:Alternative-Construction-4-2}
    Let $s\geq 4$ and let $S_1, \dots, S_t$ be the blocks of an $(n,s,2)$-Steiner system on ground set $V$.
    Let $G$ be the $4$-uniform hypergraph on vertex set $V$ obtained by placing a complete $4$-graph in each $S_i$.
    Then $e(G)=\Theta(n^2 s^2)$ and $\surp(G)= O(n^2+ns^2)$. In particular, when $s=\Theta(\sqrt{n})$ (which holds for projective planes), we have $\surp(G)=O(e(G)^{2/3})$.
\end{restatable}

Note that in any $(n,s,2)$-Steiner system, we have either $s=n$ or $s=O(\sqrt{n})$, so in almost all cases the $n^2$ term dominates the $ns^2$ term.

\section{Proof ideas}

While our proofs use certain ideas from both \cite{conlon2019hypergraph} and \cite{raty2024large}, they are mostly different from these previous arguments. One similarity to \cite{raty2024large} is that we also rely on the semidefinite programming (SDP) approach. This influential technique was developed by Goemans and Williamson \cite{goemans1995improved} to find approximation algorithms for MaxCut in graphs. Carlson et al. \cite{carlson2021lower} adapted this method to prove lower bounds for the MaxCut of $H$-free graphs, and this approach was further developed in a series of papers (see, e.g. \cite{glock2023new,balla2024maxcut,raty2023positive}).

Let us first recall the basic SDP approach, after which we explain how we will use it and what will be our main innovations. Note that finding a cut with large surplus in a graph $G$ amounts to finding an assignment $f:V(G)\rightarrow \{-1,+1\}$ such that $\sum_{uv\in E(G)} f(u)f(v)$ is quite negative. The SDP approach aims to instead find a unit vector $\vx_u\in \mathbb{R}^N$ for each $u\in V(G)$ (where $N$ is arbitrary but is the same for all $u\in V(G)$) such that $\sum_{uv\in E(G)} \langle \vx_u,\vx_v \rangle$ is quite negative. Clearly, the minimum of this ``relaxed'' optimization problem is no larger than the minimum for the original problem. Moreover, a powerful generalization of the classical Grothendieck inequality (proved in \cite{megretski2001relaxations,charikar2004maximizing,alon2006quadratic}) states that the solution of these two optimization problems are the same within a $\log(|V(G)|)$ factor, so, to prove lower bounds for the surplus of $G$, it suffices to construct unit vectors $\vx_u$ for each $u\in V(G)$ such that $\sum_{uv\in E(G)} \langle \vx_u,\vx_v \rangle$ is quite negative.

At this point we remark that there have been two rather different approaches for constructing the vectors $\vx_u$. The first approach is combinatorial, using the edges of the graph as an input. This was the approach taken by \cite{carlson2021lower} and it was significantly further developed in \cite{glock2023new}. The second approach uses algebraic and spectral information to define the vectors: the algebraic viewpoint was first used in \cite{balla2024maxcut}, and a new, spectral approach was introduced in the influential paper \cite{raty2023positive} by R\"aty, Sudakov and Tomon. It is the spectral approach which was used also in the paper of R\"aty and Tomon \cite{raty2024large} on MaxCut in hypergraphs.

\subsection{The basics of our approach}

In this paper, we use a combinatorial/probabilistic SDP approach, combined with ingredients from modular arithmetics. Moreover, we develop new tools that make the SDP technique widely applicable in the setting of hypergraphs. Let us first see how finding large $2$-cuts in graphs (or hypergraphs) can help finding large $r$-cuts in $k$-graphs. We start with the case $r=k$, as this one turns out to be the easiest to deal with. We also assume that $k\geq 4$; the $k=3$ case requires a slightly different model. By simple reductions, finding a large $k$-cut in a $k$-graph $H$ can be reduced to finding a large $2$-cut in a certain auxiliary random multigraph $G$ associated to $H$. The reason is as follows.

Partition the vertex set of $H$ as $V\cup C$, where each vertex in $H$ is independently placed in $V$ with probability $2/k$ and in $C$ with probability $(k-2)/k$. We will commit to select the parts $1,\dots,k-2$ in our $k$-cut from $C$ and the parts $k-1$ and $k$ from $V$. Note that at this stage we may restrict our attention to those hyperedges $e$ in $H$ which have precisely $2$ vertices in $V$ (and $k-2$ vertices in $C$), otherwise $e$ will certainly not be cut. The next step is to expose the partition on $C$. After this, we are only interested in edges $e$ whose $k-2$ vertices in $C$ were ``cut perfectly'', i.e. which had one vertex assigned to each part in $[k-2]$. It is only these edges which may still be cut, and our goal is to maximize (by choosing the best $2$-partition in $V$) the number of such edges which end up being cut.

Using this procedure, our task can be reformulated as follows. Let $G_0$ be a multigraph on vertex set $V$ in which each edge has a set of $k-2$ colours from a set $C$. 
Two edges in $G_0$ with the same vertex set can have different sets of colours. To each colour $c\in C$, assign a number from $[k-2]$ uniformly at random. Now let $G$ be the (random) submultigraph of $G_0$ obtained by keeping only those edges whose $k-2$ colours receive all the $k-2$ numbers in $[k-2]$. To prove good lower bounds on the $k$-surplus of $k$-graphs, it suffices to prove good lower bounds for the expected $2$-surplus of the random multigraph $G$.

We will now explain our approach for defining the vectors $\vx_u$, for each $u\in V$. We start with a naive version, which is not sufficient to prove our main results on its own (but could be used to give an alternative proof of Theorem \ref{thm:3/5} by R\"aty and Tomon). We will later explain our two main ideas, which allow us to prove the stronger results we have in this paper. More precisely, the first main idea, presented below in Section~\ref{sec:outline modular}, allows us to prove an $m^{2/3-o_{k\rightarrow \infty}(1)}$ lower bound for the $k$-surplus in $k$-graphs, while the second main idea, presented in Section \ref{sec:outline hypergraph SDP}, allows us to extend all our results to general $r$-cuts in $k$-graphs (and has further applications, to positive discrepancy and bisection width, that are closely related to $2$-cuts and will be discussed in the concluding remarks).

For an edge $e \in E(G_0)$ with vertices $u,v$, we let $B(e,u,v)$ be the indicator of the event that the colours of $e$ receive all $k-2$ numbers in $[k-2]$ (otherwise, if $e$ does not contain $u$ and $v$, let $B(e,u,v)=0$).
We let $B(u,v)=\sum_{f \in E(G_0)} B(f,u,v)$, let $L(e,u,v)=B(e,u,v)-\mathbb{E}[B(e,u,v)]$ and let $L(u,v)=B(u,v)-\mathbb{E}[B(u,v)]$.
For each $u\in V$, we define $\vx_u\in \mathbb{R}^{V}$ by taking
$$
\vx_u(v)=
\begin{cases}
	-1, \textrm{ if } v=u\\
	\frac{L(u,v)}{\beta}, \textrm{ otherwise,}
\end{cases}
$$
where $\beta$ is a suitable normalization factor not dependent on $u$ (but dependent on $G_0$), chosen so that each $\vx_u$ has length $O(1)$ (the requirement that $\vx_u$ is a unit vector can be relaxed to this). Recall that we want to prove that $G$ has large surplus in expectation, which amounts to proving that the expression
$$\mathbb{E}\left[\sum_{e=uv\in E(G)} \langle \vx_u,\vx_v \rangle\right]=\sum_{e=uv\in E(G_0)} \mathbb{E}\left[\ind_{e\in E(G)}\langle \vx_u,\vx_v \rangle\right]$$
is very negative.

Now note that for some edge $e=uv\in E(G_0)$,

\begin{equation*} 
	\ind_{e\in E(G)} \langle \vx_u,\vx_v \rangle=B(e,u,v)\langle \vx_u,\vx_v \rangle=B(e,u,v)\left(-2\frac{L(u,v)}{\beta}+\sum_{w\in V\setminus \{u,v\}} \frac{L(u,w)L(v,w)}{\beta^2}\right).
\end{equation*}

We would like to prove that the first term, namely \begin{equation}
	-2\beta^{-1}B(e,u,v)L(u,v), \label{eq:outline gain}
\end{equation}
is negative in expectation (for each $e$), while the total sum corresponding to the second term, namely
\begin{equation}
	\beta^{-2}\sum_{e=uv\in E(G_0)} \sum_{w\in V\setminus \{u,v\}} \mathbb{E}[B(e,u,v)L(u,w)L(v,w)], \label{eq:outline error term}
\end{equation}
is not too positive. For the first term, notice that
\begin{align*}
	\mathbb{E}[B(e,u,v)L(u,v)]
	&=\mathbb{E}\left[B(e,u,v)\sum_{f\in E(G_0)} (B(f,u,v)-\mathbb{E}[B(f,u,v)])\right] \\
	&=\sum_{f\in E(G_0)} \mathbb{E}\left[B(e,u,v)(B(f,u,v)-\mathbb{E}[B(f,u,v)])\right].
\end{align*}
Note that in the sum we only consider edges $f$ with vertex set $\{u,v\}$, otherwise $B(f,u,v)=0$ with probability $1$. An important observation is that for any $e$ and $f$ with vertex set $\{u,v\}$, we have $\mathbb{E}[B(e,u,v)(B(f,u,v)-\mathbb{E}[B(f,u,v)])]\geq 0$. Indeed, conditioning on the event that the $k-2$ colours of $e$ receive all numbers in $[k-2]$ cannot make it less likely that the $k-2$ colours of $f$ receive all numbers in $[k-2]$. Moreover, when $f=e$, then $\mathbb{E}[B(e,u,v)(B(f,u,v)-\mathbb{E}[B(f,u,v)])]=\Var(B(e,u,v))$ is at least some positive constant depending on $k$. Thus, for each edge $e$, $\mathbb{E}[B(e,u,v)L(u,v)]$ is at least a positive constant. This implies that the terms~(\ref{eq:outline gain}) contribute very negatively to $\mathbb{E}[\sum_{e=uv\in E(G)} \langle \vx_u,\vx_v \rangle]$.

We now turn to upper bounding (\ref{eq:outline error term}), which we view as an error term. Note that
\begin{align*}
	&\sum_{e=uv\in E(G_0)} \sum_{w\in V\setminus \{u,v\}} \mathbb{E}[B(e,u,v)L(u,w)L(v,w)] \\
	&=\sum_{e=uv\in E(G_0)} \sum_{w\in V\setminus \{u,v\}} \sum_{f,f'\in E(G_0)} \mathbb{E}[B(e,u,v)L(f,u,w)L(f',v,w)].
\end{align*}
Observe that the term $\mathbb{E}[B(e,u,v)L(f,u,w)L(f',v,w)]$ vanishes provided that the colour sets of $e$, $f$ and $f'$ are pairwise disjoint, as in that case the random variables $B(e,u,v)$, $B(f,u,w)$, $B(f',v,w)$ are mutually independent. Unfortunately, if the colour sets of $e,f,f'$ overlap, this term may not vanish\footnote{In our setting, where $k\geq 4$, single overlaps in fact do not cause problems, but if say $f$ and $f'$ have at least two colours in common, that already makes the term not vanish.}. Note that if (say) $f$ and $f'$ share at least one colour, they correspond to hyperedges in $H$ that intersect in at least two vertices: one (namely $w$) in $V$, and another one in $C$. Hence, if we knew that our original hypergraph $H$ has no pairs of vertices with large codegree, we could get good bounds on the error term (\ref{eq:outline error term}). For example, with this approach we could recover the nearly tight lower bound $m^{3/4-o(1)}$ for linear hypergraphs obtained by R\"aty and Tomon \cite{raty2024large} (and in fact we could remove the $o(1)$ term), as well as their lower bound $m^{3/5-o(1)}$ in the general (i.e. not necessarily linear) case. Using the observation in the footnote, with this approach and some more ideas, we can in fact improve the $3/5$ exponent for $k$-cut in $k$-graphs for all $k\geq 4$.

\subsection{First main idea: a modular definition of the SDP vectors} \label{sec:outline modular}

Unfortunately, large codegrees of triples always cause significant issues with the above approach, and the bottleneck of the method is an exponent bounded away from the desired $2/3$. Our first main idea is to modify the definition of the vectors $\vx_u$ as follows. For an edge $e$ with vertex set $\{u,v\}$, we let $B^*(e,u,v)$ be the indicator of the event that the sum of the numbers that the colours of $e$ receive is $\sum_{i=1}^{k-2} i$ modulo $k-2$. Note that $B^*(e,u,v)$ positively correlates with $B(e,u,v)$: indeed, $B(e,u,v)=1$ implies $B^*(e,u,v)=1$. On the other hand, the major advantage is that the distribution of $B^*(e,u,v)$, conditioned on some partial assignment of the colours of $e$, is the same as the unconditioned distribution, unless we have exposed \emph{all} colours of $e$. (This is not the case for $B(e,u,v)$, where assigning numbers to two colours already changes the distribution).

We now define $\vx_u$ as before, except we use $B^*(f,u,v)$ instead of $B(f,u,v)$. More formally, let $B^*(u,v)=\sum_{f \in E(G_0)} B^*(f,u,v)$, let $L^*(e,u,v)=B^*(e,u,v)-\mathbb{E}[B^*(e,u,v)]$ and let $L^*(u,v)=B^*(u,v)-\mathbb{E}[B^*(u,v)]$.
For each $u\in V$, we define $\vx_u\in \mathbb{R}^V$ by taking
$$
\vx_u(v)=
\begin{cases}
	-1, \textrm{ if } v=u\\
	\frac{L^*(u,v)}{\beta}, \textrm{ otherwise.}
\end{cases}
$$
Thanks to the positive correlation between $B(e,u,v)$ and $B^*(e,u,v)$, terms (\ref{eq:outline gain}) provide a similarly negative contribution as before. But the key difference is that now we have a much better bound for terms (\ref{eq:outline error term}), as $\mathbb{E}[B(e,u,v)L^*(f,u,w)L^*(f',v,w)]$ vanishes unless the colour set of $f'$ is contained in the union of the colour set of $e$ and the colour set of $f$. In particular, it is not hard to prove that the main error term will now come from triples $(e,f,f')$ where $f$ and $f'$ have \emph{identical} colour sets -- a much stricter condition than intersecting in at least two elements when $k$ is large. With tools we develop in this paper, we are able to upper bound rather efficiently the contribution of such terms.

We conclude this subsection by remarking that in order to execute the above strategy, it is necessary to assume some mild control on the maximum degree and the maximum codegree of the hypergraph whose surplus we are bounding. Indeed, this bound on the maximum degree is needed to choose the normalization factor $\beta$ efficiently when defining the vectors $\vx_u(v)$, and the codegree bound is essential to control the various error terms arising. It had already been demonstrated in the work of Conlon et al. \cite{conlon2019hypergraph} that the general case can be essentially reduced to the case where we have good bounds on the maximum degree and maximum codegree of the hypergraph. We will collect the necessary reduction lemmas (see, e.g., Lemma \ref{lem:Dichotomy-Set-U}) in Section \ref{sec:prelimiaries}.

\subsection{Second main idea: an SDP bound for $2$-cuts in hypergraphs} \label{sec:outline hypergraph SDP}

So far, we have only dealt with the case $r=k$ (and with $r=k-1$, as it has a simple reduction to the $r=k$ case). The case $r<k-1$ is fundamentally different.\footnote{Recall that the method of R\"aty and Tomon \cite{raty2024large} does not work in this case.} While it is still true that for every $r$, lower bounding the $r$-surplus in a $k$-graph $H$ can be reduced to lower bounding the expected $2$-surplus in a certain auxiliary random coloured multi\emph{hypergraph} $G$, in general this $G$ is not a multi\emph{graph} (and is not even uniform). Prior to our work, there was no widely applicable result in the literature that allowed one to obtain large $2$-cuts in \emph{hypergraphs} from vectors assigned to each vertex. A result along these lines was proved by R\"aty and Tomon in their paper on bisection width and positive discrepancy (see \cite{raty2024bisection}, Lemma 3.1), but we need something significantly stronger. Our result, which extends the usual graph SDP bound to hypergraphs, is as follows. (Here and below, a mixed $k$-graph is a hypergraph whose edges are of size at most $k$. The surplus of a $2$-cut in such a hypergraph is the difference between the size of the cut and the expected size of a random $2$-cut.)
\begin{restatable}{theorem}{SDPformula}[SDP bound for $2$-cuts in hypergraphs]
\label{thm:outline hypergraph SDP}
Fix $k \in \mathbb{N}$.
Let $G=(V,E)$ be a mixed $k$-multigraph. 
Let $N \in \mathbb{N}$, and for every $u\in V$, let $x_u$ be a vector in $\mathbb{R}^N$ such that $||x_u|| \in [1,2]$.
Then $G$ has a $2$-cut of surplus at least
\begin{align*}
    - \sum_{e \in E} C_{|e|} \sum_{u\neq v \in e} \langle x_u,x_v\rangle - O_k(\sum_{e \in E} \sum_{u\neq v \in e} \langle x_u,x_v\rangle^2),
\end{align*}
where $C_{|e|}$ are positive constants, depending only on the size of $e$.
\end{restatable}

Note that in this result we do not have a $\log n$ loss like in the general SDP rounding result of \cite{megretski2001relaxations,charikar2004maximizing,alon2006quadratic} (and, of course, that result does not even apply in the hypergraph setting). This helps us to get a tight bound for linear hypergraphs.

\subsection{Organization of this paper}

In the next section we collect some preliminary results that we will use in our proofs. In Section \ref{sec:SDP for hypergraphs}, we prove our SDP bound for hypergraphs (Theorem~\ref{thm:outline hypergraph SDP}). In the following two sections we prove our results about linear hypergraphs: in Section \ref{sec:linear lower}, we prove the lower bound, Theorem \ref{thm:linear}, while in Section \ref{sec:linear upper}, we prove the matching upper bound, Proposition~\ref{thm:Upper-Bound-Linear-Hypergraphs}. Our main result about general hypergraphs, Theorem \ref{thm:Hypergraph-Excess}, is proved in Section~\ref{sec:general}. In Section \ref{sec:(4,2)-construction}, we prove Proposition \ref{prop:Alternative-Construction-4-2}. Finally, in Section~\ref{sec:concluding remarks}, we present several concluding remarks and mention further applications and upcoming work.

\section{Preliminaries} \label{sec:prelimiaries}

\textbf{Notation.}
For two functions $f$ and $g$, we write $f=O(g)$ to say that there is a constant $C$ such that $|f|\le Cg$, and we write $f=\Omega(g)$ to say that there is a constant $c>0$ such that $f\ge cg$. We write $O_t(g)$ and $\Omega_t(g)$ if the implied constants may depend on some parameter $t$. We sometimes omit the subscript if the dependence is on the uniformity $k$, which is always a constant throughout the paper.
A \emph{$k$-graph} is a hypergraph whose hyperedges have size $k$.
A \emph{mixed $k$-graph} is a hypergraph whose hyperedges have size at most $k$.
A \emph{mixed $k$-multigraph} is a multihypergraph whose hyperedges have size at most $k$.
For a mixed $k$-multigraph $H$, the \emph{surplus} of an $r$-cut is the difference between the size of the cut and the expected size of a random $r$-cut. The \emph{$r$-surplus} of $H$, denoted by $\surp_r(H)$, is the maximum surplus of an $r$-cut of $H$. We also write $\surp(H)$ for the $2$-surplus of $H$. For a hypergraph $H$ and distinct vertices $u$ and $v$, we write $\mathrm{deg}_H(u,v)$ to denote the number of edges containing both $u$ and $v$.  \\

We first recall the Azuma-Hoeffding inequality (see, for instance, Theorem 2.25 in \cite{JLR}).
 
\begin{lemma}[Azuma-Hoeffding inequality]
\label{lem:azuma}
Let $(X_i)_{i \geq 0}$ be a martingale and let $c_i>0$ for each $i\geq 1$. If $|X_i -X_{i-1}| \leq c_i$ for each $i\geq 1$, then, for each $n\geq 1$,
\[
\mathrm{Pr}(|X_n - X_0 | \geq t ) \leq  2\exp \left( -\frac{t^2}{2\sum_{i=1}^n c_i^2} \right).
\]
\end{lemma}

We will need some basic properties of cuts. We start with the two following simple lemmas, which can be found as Lemma 2.5 and Lemma 2.6 in \cite{raty2024large} for instance.
For a multihypergraph~$H$, the \emph{underlying $\ell$-multigraph} of $H$ is the multihypergraph on the same vertex set as $H$ obtained by connecting every set of $\ell$ vertices by an edge as many times as it appears in an edge of $H$.

\begin{lemma}
\label{lem:Surplus-Ineq-Induced-Subgraph}
Let \( G \) be a multigraph, and let \( V_1, \ldots, V_k \subseteq V(G) \) be disjoint sets. Then
\[
\surp(G) \geq \sum_{i=1}^{k} \surp(G[V_i]).
\]
\end{lemma}

\begin{lemma}
\label{lem:Surplus-Relation-UnderlyingGraph}
    Let \( H \) be a \( k \)-multigraph, and let \( H' \) be the underlying \((k - 1)\)-multigraph. Then
\[
2k \cdot \surp_k(H) \geq \surp_{k-1}(H) \geq \frac{1}{2} \cdot \surp_{k-1}(H').
\]
\end{lemma} 

We will also make use of the following lemma, which can be found in \cite{conlon2019hypergraph} as Lemma 5.1.

\begin{lemma}
\label{lem:Reduce-D-Delta}
    For any fixed $2 \leq r \leq k$, consider an $n$-vertex, $m$-edge $k$-multigraph $H$. For any $q,D,\Delta$, at least one of the following holds:
    \begin{itemize}
        \item there is an $r$-cut with surplus $\Omega(Dq)$, or
        \item there is a vertex subset $U$ with $|U| \geq n-2q-km/\Delta$ such that for every distinct $u,v \in U$, $\deg_H(v) \leq \Delta$ and $\deg_H(u,v) \leq D$.
    \end{itemize}
\end{lemma}

Moreover, we need the following slight strengthening of Lemma 5.6 from \cite{conlon2019hypergraph} which can be obtained via the same proof. 

\begin{lemma}
\label{lem:Conlon-U-large-number-edges}
    Fix $2 \le r \le k$ and consider an $n$-vertex, $m$-edge $k$-multigraph $H$.  
Suppose $U$ is a vertex subset with $|U| = n - b$.  
Then, for every $\eps>0$, at least one of the following holds:  
\begin{itemize}
    \item there is an $r$-cut with surplus $\Omega\!\left(\eps m/b\right)$, or
    \item there are at least $(1-\eps)m$ edges which have at least $k - 1$ of their vertices in $U$.
\end{itemize}
\end{lemma}

We will also need the following result.

\begin{lemma}
\label{lem:U-large-number-edges}
    Let $H$ be an $n$-vertex, $m$-edge $k$-multigraph, and let $2 \leq r \leq k$.
    There exists some $\eps=\eps(k) >0$ such that the following holds.
    Let $U$ be a vertex subset, and suppose that at least $(1-\eps)m$ edges have at least $k - 1$ of their vertices in $U$.
    Then at least one of the following holds:
    \begin{itemize}
    \item $H$ has $r$-surplus $\Omega(m)$,
    \item $H[U]$ has $\Omega(m)$ edges.
    \end{itemize}
\end{lemma}

\begin{proof}
    Suppose that $H[U]$ has less than $\eps m$ edges.
    Thus there are at least $(1-2\eps)m$ edges which have exactly $k-1$ of their vertices in $U$.
    Let $\phi$ be a random $r$-cut obtained by assigning each vertex in $U$ to each part $\{1, \dots, r-1\}$ uniformly at random and independently, and assigning each vertex in $U^c$ to part $r$.
    For integers $2\leq y\leq x$, we let $p(x,y)$ be the probability that an edge of size $x$ is cut by an uniformly random $y$-cut.
    We remark that $p(x-1,y-1) > p(x,y)$; indeed once the part (in the cut) of the first vertex of an edge of size $x$ has been revealed, the edge is cut if and only if the parts of the other $x-1$ vertices include the $y-1$ parts left to be hit, an event which happens with probability strictly less than $p(x-1,y-1)$.
    Then the expected size of $\phi$ is at least $p(k-1,r-1)(1-2\eps)m \geq p(k,r)m + \Omega(m)$ for small enough $\eps$, which concludes the proof.
\end{proof}

Combining the previous lemmas we obtain the following result.

\begin{lemma}
\label{lem:Dichotomy-Set-U}
Let $G$ be a $n$-vertex, $m$-edge $k$-multigraph, and let $2 \leq r \leq k$.
Let $\eps > 0$.
Then either $G$ has an $r$-cut of surplus $\Omega(m^{2/3-\eps})$, or there exists $U \subseteq V(G)$ such that
\begin{itemize}
    \item $G[U]$ has $\Omega(m)$ edges,
    \item for every distinct $u,v \in U$, $\deg_G(v) \leq m^{2/3-\eps}$ and $\deg_G(u,v) \leq m^{1/3-2\eps}$. 
\end{itemize}
\end{lemma}

\begin{proof}
    By applying \Cref{lem:Reduce-D-Delta} for $\Delta = m^{2/3-\eps}$, $q= m^{1/3+\eps}$ and $D=m^{1/3-2\eps}$, if $G$ has no $r$-cut of surplus $\Omega(m^{2/3-\eps})$, then there exists some vertex subset $U$ with $|U| \geq n - O(m^{1/3+\eps})$ such that for all distinct $u,v \in U$, $\deg_G(v) \leq \Delta$ and $\deg_G(u,v) \leq D$.
    Fix some $\delta> 0$ depending only on~$k$. By \Cref{lem:Conlon-U-large-number-edges}, we may assume that there are at least $(1-\delta)m$ edges which have at least $k - 1$ of their vertices in $U$ (otherwise we have an $r$-cut of desired size).
    Finally by choosing $\delta$ sufficiently small and applying \Cref{lem:U-large-number-edges}, we obtain either an $r$-cut of desired size, or that $G[U]$ has $\Omega(m)$ edges, as wanted.
\end{proof}

The next lemma shows that, to obtain an $r$-cut with large surplus, it is sufficient to select $r-q$ parts of the cut uniformly at random, and then find a $q$-cut with large surplus in a naturally defined auxiliary hypergraph.
The proof is straightforward, and therefore deferred to \Cref{Appendix-lem:Reduction-r-cut-p-cut}. 

\begin{lemma}
\label{lem:Reduction-r-cut-p-cut}
Let $q,r\geq 2$ be integers, such that $q \leq r-1$, and let $H$ be a hypergraph.
Let $\sigma$ assign independently to each vertex $v \in V(H)$ either one of the numbers in $\{q+1, \dots, r\}$, each with probability $1/r$, or $*$, with probability $q/r$.
Let $H_{\sigma}$ be the multihypergraph on vertex set $\sigma^{-1}(*)$ and edge set $\{ e \cap \sigma^{-1}(*) : e \in E(H), \forall j \in [q+1,r], |e \cap \sigma^{-1}(j)| \geq 1 \}$.
Then $H$ has an $r$-cut of surplus at least $\mathbb{E}_{\sigma}[\surp_{q}(H_{\sigma})]$.
\end{lemma}

We will also need the following slightly more involved reduction lemma, whose proof is essentially the same as Lemma 7.1 from \cite{conlon2019hypergraph}, and therefore deferred to \Cref{Appendix-lem:Reduction-Key-r-large}.

\begin{lemma}
\label{lem:Reduction-Key-r-large}
    Let $H=(V,E)$ be a multihypergraph. Let $V=U \cup U^c$.
    Let $\omega:V\rightarrow [r]$ be chosen uniformly at random and let $P\omega$ be the function from $V$ to $\{ *,1, \dots, r\}$ such that
    \begin{align*}
        P\omega(v)= \begin{cases}
            \omega(v) & \text{if } v \in U^c \text{ or }  \omega(v) \in \{3, \dots, r\}, \\
           * & \text{otherwise}.
        \end{cases}
    \end{align*}
    Let $H_{P\omega}$ be the multihypergraph on vertex set $(P\omega)^{-1}(*)$ and edge set $\{ e \cap (P\omega)^{-1}(*): e \in E, \{3, \dots, r\} \subseteq P\omega(e) , \{1, \dots, r\} \not \subset  P\omega(e) \} $, where two copies of $e\cap (P\omega )^{-1}(*)$ are added to $H_{P\omega}$ if $P\omega(e) \cap \{1, \dots, r\} =  \{3, \dots, r\}$. 
    Then $H$ has $r$-surplus at least $\mathbb{E}_{\omega}[\surp(H_{P\omega})]/2$. 
\end{lemma}

We rewrite here the previous lemma in the special case $r=2$ as we will use this version in a couple of places.
More precisely it shows that, when looking for a $2$-cut with large surplus, one may assign to some vertices one part of the cut uniformly at random, and then find a $2$-cut with large surplus in a naturally defined auxiliary hypergraph.

\begin{Corollary}
\label{lem:Reduction-Reveal-Outside-W}
    Let $H=(V,E)$ be a multihypergraph, and let $W \subseteq V$.
    Let $\sigma : W^c \rightarrow \{ 1,2 \}$ be taken uniformly at random.
    Let $H_{\sigma}$ be the multihypergraph on vertex set $W$ such that for every edge $e \in E(H)$, if all vertices of $e$ are in $W$ then we add two copies of $e$, otherwise, if $|\sigma(e \setminus W)| = 1$ and $|e \cap W| > 0$, we add one copy of $e \cap W$.
    Then
    \begin{align*}
        \surp(H) \geq \mathbb{E}_{\sigma}[\surp(H_{\sigma})]/2.
    \end{align*}
\end{Corollary}

The next lemma reduces the problem of finding large $r$-cuts in $k$-graphs to finding large $2$-cuts in hypergraphs.
The proof follows closely the argument in the proof of Theorem 1.4 in \cite{conlon2019hypergraph}.
Therefore, we defer this proof to \Cref{Appendix-Reduction-2-Cut}. 

\begin{lemma}
\label{lem:Reduction-Main-Small-r-Lem}
    Let $\eps > 0$ be fixed.
    Let $q \geq 2$, $k \geq 4$ be some fixed integers.
    Suppose that every mixed $k$-multigraph with maximum degree $O(m^{2/3-\eps})$, maximum codegree $O(m^{1/3-2\eps})$, $O(m)$ hyperedges in total, and $\Omega(m)$ hyperedges of size $q+2$ has a $2$-cut of surplus $\Omega(m^{2/3-\eps})$.
    
    Then for every $2 \leq r \leq k-q$, every $m$-edge $k$-multigraph has an $r$-cut of surplus $\Omega(m^{2/3-\eps})$.
\end{lemma}

A similar (but simpler) reduction holds for linear hypergraphs.
This can be proven the same way as the previous lemma, and therefore its proof is left to the reader. 

\begin{lemma}
\label{lem:Full-Reduction-2-cut-Linear}
    Fix $k \geq 4$ and $\alpha \in (0,1)$.
    Suppose that every linear mixed $k$-graph with $O(m)$ hyperedges in total and $\Omega(m)$ hyperedges of size at least $4$ has a $2$-cut of surplus $\Omega(m^{\alpha})$.
    
    Then for every $2 \leq r \leq k-2$, every linear $k$-graph with $O(m)$ hyperedges in total and $\Omega(m)$ hyperedges of size at least $r+2$ has an $r$-cut of surplus $\Omega(m^{\alpha})$.
\end{lemma}

We recall the following standard integral estimates.

\begin{lemma}
\label{lem:Standard-Gaussian-Estimates}
For any $a >0$, we have
\begin{align*}
\int_{-\infty}^0 e^{-a z^2} \, dz = \frac{1}{2} \sqrt{\frac{\pi}{a}}, \text{  and  }
\int_{-\infty}^0 z\, e^{-a z^2} \, dz = -\frac{1}{2a}.
\end{align*}
\end{lemma}

From the previous estimates, the following lemma is immediate. In what follows, we write $Z$ for the vector $(z_1,\dots,z_k)$.

\begin{lemma}
\label{lem:First-Moment-Derivative-Estimate}
    Fix a positive integer $k$.
    Let $M \in [0,1/2)$ and let $A$ be a $k$ by $k$ diagonal matrix, such that all diagonal entries have values between $1-M$ and $1+M$.
    Then, for any $1\leq s<\ell\leq k$,
    \begin{align*}
        \frac{1}{(2\pi)^{k/2}} \int_{z_1 = -\infty}^{0} \dots \int_{z_{k} = -\infty}^{0} -\frac{1}{2}z_s z_{\ell}\exp(-\frac{1}{2}Z^TAZ)  dZ = -C + O_k(M),
    \end{align*}
    where $C= \frac{1}{\pi 2^{k}}$.
\end{lemma}

\begin{proof}
    Note that
    \begin{align*}
        \frac{1}{(2\pi)^{k/2}}\int_{z_1 = -\infty}^{0} \dots &\int_{z_{k} = -\infty}^{0} -\frac{1}{2}z_s z_{\ell}\exp(-\frac{1}{2}Z^TAZ)  dZ \\
        &= -\frac{1}{2} \frac{1}{(2\pi)^{k/2}} \prod_{i=1}^k \int_{z_i = -\infty}^{0} f_i(z_i) \exp(-\frac{1}{2} A_{i,i}z_i^2) dz_i,
    \end{align*}
    where $f_i(x)=x$ if $i=s$ or $i=\ell$, and $f_i(x)=1$ otherwise.
    Applying \Cref{lem:Standard-Gaussian-Estimates}, we obtain
    \begin{align*}
        \frac{1}{(2\pi)^{k/2}}\int_{z_1 = -\infty}^{0} \dots \int_{z_{k} = -\infty}^{0} -\frac{1}{2}z_s z_{\ell}\exp(-\frac{1}{2}Z^TAZ)  dZ &= - \frac{1}{2} \frac{1}{\sqrt{2\pi}A_{s,s}} \cdot  \frac{1}{\sqrt{2\pi}A_{\ell,\ell}} \cdot \prod_{i \neq s,\ell} \frac{1}{2} \sqrt{\frac{1}{A_{i,i}}}.
    \end{align*}
    Since $1-M\leq A_{i,i}\leq 1+M$ for every $i$, we obtain
    \begin{align*}
        \frac{1}{(2\pi)^{k/2}} \int_{z_1 = -\infty}^{0} \dots \int_{z_{k} = -\infty}^{0} -\frac{1}{2}z_s z_{\ell}\exp(-\frac{1}{2}Z^TAZ) dZ = -C+ O_k(M),
    \end{align*}
    for $C= \frac{1}{\pi 2^{k}}$, as wanted.
\end{proof}

\section{An SDP bound for $2$-cut in hypergraphs} \label{sec:SDP for hypergraphs}

The goal of this section is to establish \Cref{thm:outline hypergraph SDP}, which, we recall, is a hypergraph generalization of the standard SDP bound for MaxCut in graphs and will be used repeatedly throughout the proofs. For the reader’s convenience, we restate it below.

\SDPformula*

We will establish this result via the following geometric lemma.

\begin{lemma}
\label{lem:Strong-Geometric}
    Fix some positive integer $k$. 
    Let $N \in \mathbb{N}$, and let $a_1, \dots, a_k$ be points on $\mathbb{S}^{N-1}$.
    Suppose that for some real $M \in [0,1]$ we have that $|\langle a_i,a_j\rangle| \leq M$ for every distinct $i,j \in [k]$.
    Take a random point on $\mathbb{S}^{N-1}$, and let $\overrightarrow{r}$ be the associated random vector.
    Then the probability that all inner products $\langle a_i, \overrightarrow{r}\rangle$ are negative is $(1/2)^{k}+C \sum_{i \neq j} \langle a_i,a_j\rangle + O_k(M^2)$, where $C= \frac{1}{\pi 2^{k}}$.
\end{lemma}

We remark that a similar (but significantly weaker) result was obtained in \cite{raty2024bisection} (see their Lemma 3.1). Our result implies theirs but is both more general and more accurate: we do not have an assumption saying that all $\langle a_i,a_j \rangle$ have the same sign (which would not be satisfied in our setting), and our error term is smaller.

We now turn to the proof of Lemma \ref{lem:Strong-Geometric}. We first rewrite the desired probability as an integral. 

\begin{lemma}
\label{lem:Translate-Into-Integral}
    Let $k, N$ be positive integers, let $a_1, \dots, a_k$ be points on $\mathbb{S}^{N-1}$ and let $0\leq M\leq 1$.
    Suppose that $|\langle a_i,a_j\rangle| \leq M$ for every distinct $i,j \in [k]$.
    Take a random point on $\mathbb{S}^{N-1}$, and let $\overrightarrow{r}$ be the associated random vector.
    Then the probability that all inner products $\langle a_i, \overrightarrow{r}\rangle$ are negative can be written as
    \begin{align*}
        \int_{z_1 = -\infty}^{0} \dots \int_{z_{k} = -\infty}^{0} \frac{1}{(2\pi)^{k/2}}\exp(-\frac{1}{2}Z^TAZ)  dZ + O_k(M^2),
    \end{align*}
    for some symmetric matrix $A$ such that
    \begin{itemize}
        \item for every $i$, we have $A_{i,i}=1+O_k(M^2)$,
        \item for every $i \neq j$, we have $A_{i,j}= -\langle a_i,a_j \rangle +O_k(M^2)$.
    \end{itemize}
\end{lemma}

\begin{proof}
    First, observe that we may assume that $M\leq c(k)$ for a sufficiently small positive constant $c(k)$, otherwise the statement is trivial by choosing $A=I$ and sufficiently large implicit constants in the $O_k(\cdot)$ terms.
    It is well known (see for instance \cite{muller1959note}) that by sampling uniformly random independent standard Gaussians $W_1, \dots, W_N$, the random vector $\overrightarrow{r}$ has the same distribution as
    \begin{align*}
         \left( \frac{W_1}{\sqrt{\sum_{i=1}^N W_i^2}}, \dots, \frac{W_N}{\sqrt{\sum_{i=1}^N W_i^2}} \right).
    \end{align*}
    Let $\overrightarrow{s}=(W_1,\dots,W_N)$.
    It is clear that $(\langle a_1, \vec{s} \rangle, \dots, \langle a_k, \vec{s} \rangle)$ is a Gaussian vector of covariance matrix $G$ and mean vector $0$, where $G$ is the Gram matrix of $a_1, \dots, a_k$.
    Therefore,
    \begin{align*}
         \mathbb{P}(\langle a_1, \vec{r} \rangle < 0, \dots, \langle a_k, \vec{r} \rangle < 0)
         &=\mathbb{P}(\langle a_1, \vec{s} \rangle < 0, \dots, \langle a_k, \vec{s} \rangle < 0) \\
&= \int_{z_1 < 0} \cdots \int_{z_k < 0} 
\frac{1}{(2\pi)^{k/2} \det(G)^{1/2}} 
\exp\left( -\frac{1}{2} Z^T G^{-1} Z \right) dZ.
    \end{align*}
    We let $G=I+E$, where $I$ is the identity matrix. Note that $E_{i,i}=0$ and $E_{i,j}=\langle a_i,a_j \rangle$ for every $i\neq j$, so in particular all entries of $E$ are bounded by $M$ in absolute value.
    By the standard definition of the determinant, it is easy to see that $\det(G) = 1+O_k(M^2)$ since each product in the expression for $\det(G)$, except the one which is the product of all diagonal entries, has at least two non-diagonal terms.
    
    Moreover, we have by Neumann series expansion (which converges since $0\leq M\leq c(k)$ for a small constant $c(k)$ by our assumption at the beginning) that $G^{-1}=I-E+K$, where $K= \sum_{\ell \geq 2} (-E)^\ell$, and therefore each entry of $K$ is at most $O_k(M^2)$ in absolute value.
    Thus $A=G^{-1}$ satisfies the required conditions.
\end{proof}

The next lemma bounds the integral obtained in Lemma \ref{lem:Translate-Into-Integral}.

\begin{lemma}
\label{lem:Derivating-Integral}
    Fix some positive integer $k$.
    Let $0\leq M\leq \frac{1}{4k}$, let $A$ be a symmetric $k$ by $k$ matrix, with non-diagonal entries at most $M$ in absolute value, and diagonal entries between $1-M^2$ and $1+M^2$.
    
    Let $\Phi(A) = \int_{z_1 = -\infty}^{0} \dots \int_{z_{k} = -\infty}^{0} \frac{1}{(2\pi)^{k/2}} \exp(-\frac{1}{2}Z^TAZ)  dZ$.
    Then, for $C= \frac{1}{\pi 2^{k}}$, we have 
    \begin{align*}
        \Phi(A) &= (1/2)^{k} - C(\sum_{i \neq j} A_{i,j}) + O_k(M^2).
    \end{align*}
\end{lemma}

\begin{proof}
    We first make the following claim.
    \begin{claim}
    \label{cl:part-derivative}
    For any two coordinates $s \neq \ell$, we have 
    \begin{align*}
        \frac{\partial \Phi(A)}{\partial A_{s,\ell}} = -C + O_k(M),
    \end{align*}
    where $C= \frac{1}{\pi 2^{k}}$.
    \end{claim}
    \begin{proof}[Proof of \Cref{cl:part-derivative}]
    First, we note that by applying AM-GM inequality, we have for any $Z \in \mathbb{R}^n$ that
    \begin{align*}
        \exp(-\frac{1}{2}Z^TAZ) = \exp(-\sum_{i, j} \frac{1}{2}A_{i,j}z_iz_j) &\leq \exp\left(-\sum_{i} \frac{1}{2}A_{i,i} z_i^2+\sum_{i \neq j} \frac{1}{2}|A_{i,j}|\left(\frac{z_i^2+z_j^2}{2}\right)\right) \\
        &= \exp\left(-\sum_i \left(\frac{1}{2}A_{i,i}-\sum_{j \neq i} \frac{1}{2}|A_{i,j}|\right)z_i^2 \right) \\
        &= \exp(-\frac{1}{2}Z^TA^+_{diag}Z),
    \end{align*}
    where the matrix $A^+_{diag}$ is diagonal, with diagonal entries between $1-kM-M^2$ and $1+kM+M^2$.
    Similarly, there exists a diagonal matrix $A^-_{diag}$, with diagonal entries between $1-kM-M^2$ and $1+kM+M^2$, such that for any $Z \in \mathbb{R}^n$, we have
    \begin{align*}
        \exp(-\frac{1}{2}Z^TAZ) \geq \exp(-\frac{1}{2}Z^TA^-_{diag}Z).
    \end{align*}
    Next we have by definition of $\Phi$ that
    \begin{align*}
        \frac{\partial \Phi(A)}{\partial A_{s,\ell}} = \frac{1}{(2\pi)^{k/2}} \cdot \frac{\partial \int_{z_1 = -\infty}^{0} \dots \int_{z_{k} = -\infty}^{0} \exp(-\frac{1}{2}Z^TAZ)  dZ}{\partial A_{s,\ell}}.
    \end{align*}
    Note that 
    \begin{align*}
        \left|\frac{\partial \exp(-\frac{1}{2}Z^TAZ) }{\partial A_{s,\ell}}\right|= \left|\frac{1}{2}z_s z_{\ell}\exp(-\frac{1}{2}Z^TAZ) \right| \leq \left|z_s z_{\ell}\exp(-\frac{1}{2}Z^TA^+_{diag}Z) \right|.
    \end{align*}
    As the function $\left|z_s z_{\ell}\exp(-\frac{1}{2}Z^TA^+_{diag}Z) \right|$ is integrable, we may invert the integration and derivation signs (see for instance Theorem 2.27 in \cite{folland1999real}) and obtain that 
    \begin{align*}
        \frac{\partial \Phi(A)}{\partial A_{s,\ell}} &= \frac{1}{(2\pi)^{k/2}} \int_{z_1 = -\infty}^{0} \dots \int_{z_{k} = -\infty}^{0} \frac{ \partial\exp(-\frac{1}{2}Z^TAZ)}{\partial A_{s,\ell}} dZ\\
        &= \frac{1}{(2\pi)^{k/2}}\int_{z_1 = -\infty}^{0} \dots \int_{z_{k} = -\infty}^{0} -\frac{1}{2}z_s z_{\ell}\exp(-\frac{1}{2}Z^TAZ)  dZ.
    \end{align*}
    As $\exp(-\frac{1}{2}Z^TA^-_{diag}Z) \leq \exp(-\frac{1}{2}Z^TAZ)  \leq \exp(-\frac{1}{2}Z^TA^+_{diag}Z)$, it follows that 
    \begin{align*}
        \frac{1}{(2\pi)^{k/2}}\int_{z_1 = -\infty}^{0} \dots \int_{z_{k} = -\infty}^{0} &-\frac{1}{2}z_s z_{\ell}\exp(-\frac{1}{2}Z^TA^-_{diag}Z)  dZ \leq  \frac{\partial \Phi(A)}{\partial A_{s,\ell}} \\
        &\leq  \frac{1}{(2\pi)^{k/2}}\int_{z_1 = -\infty}^{0} \dots \int_{z_{k} = -\infty}^{0} -\frac{1}{2}z_s z_{\ell}\exp(-\frac{1}{2}Z^TA^+_{diag}Z)  dZ.
    \end{align*}
    Applying \Cref{lem:First-Moment-Derivative-Estimate}, we obtain
    \begin{align*}
         \frac{\partial \Phi(A)}{\partial A_{s,\ell}} = - C + O_k(M),
    \end{align*}
    where $C= \frac{1}{\pi 2^{k}}$, as wanted.
    \end{proof}
    From \Cref{cl:part-derivative}, the result then follows by a classical multivariate Taylor expansion over the variables $A_{s,\ell}$ for all $s\neq \ell$, as it is easy to see that for any $s\neq \ell$ and $s'\neq \ell'$, $\frac{\partial^{2} \Phi(A)}{\partial A_{s,\ell}\partial A_{s',\ell'}}$ is uniformly bounded (by the same computation as the first partial derivative in \Cref{cl:part-derivative}), and $\Phi(A_{diag})=(1/2)^{k}+ O_k(M^2)$ by \Cref{lem:Standard-Gaussian-Estimates}, where $A_{diag}$ is the diagonal matrix obtained from $A$ by removing the non-diagonal entries.
\end{proof}

\Cref{lem:Strong-Geometric} then follows as a simple combination of \Cref{lem:Translate-Into-Integral} and \Cref{lem:Derivating-Integral} (noting, once again, that the statement is trivial unless $M\leq c(k)$ for a small positive constant $c(k)$). 
With \Cref{lem:Strong-Geometric} in hand, it is easy to derive \Cref{thm:outline hypergraph SDP}.

\begin{proof}[Proof of \Cref{thm:outline hypergraph SDP}.]
    First, without loss of generality, we can assume, by adding dummy coordinates without changing the cross-inner products, that all vectors have norm $2$. Further, we can assume, by rescaling the vectors, that all $x_u$ have same norm $1$. 
    We take a random point $r$ on the sphere $\mathbb{S}^{N-1}$, and we let $ \overrightarrow{r}$ be the associated vector. 
    Let $H^{+}$ be the set of points in $\mathbb{R}^{N}$ whose inner product with $ \overrightarrow{r}$ is non-negative.
    Let $\phi$ be the $2$-cut associated to $\overrightarrow{r}$ by putting every vertex $u$ in class $1$ if $x_u$ is in $H^+$, and in class $2$ otherwise.
    Let $S_{\phi}$ be the size of $\phi$.
    We now compute $\mathbb{E}[S_{\phi}]$.
    Clearly, we have that $S_{\phi}= \sum_{e \in E} 1_{e \text{ is cut}}$, and therefore $\mathbb{E}[S_{\phi}] = \sum_{e \in E} \mathbb{P}(e \text{ is cut})$. \\
    Let $e \in E$. By \Cref{lem:Strong-Geometric}, we have
    \begin{align*}
         \mathbb{P}(e \text{ is cut}) \geq 1- \frac{1}{2^{|e|-1}} - C_{|e|} \sum_{u\neq v \in e} \langle x_u,x_v\rangle - O_k(\sum_{u\neq v \in e} \langle x_u,x_v\rangle^2).
    \end{align*}
    Summing over all hyperedges $e \in E$ gives the desired result.
\end{proof}

\section{Linear hypergraphs: tight lower bound} \label{sec:linear lower}

In this section we prove Theorem \ref{thm:linear}.
Before presenting the result from which we will deduce \Cref{thm:linear}, we introduce an important definition.

\begin{definition}[Random $(m,\Delta,k)$-multihypergraph of linear type] \label{defn:random linear}
    Fix some positive integer $k$ and a vertex set $V$.
We say that a random multihypergraph $H=(V,R)$ is a \emph{random $(m,\Delta,k)$-multihypergraph of linear type} if it satisfies the following properties: 
\begin{enumerate}[label=(\roman*)]
    \item There exists a multiset $X$ consisting of subsets of $V$ of size at most $k$ such that $|X|= O(m)$, and $R$ is a random subset of $X$ defined as follows: there exist a set $\mathcal{S}$ and some independent random variables $(Z_i)_{i \in \mathcal{S}}$ such that for every $e \in X$, there exists a set $S_e \subseteq \mathcal{S}$ of size at most $k$ such that the event $e \in R$ is a function of $\{ Z_i : i \in S_e \}$.
    
    \noindent (We will call the elements of $X$ \emph{potential hyperedges} in the hypergraph $H$, as whether they belong to the edge set $R$ of $H$ depends on the random variables $(Z_i)_{i \in \mathcal{S}}$.) 
    \item \label{prop:Correlation} For any potential hyperedges $e,f \in X$, we have $\mathbb{P}(e \in R,f \in R) \geq \mathbb{P}(e \in R)\mathbb{P}(f \in R)$.
    \item \label{def:Gain-Set-Y} There is a set $Y \subseteq X$ of size $|Y| = \Omega(m)$ such that for every $e \in Y$, we have $|e|\geq 2$ and $\Omega(1) =\mathbb{P}(e \in R) =1 -\Omega(1)$.
    \item \label{prop:Max-degree} For every vertex $u\in V$, there are at most $\Delta$ potential hyperedges $e \in X$ such that $u \in e$.
    \item \label{prop:Codegree} There exists $D=O(1)$ such that for any distinct vertices $u,v\in V$, there are at most $D$ potential hyperedges $e \in X$ such that $u,v \in e$, and for every $u\in V$ and $i\in \mathcal{S}$, there are at most $D$ potential hyperedges $e \in X$ such that $u \in e$ and $i \in S_e$.
    
\end{enumerate}

\end{definition}

\begin{remark}
    In Definition \ref{defn:random linear}, since $X$ is a multiset, it is possible that two potential hyperedges $e$ and $f$ have the same vertex set but $S_e\neq S_f$. To avoid clutter in the proof, we will often refer to the set $S_e$ as the \emph{set of colours} of the potential hyperedge $e$. In what follows, we will call the elements of $X$ hyperedges without writing \emph{potential}.
\end{remark}

\begin{remark}
    Roughly speaking, a random $(m,\Delta,k)$-multihypergraph of linear type arises from exposing a random partial cut on a (deterministic) linear $k$-graph $H_0$ with roughly $m$ edges and maximum degree at most $\Delta$. The hyperedges of $H_0$ are precisely $e\cup S_e$ for each $e\in X$, where the random cut is being exposed on $S_e$. Note that \ref{prop:Max-degree} encodes the maximum degree condition on $H_0$, whereas \ref{prop:Codegree} encodes the (approximate) linearity condition for $H_0$.
\end{remark}

The key result to prove \Cref{thm:linear} is the following.

\begin{theorem}
\label{lem:Key-Random-Linear}
Fix $k \in \mathbb{N}$.
Let $m,\Delta \geq 1$, and let $H$ be a random $(m,\Delta,k)$-multihypergraph of linear type.
Then we have 
\begin{align*}
    \mathbb{E}[\surp(H)] = \Omega\left(\frac{m}{\sqrt{\Delta}}\right).
\end{align*}
\end{theorem}

\begin{proof}
    For a hyperedge $e \in X$ containing vertices $u,v$, we let $B(e,u,v)$ be the indicator of the event $e \in R$ (otherwise, if $e$ does not contain $u$ and $v$, let $B(e,u,v)=0$ whether or not $e\in R$).
    We let $B(u,v)=\sum_{f \in X} B(f,u,v)$, let $L(e,u,v)=B(e,u,v)-\mathbb{E}[B(e,u,v)]$ and let $L(u,v)=B(u,v)-\mathbb{E}[B(u,v)]$.
      Let $c$ be a sufficiently small positive constant. For each $u\in V$, we define $x_u\in \mathbb{R}^V$ by
    $$
    x_u(v)=
    \begin{cases}
    -1, \textrm{ if } v=u\\
    c\cdot \frac{L(u,v)}{\sqrt{\Delta}}, \textrm{ otherwise.}
    \end{cases}
    $$

    Observe that by \ref{prop:Max-degree} and \ref{prop:Codegree}, we have $\|x_u\|^2\leq 1+O(D^2c^2)\leq 2$ for $c$ sufficiently small.
    
We apply \Cref{thm:outline hypergraph SDP} to $H$ and get
\begin{align*}
    \surp(H) \geq  - \sum_{e \in R} C_e \sum_{u\neq v \in e} \langle x_u,x_v\rangle - O(\sum_{e \in R} \sum_{u\neq v \in e} \langle x_u,x_v\rangle^2).
\end{align*}
Taking expectation, this gives 
\begin{align}
\label{eq:Surplus-Random-Linear-Hypergraph}
    \mathbb{E}[\surp(H)] &\geq  \mathbb{E} [ - \sum_{e \in R} C_e \sum_{u\neq v \in e} \langle x_u,x_v\rangle] - O(\mathbb{E}\sum_{e \in R} \sum_{u\neq v \in e} \langle x_u,x_v\rangle^2) \nonumber \\
    &\geq \sum_{e \in X} C_e \sum_{u\neq v \in e} \mathbb{E} [-\ind_{e \in R} \langle x_u,x_v\rangle] - O(\mathbb{E}\sum_{e \in X} \sum_{u\neq v \in e} \langle x_u,x_v\rangle^2).
\end{align}
Our goal is now to show that the RHS of the previous expression is $\Omega\left(\frac{m}{\sqrt{\Delta}}\right)$.
We start by giving a lower bound of the first term.

For a given edge $e \in X$, and distinct vertices $u,v \in e$, we have 
\begin{align}
\label{eq:Main-contribution-Linear-1}
    \mathbb{E}[- \ind_{e \in R} \langle x_u,x_v \rangle] 
    &= \mathbb{E}[-B(e,u,v)\Big(x_u(u)x_v(u)+x_u(v)x_v(v)+\sum_{w\neq u,v} x_u(w)x_v(w)\Big)] \nonumber \\
    &=  \mathbb{E}[B(e,u,v)2c\cdot \frac{L(u,v)}{\sqrt{\Delta}} 
    - \sum_{w \neq u,v} c^2 B(e,u,v) \frac{L(u,w)}{\sqrt{\Delta}}\cdot \frac{L(v,w)}{\sqrt{\Delta}}].
\end{align}
We first focus on the first subterm. Note that 
\begin{align}
\label{eq:Main-contribution-Decomposition}
    \mathbb{E}[2c \cdot B(e,u,v) \frac{L(u,v)}{\sqrt{\Delta}}]= \mathbb{E}[2c \cdot B(e,u,v) \frac{L(e,u,v)}{\sqrt{\Delta}} 
    +2c \cdot B(e,u,v)  \frac{\sum_{f \in X\setminus \{e\}} L(f,u,v)}{\sqrt{\Delta}}].
\end{align}
By definition of $Y$ in \ref{def:Gain-Set-Y}, we have 
\begin{align*}
    \mathbb{E}[2c \cdot B(e,u,v) \frac{L(e,u,v)}{\sqrt{\Delta}}]=\mathbb{E}[2c \cdot \frac{\Var(B(e,u,v))}{\sqrt{\Delta}}] = \Omega\left(\frac{c \ind_{e \in Y}}{\sqrt{\Delta}}\right).
\end{align*}
Furthermore, by \ref{prop:Correlation}, we have 
\begin{align*}
    \mathbb{E}[2c \cdot B(e,u,v) \frac{\sum_{f\in X\setminus \{e\}} L(f,u,v)}{\sqrt{\Delta}}]=\mathbb{E}\Big[2c \cdot \frac{\sum_{f \in X\setminus \{e\}} \Cov(B(e,u,v),B(f,u,v))}{\sqrt{\Delta}}\Big] \geq 0.
\end{align*}
Plugging the two above bounds in \eqref{eq:Main-contribution-Decomposition}, we get  
\begin{align}
\label{eq:Main-contribution-Linear-First-term}
    \mathbb{E}[2c \cdot B(e,u,v) \frac{L(u,v)}{\sqrt{\Delta}}] = \Omega\left(\frac{c \ind_{e \in Y}}{\sqrt{\Delta}}\right).
\end{align}
We now consider the second subterm. Note that
\begin{align*}
\mathbb{E}[\sum_{w \neq u,v} c^2 B(e,u,v) \frac{L(u,w)}{\sqrt{\Delta}}\cdot \frac{L(v,w)}{\sqrt{\Delta}}] = \sum_{w \neq u,v} \sum_{f,f'\in X} c^2 \mathbb{E}[B(e,u,v) \frac{L(f,u,w)}{\sqrt{\Delta}}\cdot \frac{L(f',v,w)}{\sqrt{\Delta}}].
\end{align*}
We remark that when the random variables $B(e,u,v)$, $B(f,u,w)$ and $B(f',u,w)$ are mutually independent, the corresponding expectation vanishes, which is always the case unless at least two of $S_e$, $S_f$ and $S_{f'}$ share at least one colour.
Suppose that $S_e$ and $S_f$ share a colour $i$. Then we have $O(m)$ choices for $(e,u,v)$. As $|S_e|=O(1)$, there are $O(1)$ choices for $i$, and therefore by \ref{prop:Codegree}, we have $O(D)$ choices for $(f,w)$. Finally, by \ref{prop:Codegree} again, we have $O(D)$ choices for $f'$, thereby $O(mD^2)$ choices in this case. 
The two other cases can be dealt with similarly, and we find that there are at most $O(mD^2)$ tuples $(e,f,f',u,v,w)$ for which $\mathbb{E}[B(e,u,v) \frac{L(f,u,w)}{\sqrt{\Delta}}\cdot \frac{L(f',v,w)}{\sqrt{\Delta}}]$ does not vanish, and thus (since $B(e,u,v) \frac{L(f,u,w)}{\sqrt{\Delta}}\cdot \frac{L(f',v,w)}{\sqrt{\Delta}}=O(\frac{1}{\Delta})$ with probability $1$)
\begin{align}
\label{eq:Error-Term-Linear-Cross-Term}
\sum_{e\in X} \sum_{u\neq v\in e} \mathbb{E}[\sum_{w \neq u,v} c^2 B(e,u,v) \frac{L(u,w)}{\sqrt{\Delta}}\cdot \frac{L(v,w)}{\sqrt{\Delta}}] = O\left( \frac{c^2mD^2}{\Delta}\right).
\end{align}
Combining \eqref{eq:Main-contribution-Linear-1}, \eqref{eq:Main-contribution-Linear-First-term} and \eqref{eq:Error-Term-Linear-Cross-Term}, together with $|Y| = \Omega(m)$ from \ref{def:Gain-Set-Y}, we obtain
\begin{align}
\label{eq:Main-contribution-Linear-3}
\sum_{e \in X} C_e \sum_{u\neq v \in e} \mathbb{E} [-\ind_{e \in R} \langle x_u,x_v\rangle] = \Omega\left(\frac{cm}{\sqrt{\Delta}}\right) - O\left( \frac{c^2mD^2}{\Delta}\right).
\end{align}
We now turn to the second term of \eqref{eq:Surplus-Random-Linear-Hypergraph}. 
We fix some hyperedge $e \in X$ and distinct vertices $u,v \in e$.
Using the inequality $(a+b)^2 \leq 2a^2+2b^2$, we have 
\begin{align*}
     \mathbb{E}[\langle x_u,x_v \rangle^2] &= \mathbb{E}\left[\left(2c \cdot \frac{L(u,v)}{\sqrt{\Delta}} - \sum_{w\in V\setminus \{u,v\}} c^2\frac{L(u,w)}{\sqrt{\Delta}}\cdot \frac{L(v,w)}{\sqrt{\Delta}}\right)^2\right] \\
     &\leq \mathbb{E}\left[ 2\left(2c \frac{L(u,v)}{\sqrt{\Delta}}\right)^2 + 2 \left( \sum_{w\in V\setminus \{u,v\}} c^2\frac{L(u,w)}{\sqrt{\Delta}}\cdot \frac{L(v,w)}{\sqrt{\Delta}}\right)^2\right] \\
    &= \mathbb{E}\left[ 8c^2 \frac{L(u,v)^2}{\Delta}\right] 
    + \mathbb{E}\left[2c^4\sum_{w,w' \in V \setminus \{u,v\}} \frac{L(u,w)}{\sqrt{\Delta}}\cdot \frac{L(v,w)}{\sqrt{\Delta}}\cdot\frac{L(u,w')}{\sqrt{\Delta}}\cdot \frac{L(v,w')}{\sqrt{\Delta}}  \right].
\end{align*}

The first term is trivially $O(\frac{D^2c^2}{\Delta})$ by \ref{prop:Codegree}.
We therefore focus on the second term, and we want to bound
        \begin{align}
        \label{eqn:Ebsquared}
            \sum_{e\in X} \sum_{u\neq v \in e} \sum_{f_1,f_2,f_3,f_4} \sum_{w,w'\in V\setminus \{u,v\}} \mathbb{E}[L(f_1,u,w)L(f_2,v,w)L(f_3,u,w')L(f_4,v,w')] 
        \end{align}
        For a given term not to vanish in expectation, each $f_i$ must have at least one colour that also belongs to some $f_j$ with $j \neq i$.
        For the tuple $(f_1,u,w)$, we have $O(m)$ choices.
        Then $f_1$ must share a colour with $f_2$, $f_3$ or $f_4$.
        If it is $f_2$, then we have $O(D)$ choices for $(f_2,v)$ by \ref{prop:Codegree}.
        We then have $O(D)$ choices for $e$, $O(\Delta)$ choices for $(f_3,w')$ by \ref{prop:Max-degree}, and $O(D)$ choices for $f_4$, giving in total $O(mD^3\Delta)$ choices. 
        If $f_1$ shares a colour with $f_3$, then
        we have $O(\Delta)$ choices for $(f_2,v)$ by \ref{prop:Max-degree}, $O(D)$ choices for $e$, $O(D)$ choices for $(f_3,w')$, and $O(D)$ choices for $f_4$, thereby giving a total of $O(mD^3\Delta)$ possibilities. The case where $f_1$ shares a colour with $f_4$ is nearly identical.
        We conclude that 
        \begin{align*}
            \sum_{e\in X} \sum_{u\neq v \in e} \mathbb{E}[2c^4\sum_{w,w' \in V \setminus \{u,v\}} \frac{L(u,w)}{\sqrt{\Delta}}\cdot \frac{L(v,w)}{\sqrt{\Delta}}\cdot\frac{L(u,w')}{\sqrt{\Delta}}\cdot \frac{L(v,w')}{\sqrt{\Delta}} ] = O\left(\frac{c^4mD^3}{\Delta}\right).
        \end{align*}
        It follows that
        \begin{align}
        \label{eq:Main-contribution-Linear-4}
        \mathbb{E}[\sum_{e \in X} \sum_{u\neq v \in e} \langle x_u,x_v\rangle^2] = O\left(\frac{D^2c^2m}{\Delta}\right)+O\left(\frac{c^4mD^3}{\Delta}\right).
        \end{align}
        By plugging in \eqref{eq:Main-contribution-Linear-3} and \eqref{eq:Main-contribution-Linear-4} into \eqref{eq:Surplus-Random-Linear-Hypergraph}, we obtain by taking $c$ sufficiently small that
        \begin{align*}
            \mathbb{E}[\surp(H)] = \Omega\left(\frac{m}{\sqrt{\Delta}}\right),
        \end{align*}
    as wanted.
\end{proof}

With \Cref{lem:Key-Random-Linear} in hand, it is relatively easy to establish \Cref{thm:linear} using the tools developed by Conlon, Fox, Kwan, and Sudakov \cite{conlon2019hypergraph} and R\"aty and Tomon \cite{raty2024large}.
For instance, the following lemma can be obtained from \Cref{lem:Key-Random-Linear} in a similar manner as done by R\"aty and Tomon \cite{raty2024large}.
We therefore defer its proof to the \Cref{Appendix-Auxiliary-Large-3-cut-Linear-Hypergraph}.

\begin{lemma}
\label{lem:Large-3-cut-Linear-Hypergraph}
    Let $H$ be a $3$-multigraph with $O(m)$ edges and maximum codegree $O(1)$. Assume that there is a set $S\subset V(H)$ such that $H[S]$ has $m$ hyperedges, and every vertex in $S$ has degree at most~$\Delta$ in $H$. 
   Then
   \begin{align*}
       \surp_3(H) = \Omega\left( \frac{m}{\sqrt{\Delta}}\right).
   \end{align*}
\end{lemma}

Similarly, from \Cref{lem:Key-Random-Linear}, it is easy to obtain to obtain the following lemma by using the work of Conlon, Fox, Kwan, and Sudakov \cite{conlon2019hypergraph}.
We therefore also defer its proof to the \Cref{Appendix-Proof-Auxiliary-Large-2-cut-Linear-Hypergraph}.

\begin{lemma}
\label{lem:Large-2-cut-Linear-Hypergraph}
Fix $k \geq 4$, and let $H$ be a linear mixed $k$-graph with $O(m)$ edges. Assume that there is a set $S\subset V(H)$ such that there are at least $m$ hyperedges $e\in E(H)$ with $|e| \geq 4$ and $|e \cap S| \geq 2$, and every vertex in $S$ has degree at most $\Delta$ in $H$. Then
    \begin{align*}
       \surp_2(H) = \Omega\left( \frac{m}{\sqrt{\Delta}}\right).
   \end{align*}
\end{lemma}

We are now ready to prove \Cref{thm:linear}.

\begin{proof}[Proof of \Cref{thm:linear}]
    We start with the case where $r \in \{k-1,k \}$.
    Let $H$ be a $k$-uniform linear hypergraph with $m$ edges.
    Let $H_k=H$, and for $i=k-1, \dots, 3$, we let $H_i$ be the underlying $i$-graph of $H_{i+1}$.
    Then by \Cref{lem:Surplus-Relation-UnderlyingGraph}, we have $\surp_{i+1}(H_{i+1}) \geq \frac{1}{4(i+1)}\surp_{i}(H_{i})$ for all $3\leq i\leq k-1$. Furthermore, by the same lemma, we have $\surp_{k-1}(H_k)\geq \frac{1}{2}\surp_{k-1}(H_{k-1})$.
    It follows that in either case $r\in \{k-1,k\}$, we have $\surp_r(H)= \Omega_r(\surp_3(H_3))$.
    Note that $H_3$ is a multihypergraph with maximum codegree $D=O(1)$ and $\tilde{m}=\Theta(m)$ hyperedges.

    Let $G$ be the underlying multigraph of $H_3$.
    Thus $e(G) = 3\tilde{m}$.
    We let $\Delta = 100 D^{1/2} \tilde{m}^{1/2}$, let $T$ be the set of vertices of $G$ which have degree more than $\Delta$ and let $S=V(G) \setminus T$.
    Since $6\tilde{m} = 2e(G) \geq |T| \Delta$, we have $|T| \leq 0.1\tilde{m}^{1/2}D^{-1/2}$.
    Since the multiplicity of every edge of $G$ is at most $D$, we thus have $e(G[T]) \leq D |T|^2/2 \leq \tilde{m}/200$. \\
    If $e_G(S,T) \geq 1.6\tilde{m}$, we have $\surp(G) \geq 0.1\tilde{m}$. 
    Since by \Cref{lem:Surplus-Relation-UnderlyingGraph}, we have $\surp_3(H_3) \geq \surp(G)/12$, we are done in this case. \\
    If $e_G(S,T) < 1.6\tilde{m}$, then $e(G[S]) = 3\tilde{m} -e(G[T]) - e_G(S,T) \geq 1.3\tilde{m}$.
    However,
    \begin{align*}
        e(G[S]) \leq 3e(H_3[S]) + e_{H_3}(S,T) \leq 3e(H_3[S]) + \frac{1}{2}e_G(S,T)
    \end{align*}
    (where $e_{H_3}(S,T)$ is the number of edges in $H_3$ intersecting both $S$ and $T$), and therefore $e(H_3[S]) \geq 0.1\tilde{m}$. 
    Since each vertex in $S$ has degree at most $\Delta = 100 D^{1/2} \tilde{m}^{1/2}$ in $H_3$, \Cref{lem:Large-3-cut-Linear-Hypergraph} implies that $\surp_3(H_3)=\Omega(m^{3/4})$, from which we conclude that $\surp_r(H)=\Omega(m^{3/4})$.

    We now turn to the remaining cases, i.e. $2 \leq r \leq k-2$. 
    By \Cref{lem:Full-Reduction-2-cut-Linear}, it is enough to prove that if $H'$ is a linear mixed $k$-graph with $O(m)$ hyperedges in total and at least $m$ hyperedges of size at least $4$, then $\surp(H')=\Omega(m^{3/4})$. 
    We start similarly as in the proof of the cases $r=k=3$ by letting $\Delta = k m^{1/2}$, $T$ be the set of vertices of $H'$ which have degree more than $\Delta$ and $S=V(H') \setminus T$.
    Since $km \geq |T| \Delta$, we have $|T| \leq m^{1/2}$.
    Therefore the number of edges in $H'$ with at least $2$ vertices in $T$ is at most $|T|^2/2 \leq m/2$.
    Thus, the number of edges in $H'$ with at least $4$ vertices in total and at least $2$ vertices in $S$ is $\Omega(m)$, and every vertex in $S$ has degree at most $\Delta$ in $H'$. Hence, by \Cref{lem:Large-2-cut-Linear-Hypergraph}, we have $\surp(H')=\Omega(m^{3/4})$. 
\end{proof}

\section{Upper bound for the surplus of linear hypergraphs} \label{sec:linear upper}

In this section we prove Proposition \ref{thm:Upper-Bound-Linear-Hypergraphs}.

\begin{proof}[Proof of \Cref{thm:Upper-Bound-Linear-Hypergraphs}.]
    Note that we may assume that $m$ is sufficiently large compared to $k$. Given $m$, let $n=Cm^{1/2}$ for some large constant $C$ that may depend on $k$. 
    Let $m'$ be another parameter to be specified.
    Let $e_1,e_2,\dots,e_{m'}$ be independent and uniformly random $k$-sets in $[n]$. 
    Let $E = \{ e_i: i \in [m'] \text{  such that } \forall j \in [m']\setminus \{i\}, |e_i \cap e_j| \leq 1 \}$, i.e. the set of $e_i$ which do not intersect any other $e_j$ in at least two vertices. 
    Let $H$ be the $k$-uniform hypergraph with vertex set $[n]$ and edge set $E$. 
    By definition, $H$ is linear. 
    Let $f(m')$ be the expected number of edges in $H$.
    It is clear that $f(m') = m'(1-p)^{m'-1}$, where $p= \frac{\sum_{i=2}^k \binom{k}{i}\binom{n-k}{k-i}}{\binom{n}{k}} = \frac{\beta_k}{n^2} + O(\frac{1}{n^3})$ for some constant $\beta_k$.
    Therefore, for $C$ large enough, there exists $m'=\Theta_k(m)$ such that $m+m^{2/3}<f(m')<m+2m^{2/3}$.
    Note that $|E|$ is $(\binom{k}{2}+1)$-Lipschitz with respect to revealing each $e_i$, and therefore by Azuma's inequality (see \Cref{lem:azuma}), for any constant $\gamma>0$ we have with high probability that
    \begin{align*}
        ||E| -f(m')| \leq \gamma(m')^{2/3}.
    \end{align*}
    As $m'=O(m)$, it follows by choosing $\gamma$ sufficiently small that with high probability we have $m \leq |E| \leq m+3m^{2/3}$.
    We delete edges if necessary in an arbitrary way to obtain an $m$-edge subgraph $G$. 
    We now show that with high probability $G$ has $r$-surplus $O(m^{3/4})$.
    Clearly, it suffices to prove that with high probability $H$ has $r$-surplus $O(m^{3/4})$.
    We argue that this holds by taking a union bound over all $r$-cuts.
    We fix some $r$-cut $\psi : [n] \rightarrow [r]$.
    Let $\phi$ be the random variable associated to the size of $\psi$ as an $r$-cut of $H$.
    Then $\phi$ is clearly $(\binom{k}{2}+1)$-Lipschitz with respect to revealing each $e_i$, and therefore by Azuma's inequality,
    \begin{align}
    \label{eq:Azuma-Fixed-Cut}
        \mathbb{P}(|\phi-\mathbb{E}[\phi]| \geq C'm^{3/4}) \leq \exp( -2C (\log r) m^{1/2})
    \end{align}
    for some constant $C'=C'(k)$.
    Let $q_{\psi}$ be the probability that $\psi$ cuts a random $k$-set in $[n]$.
    For $i \in [r]$, we let $n_i = |\psi^{-1} (i)|$. We claim that $q_{\psi}$ is maximal when $\psi$ is balanced, i.e., when $|n_i-n_j|\leq 1$ for each $i,j$. To see this, assume that $n_i\geq n_j+2$ for some $i,j$. Consider an $r$-cut $\psi'$ which agrees with $\psi$ on all but one vertex $v$, but has $\psi'(v)=j$ whereas $\psi(v)=i$. Now if $e$ is a random $k$-set in $[n]$, reveal $e\cap \left(\cup_{\ell\in [r]\setminus \{i,j\}} \psi^{-1}(\ell)\right)$. Conditioning on any outcome, the probability that $\psi'$ cuts $e$ is at least as large as the probability that $\psi$ cuts $e$. Hence, $q_{\psi}$ is indeed maximal when $\psi$ is balanced.

    Let $q=\frac{S(k,r)r!}{r^k}$ be the probability that a uniformly random $r$-cut cuts a given hyperedge. It is easy to see that when $\psi$ is balanced, then $q_{\psi}=(1+O(1/n))q$, and therefore we have $q_{\psi}\leq (1+O(1/n))q$ for any fixed $r$-cut $\psi$.
    Now observe that $\mathbb{E}[\phi]=q_{\psi}\mathbb{E}[e(H)]$. Indeed, $\phi=\sum_{i=1}^{m'} \ind_{e_i\in H} \ind_{e_i \textrm{ is cut by } \psi}$, so
    \begin{align*}
    \mathbb{E}[\phi]
    &=\sum_{i=1}^{m'} \mathbb{P}(e_i\in H \textrm{ and } e_i \textrm{ is cut by } \psi)=\sum_{i=1}^{m'} \mathbb{P}(e_i\in H) \mathbb{P}(e_i \textrm{ is cut by } \psi)=q_{\psi}\sum_{i=1}^{m'} \mathbb{P}(e_i\in H) \\
    &=q_{\psi}\mathbb{E}[e(H)].
    \end{align*}
    Therefore
    \begin{align*}
        \mathbb{E}[\phi] =q_{\psi} \mathbb{E}[e(H)]=q_{\psi}f(m') \leq \left(1+O\left(\frac{1}{n}\right)\right) q(m+2m^{2/3}) = \frac{S(k,r)r!}{r^k}m + O(m^{2/3}).
    \end{align*}
    Plugging this in \eqref{eq:Azuma-Fixed-Cut}, we find that there exists a constant $C''=C''(k)$ such that
    \begin{align*}
        \mathbb{P}(\phi-\frac{S(k,r)r!}{r^k}m \geq C''m^{3/4}) \leq \exp( -2C (\log r) m^{1/2}).
    \end{align*}
    Since with high probability, $H$ has at least $m$ edges, it follows by a union bound over all $r^n=r^{Cm^{1/2}}$ possible $r$-cuts $\psi$ of $[n]$ that with high probability $H$ has $r$-surplus $O(m^{3/4})$.
\end{proof}

\section{General hypergraphs} \label{sec:general}

In this section we prove Theorem \ref{thm:Hypergraph-Excess}. We start with the following analogue of Definition \ref{defn:random linear}. Some explanation is given after the definition.

\begin{definition}[Random $(m,\Delta,D,\Lambda,\lambda,k)$-multihypergraph of general type] \label{defn:random general}
    Fix some positive integer $k$ and a vertex set $V$.
We say that a random multihypergraph $H=(V,R)$ is a \emph{random $(m,\Delta,D,\Lambda,\lambda,k)$-multihypergraph of general type} if it satisfies the following properties: 
\begin{enumerate}[label=(\roman*)]
    \item \label{prop:Def-X-Y} There are multisets $X$ and $Y$ consisting of subsets of $V$ of size at most $k$, and $R$ is a random subset of $X\cup Y$, defined as follows: there exist a set $\mathcal{S}$ and some independent random variables $(Z_i)_{i \in \mathcal{S}}$ such that for every $e \in X\cup Y$, there exists a set $S_e \subseteq \mathcal{S}$ of size at most $k$ such that the event $e \in R$ is a function of $\{ Z_i : i \in S_e \}$.
    
    \noindent (We will call the elements of $X\cup Y$ \emph{potential hyperedges} in the hypergraph $H$, as whether they belong to the edge set $R$ of $H$ depends on the random variables $(Z_i)_{i \in \mathcal{S}}$.)
    
    \item \label{prop:General-Sizes} $|X\cup Y|=O(m)$ and $|Y| = \Omega(m/\log m)$. 
    \item There is a random subset $T\subseteq Y$ such that for every $e\in Y$, the event $e\in T$ is a function of $\{Z_i:i\in S_e\}$. 
    \item \label{prop:Independence-T-T-T} For any $e_1,\dots,e_j,f\in Y$, the event $f\in T$ is independent\footnote{For events $A, A_1, \dots, A_{\ell}$, we say that $A$ is independent of $A_1, \dots, A_{\ell}$ if for any events $(B_i)_{i \in [\ell]}$ such that $B_i=A_i$ or $B_i=A_i^c$ for each $i \in [\ell]$, we have $\mathbb{P}(A,B_1, \dots, B_{\ell})=\mathbb{P}(A)\mathbb{P}(B_1, \dots, B_{\ell})$.} of the events $e_1\in T,\dots,e_j\in T$ unless $S_f\subseteq S_{e_1}\cup \dots \cup S_{e_j}$. 
    \item \label{prop:Independence-T-R-T} For any $e\in X\cup Y$ and $f,g\in Y$, the event $g\in T$ is independent of the events $ e\in R, f\in T$ unless $S_g\subseteq S_e\cup S_f$. 
    \item \label{prop:Correlation-R-T} For any $e\in X\cup Y$ and $f \in Y$, we have $\mathbb{P}(e \in R,f \in T)\geq \mathbb{P}(e \in R)\mathbb{P}(f \in T)$. 
    \item \label{prop:Correlation-Gain-R-T} For any $e,f\in Y$ with $S_e=S_f$, we have $\mathbb{P}(e \in R,f \in T)-\mathbb{P}(e \in R)\mathbb{P}(f \in T)\geq \Omega(1)$. 
    \item \label{prop:Edges-size-General} For every $e \in Y$, we have $|e|= 2$, and $|S_e|$ is the same for each $e\in Y$. 
    \item \label{prop:Max-Degree-Def} For every vertex $u\in V$, there are at most $\Delta$ potential hyperedges $e \in X\cup Y$ such that $u \in e$.
    \item \label{prop:Def-codegree-2} For any distinct vertices $u,v\in V$, there are at most $D$ potential edges $e \in X\cup Y$ such that $u,v \in e$, and for every $u\in V$ and $i\in \mathcal{S}$, there are at most $D$ potential edges $e \in Y$ such that $u \in e$ and $i \in S_e$.
    \item \label{prop:Bounded-Full-Codegree} For every $e \in Y$, the number of potential edges $f\in Y$ with $S_e=S_f$ is at most $\Lambda\lambda$.
    \item \label{prop:Multiplicity} For each $e\in Y$, there are at least $\lambda/2$ but at most $\lambda$ potential edges $f\in Y$ with $V(e)=V(f)$ (i.e., such that $e$ and $f$ have the same vertex set) and $S_e=S_f$.
    \item \label{prop:Concentration-SDP-Edges-General} there exists some constant $C$ such that with probability at least $1-m^{-10}$, for every distinct $u,v\in V$, the following holds: letting $B(u,v)$ be the number of potential edges in $T$ with vertex set $\{u,v\}$ and letting $e_g(u,v)$ be the number of potential edges in $Y$ with vertex set $\{u,v\}$, we have $|B(u,v)-\mathbb{E}[B(u,v)]|\leq C(\log m)^{k/2}\lambda^{1/2}e_g(u,v)^{1/2}$.
    
\end{enumerate}
\end{definition}

\begin{remark}
    In Definition \ref{defn:random general}, since $X$ and $Y$ are multisets, it is possible that two potential hyperedges $e$ and $f$ have the same vertex set but $S_e\neq S_f$. To avoid clutter in the proof, we will often refer to the set $S_e$ as the \emph{set of colours} of the potential hyperedge $e$. In what follows, we will often call the elements of $X\cup Y$ hyperedges without writing \emph{potential}.
\end{remark}

\begin{remark}
    Roughly speaking, a random $(m,\Delta,D,\Lambda,\lambda,k)$-multihypergraph of general type will arise from revealing a random partial cut on a (deterministic) $k$-multigraph $H_0$, with the following degree conditions:
    \begin{itemize}
        \item the maximum degree of $H_0$ is at most $\Delta$
        \item the maximum codegree of a pair of vertices in $H_0$ is at most $D$
        \item for a large proportion of hyperedges in $H_0$, the multiplicity is roughly $\lambda$
        \item the maximum codegree of a set of $k-2$ vertices in $H_0$ is at most $\Lambda \lambda$ (or at most $\Lambda$ if multiplicities are ignored).
    \end{itemize}
    Properties \ref{prop:Max-Degree-Def}, \ref{prop:Def-codegree-2}, \ref{prop:Multiplicity} and \ref{prop:Bounded-Full-Codegree} (respectively) make this intuition formal. The hyperedges of $H_0$ are $e\cup S_e$ for each $e\in X\cup Y$.

    Compared to Definition \ref{defn:random linear}, the most significant difference in Definition \ref{defn:random general} is that we now have two sets $X$ and $Y$ of potential hyperedges in $H$. Roughly speaking, $Y$ is the set of potential hyperedges (in fact potential \emph{edges}: see property \ref{prop:Edges-size-General}) that we can ``control well'' -- note that there are several conditions in Definition \ref{defn:random general} that apply only to elements of $Y$ and not to $X$. Property \ref{prop:General-Sizes} states that the number of such ``controlled'' potential edges is comparable to the total number of potential edges. The set $T$ of potential edges is sampled from this good set $Y$, and these are the edges that we will use to define our SDP vectors $x_u$, for each $u\in V$. Although this is not a condition, one should think of $T$ as a superset of $R\cap Y$ -- in particular, whether a potential edge $e\in Y$ gets sampled to $R$ will correlate with the event $e\in T$ (see property \ref{prop:Correlation-R-T} applied with $f=e$).
\end{remark}

The key result to prove \Cref{thm:Hypergraph-Excess} is the following.

\begin{theorem}
\label{thm:Key-Random-General}
    Fix $k\in \mathbb{N}$ and let $m,\Delta,D,\Lambda,\lambda \geq 1$ be some reals. 
    Let $H=(V,R)$ be a random $(m,\Delta,D,\Lambda,\lambda,k)$-multihypergraph of general type. 
    Then we have 
    \begin{align*}
            \mathbb{E}[\surp(H)] = \Omega\left(\frac{m \lambda}{\beta \log m}\right) - O\left( \frac{m\Lambda \lambda D}{\beta^2}+ \frac{m D^3 \lambda}{\beta^4}\right),
    \end{align*}
    whenever $\beta \geq (\log m)^{(k+1)/2}\sqrt{\Delta \lambda}$.
\end{theorem}

Before we prove \Cref{thm:Key-Random-General}, we establish the following counting lemma, which we will use to get a good control on the error terms in the proof of \Cref{thm:Key-Random-General}.

\begin{lemma}
\label{lem:Counting-Error-Term}
    Fix $k\in \mathbb{N}$ and let $m,\Delta,D,\Lambda,\lambda$ be some positive integers. 
    Let $H=(V,R)$ be a random $(m,\Delta,D,\Lambda,\lambda,k)$-multihypergraph of general type.
    Let $Q$ be the number of tuples $(u,v,w,w',f_1,f_2,f_3,f_4,e)$ such that $f_1, f_2,f_3,f_4 \in Y$ satisfy $S_{f_i} \subseteq \bigcup_{j \neq i} S_{f_j}$ for every $i \in [4]$, $e \in X \cup Y$ and $u,v,w,w'\in V$ satisfy that $u,v,w$ are distinct, $u,v,w'$ are distinct, $u,w \in f_1$, $v,w \in f_2$, $u,w' \in f_3$, $v,w' \in f_4$ and $u,v \in e$.
    Then
    \begin{align*}
        Q = O(m \Lambda \lambda^2 \Delta D+m D^3 \lambda).
    \end{align*}
\end{lemma}

We remark that the proof of \Cref{lem:Counting-Error-Term} uses only items \ref{prop:Def-X-Y}, \ref{prop:General-Sizes}, \ref{prop:Edges-size-General}, \ref{prop:Max-Degree-Def}, \ref{prop:Def-codegree-2}, \ref{prop:Bounded-Full-Codegree}, and \ref{prop:Multiplicity} in \Cref{defn:random general} of a random multihypergraph of general type.

\begin{proof}
        Firstly, we note that for such a tuple, once $(u,v,w,w',f_1,f_2,f_3,f_4)$ is fixed, we have $O(D)$ choices for $e$ by \ref{prop:Def-codegree-2}, and therefore we simply focus on bounding the number of tuples $(u,v,w,w',f_1,f_2,f_3,f_4)$ such that $f_1, f_2,f_3,f_4 \in Y$ satisfy $S_{f_i} \subseteq \bigcup_{j \neq i} S_{f_j}$ for every $i \in [4]$, and $u,v,w,w'\in V$ satisfy that $u,v,w$ are distinct, $u,v,w'$ are distinct, $u,w \in f_1$, $v,w \in f_2$, $u,w' \in f_3$, and $v,w' \in f_4$. \\
        For the tuple $(f_1,u,w)$, we have $O(m)$ choices by \ref{prop:Def-X-Y} and \ref{prop:General-Sizes}.
        Since $S_{f_1} \subseteq \bigcup_{j \neq 1} S_{f_j}$, then $f_1$ must share at least one colour with either $f_2$, $f_3$ or $f_4$. \\
        Suppose first that $f_1$ shares at least one colour with $f_2$, and we distinguish between two further cases, whether $S_{f_1}=S_{f_2}$ or not. \\
        In the first subcase (i.e. if $S_{f_1}=S_{f_2}$), we have $O(\Lambda \lambda)$ choices for $(f_2,v)$ by \ref{prop:Bounded-Full-Codegree}.
        We then have $O(\Delta)$ choices for $(f_3,w')$ by \ref{prop:Max-Degree-Def}, and since $S_{f_4} \subseteq \bigcup_{j \neq 4} S_{f_j}$, we have $O(1)$ choices for $S_{f_4}$, and therefore by \ref{prop:Multiplicity}, we have $O(\lambda)$ choices for $f_4$.
        This gives $O(m \Lambda \lambda^2 \Delta)$ choices in total in this subcase. \\
        In the second subcase (i.e. if $S_{f_1}\neq S_{f_2}$), we have $O(D)$ choices for $(f_2,v)$ by \ref{prop:Def-codegree-2}.
        Since $S_{f_1} \subseteq \bigcup_{j \neq 1} S_{f_j}$, $S_{f_1} \neq S_{f_2}$, and all $S_f$ for $f \in Y$ have the same size by \ref{prop:Edges-size-General}, we must have some colour $i$ such that $i \in S_{f_1} \cap S_{f_\ell}$ for some $\ell =3,4$.
        Therefore, we have $O(D)$ choices for $(f_\ell,w')$ by \ref{prop:Def-codegree-2}.
        Let $h$ be such that $f_h$ has not been chosen yet.
        As $S_{f_h} \subseteq \bigcup_{j \neq h} S_{f_j}$, we have $O(1)$ choices for $S_h$, and thus $O(\lambda)$ choices for $f_h$ by \ref{prop:Multiplicity}.
        Thus we have $O(m D^2 \lambda)$ choices in total in this subcase. \\
        The case where $f_1$ shares at least one colour with $f_3$ is symmetric to the case where $f_1$ shares at least one colour with $f_2$, and thus we have the same bounds here. \\
        We now argue in the last remaining case, that is, where $f_1$ shares a colour with $f_4$.
        As we have already dealt with the cases where $f_1$ shares a colour with $f_2$ or $f_3$, we may assume that $S_{f_1}=S_{f_4}$.
        We have $O(\Delta)$ choices for $(f_2,v)$ by \ref{prop:Max-Degree-Def}. 
        Then we have $O(\Lambda \lambda)$ choices for $(f_4,w')$ by \ref{prop:Bounded-Full-Codegree}, and $O(\lambda)$ choices for $f_3$ by \ref{prop:Multiplicity} (as there are $O(1)$ choices for $S_{f_3}$).
        Thus we have $O(m \Lambda \lambda^2 \Delta)$ choices in total in this subcase. \\
        In total, taking into account the $O(D)$ choices for $e$, we find that $Q=O(m \Lambda \lambda^2 \Delta D+m D^3 \lambda)$, as desired.
\end{proof}

We are now ready to prove \Cref{thm:Key-Random-General}.

\begin{proof}[Proof of \Cref{thm:Key-Random-General}]
    For a hyperedge $e \in Y$ such that $e=\{ u,v \}$, we let $B(e,u,v)$ be the indicator of the event $e \in T$ (otherwise, if $e \neq \{ u,v \}$, let $B(e,u,v)=0$ whether or not $e\in T$).
    We let $B(u,v)=\sum_{f \in Y} B(f,u,v)$, let $L(e,u,v)=B(e,u,v)-\mathbb{E}[B(e,u,v)]$ and let $L(u,v)=B(u,v)-\mathbb{E}[B(u,v)]$. \\
    By \ref{prop:Concentration-SDP-Edges-General}, for some constant $C$, we have that, with probability at least $1-m^{-10}$, for every $u,v$ the following holds:
    \begin{align}
    \label{eq:Concentration-B(u,v)-k-cut-k-graph}
        |L(u,v)|=|B(u,v)-\mathbb{E}[B(u,v)]| \leq C(\log m)^{k/2} \lambda^{1/2} e_g(u,v)^{1/2},
    \end{align}
    where $e_g(u,v)$ is the number of edges $e \in Y$ with vertex set $\{u,v\}$. We let $\Psi$ be the associated event. \\
    For each $u\in V$, we define $x_u\in \mathbb{R}^V$ by taking
    \begin{align}
    \label{def:vectors-xu}
         x_u(v)=
    \begin{cases}
    -1, \textrm{ if } v=u\\
    \frac{L(u,v)}{\beta}, \textrm{ otherwise.}
    \end{cases}
    \end{align}
    If $\Psi$ holds, then by \eqref{eq:Concentration-B(u,v)-k-cut-k-graph} and using that $\beta \geq (\log m)^{(k+1)/2}\sqrt{\Delta \lambda}$, we have for every $u$ that 
    \begin{align*}
        1 \leq ||x_u||^2 = 1 + \sum_{v \in V, v \neq u} \frac{L(u,v)^2}{\beta^2} \leq 1 +  \sum_{v \in V, v \neq u} \frac{C^2 (\log m)^{k} \lambda e_g(u,v)}{\Delta \lambda (\log m)^{k+1}} \leq 2,
    \end{align*}
    for large enough $m$, where the last inequality uses \ref{prop:Max-Degree-Def}. \\
    Applying \Cref{thm:outline hypergraph SDP}, we get
    \begin{align*}
    \surp(H) \geq \ind_{\Psi} \left( -\sum_{e \in R} C_e \sum_{u\neq v \in e} \langle x_u,x_v\rangle - O(\sum_{e \in R} \sum_{u\neq v \in e} \langle x_u,x_v\rangle^2) \right).
\end{align*}
    Taking expectation, we obtain 
    \begin{align}
        \mathbb{E}[\surp(H)] &\geq- \mathbb{E}[\ind_{\Psi} \sum_{e \in R} C_e \sum_{u\neq v \in e} \langle x_u,x_v\rangle] - O(\mathbb{E}[ \ind_{\Psi}\sum_{e \in R} \sum_{u\neq v \in e} \langle x_u,x_v\rangle^2]) \label{eqn:with indicator}
    \end{align}
    We will now argue that we may remove both occurrences of $\ind_{\Psi}$ in the previous equation and still get a valid lower bound on $\mathbb{E}[\surp(H)]$, up to a tiny error term.
    Indeed, for every $u,v \in V$, by Cauchy-Schwarz we have $|\langle x_u, x_v \rangle | \leq ||x_u|| \cdot||x_v||$. Since $|L(u,v)|=O(m)$ by \ref{prop:General-Sizes} and $\beta \geq 1$, we have $||x_u|| =O(m^{3/2}) $ for every $u \in V$ (with probability $1$), and therefore $|\langle x_u, x_v \rangle | = O(m^3)$ and $\langle x_u, x_v \rangle^2  =O(m^6)$.
    Therefore 
    \begin{align*}
        - \mathbb{E}[\ind_{\bar{\Psi}} \sum_{e \in R} C_e \sum_{u\neq v \in e} \langle x_u,x_v\rangle] - O(\mathbb{E}[ \ind_{\bar{\Psi}}\sum_{e \in R} \sum_{u\neq v \in e} \langle x_u,x_v\rangle^2]) &= O(\mathbb{E}[\ind_{\bar{\Psi}} m^4]+\mathbb{E}[\ind_{\bar{\Psi}}m^7]) \\
        &= O(m^7/m^{10})=O(m^{-3}).
    \end{align*}
    It then follows from (\ref{eqn:with indicator}) that
    \begin{align}
    \label{eq:Surplus-Random-General-Hypergraph}
    \mathbb{E}[\surp(H)] &\geq- \mathbb{E}[\sum_{e \in R} C_e \sum_{u\neq v \in e} \langle x_u,x_v\rangle] - O(\mathbb{E}[\sum_{e \in R} \sum_{u\neq v \in e} \langle x_u,x_v\rangle^2]) -O(m^{-3}) \nonumber \\
    &\geq \sum_{e \in X \cup Y} C_e \sum_{u\neq v \in e} \mathbb{E} [-\ind_{e \in R} \langle x_u,x_v\rangle] - O(\mathbb{E}\sum_{e \in X \cup Y} \sum_{u\neq v \in e} \langle x_u,x_v\rangle^2) -O(m^{-3}).
\end{align}
    Our goal is now to give a lower bound of the RHS of the previous expression.
    We start by giving a lower bound of the first term of \eqref{eq:Surplus-Random-General-Hypergraph}.
    For a hyperedge $e \in X \cup Y$ and $u,v \in e$, we let $K(e,u,v)$ be the indicator of the event $e \in R$ (otherwise, if $e \not \supset \{ u,v \}$, let $K(e,u,v)=0$ whether or not $e\in R$).
    For a given hyperedge $e \in X \cup Y$, and distinct vertices $u,v \in e$, we have
\begin{align}
\label{eq:Main-contribution-General-1}
    \mathbb{E}[- \ind_{e \in R} \langle x_u,x_v \rangle] =  \mathbb{E}[2K(e,u,v) \frac{L(u,v)}{\beta}] 
    -\mathbb{E}[ \sum_{w \neq u,v} K(e,u,v) \frac{L(u,w)}{\beta}\cdot \frac{L(v,w)}{\beta}].
\end{align}
    We first focus on the first subterm, and write 
    \begin{align}
    \label{eq:Main-contribution-Decomposition-General}
        \mathbb{E}[ K(e,u,v) \frac{L(u,v)}{\beta}]= \mathbb{E}[ K(e,u,v) \sum_{f \in M^1_e} \frac{L(f,u,v)}{\beta} 
        + K(e,u,v) \sum_{f \in M^2_e} \frac{ L(f,u,v)}{\beta}],
    \end{align}
where, if $e \in Y$, then $M^1_e$ is the set of edges $f \in Y$ containing $u$ and $v$ such that $S_e=S_f$, and $M^2_e$ is the set of edges $f \in Y$ containing $u$ and $v$ such that $S_e \neq S_f$, and, if $e \in X$, then $M^1_e= \emptyset$ and $M^2_e = Y$.
By \ref{prop:Correlation-Gain-R-T}, we have
\begin{align*}
     \mathbb{E}[ K(e,u,v) \sum_{f \in M^1_e} \frac{L(f,u,v)}{\beta}]= \Omega\left(\frac{|M^1_e|}{\beta}\right).
\end{align*}
By \ref{prop:Multiplicity}, we have that $|M^1_e| \geq \lambda/2$ if $e \in Y$, and therefore
\begin{align*}
     \mathbb{E}[ K(e,u,v) \sum_{f \in M^1_e} \frac{L(f,u,v)}{\beta}]= \Omega\left(\frac{\lambda \ind_{e \in Y}}{\beta}\right).
\end{align*}
Furthermore, by \ref{prop:Correlation-R-T}, we have 
    \begin{align*}
       \mathbb{E}[ K(e,u,v) \sum_{f \in M^2_e} \frac{ L(f,u,v)}{\beta}] \geq 0.
    \end{align*}
Plugging the two above bounds in \eqref{eq:Main-contribution-Decomposition-General}, we get  
\begin{align}
\label{eq:Main-contribution-General-First-term}
    \mathbb{E}[K(e,u,v) \frac{L(u,v)}{\beta}] = \Omega\left(\frac{ \lambda\ind_{e \in Y}}{\beta}\right).
\end{align}
We now consider the second subterm in (\ref{eq:Main-contribution-General-1}).
Note that 
\begin{align}
\label{eq:Sum-Split-L}
\sum_{e \in X \cup Y} C_e \sum_{u \neq v \in e} &\mathbb{E}[\sum_{w \neq u,v} K(e,u,v) \frac{L(u,w)}{\beta}\cdot \frac{L(v,w)}{\beta}] = \nonumber \\
&\sum_{e \in X \cup Y} C_e  \sum_{u \neq v \in e} \sum_{w \neq u,v} \sum_{f,f'\in Y} \mathbb{E}[ K(e,u,v) \frac{L(f,u,w)}{\beta}\cdot \frac{L(f',v,w)}{\beta}].
\end{align}
We remark that, by \ref{prop:Independence-T-R-T}, unless $S_f \subseteq S_e \cup S_{f'}$ and $S_{f'} \subseteq S_e \cup S_{f}$, we have 
\begin{align*}
    \mathbb{E}[ K(e,u,v) L(f,u,w)L(f',v,w)]=0.
\end{align*} 
We now count those terms $ K(e,u,v) L(f,u,w)L(f',v,w)$ in \eqref{eq:Sum-Split-L} which do not vanish in expectation.
We first distinguish between two cases:
\begin{itemize}
    \item $S_f = S_{f'}$,
    \item there exists some colour $i$ such that $i \in S_f$ and $i \in S_e$,
\end{itemize}
which cover all possibilities as, by \ref{prop:Edges-size-General}, we have $|S_f|=|S_{f'}|$ for all $f,f'\in Y$. 
In the first case, we have $O(m \Lambda \lambda )$ possibilities for $(f,f',u,v,w)$ by \ref{prop:Bounded-Full-Codegree}.
Then, by \ref{prop:Def-codegree-2}, we have at most $O(D)$ choices for $e$, thereby we have $O(m \Lambda \lambda D)$ choices for $(e,f,f',u,v,w)$ in total.
In the second case, we have at most $O(m)$ possibilities for $(e,u,v)$, then $O(D)$ possibilities for $(f,w)$. As $S_{f'} \subseteq S_e \cup S_{f}$, we have $O(1)$ possibilities for $S_{f'}$, and therefore $O(\lambda)$ possibilities for $f'$ by \ref{prop:Multiplicity}. 
In total, this gives $O(mD \lambda)$ possibilities. 
Therefore, as $\Lambda \geq 1$, combining all this gives that 
\begin{align}
\label{eq:Error-Term-General-Cross-Term}
\sum_{e \in X \cup Y} C_e \sum_{u\neq v\in e} \mathbb{E}[\sum_{w \neq u,v} K(e,u,v) \frac{L(u,w)}{\beta} \frac{L(v,w)}{\beta}] = O\left( \frac{m\Lambda \lambda D}{\beta^2}\right).
\end{align}
Combining \eqref{eq:Main-contribution-General-1}, \eqref{eq:Main-contribution-General-First-term} and \eqref{eq:Error-Term-General-Cross-Term}, together with $|Y| = \Omega(m/ \log m)$ from \ref{prop:General-Sizes}, we obtain
\begin{align}
\label{eq:Main-contribution-General-3}
\sum_{e \in X \cup Y} C_e \mathbb{E} [\sum_{u\neq v \in e} -\ind_{e \in R} \langle x_u,x_v\rangle] = \Omega\left(\frac{m \lambda}{\beta \log m}\right) - O\left( \frac{m\Lambda \lambda D}{\beta^2}\right).
\end{align}
We now turn to the second term of \eqref{eq:Surplus-Random-General-Hypergraph}.
We fix some hyperedge $e \in X \cup Y$ and distinct vertices $u,v \in e$.
Using the inequality $(a+b)^2\leq 2a^2+2b^2$, we have 
\begin{align}
     \mathbb{E}[\langle x_u,x_v \rangle^2] &= \mathbb{E}\left[\left(-2 \frac{L(u,v)}{\beta} + \sum_{w\in V\setminus \{u,v\}} \frac{L(u,w)}{\beta}\cdot \frac{L(v,w)}{\beta}\right)^2\right] \nonumber \\
     &\leq \mathbb{E}\left[2\left(2\frac{L(u,v)}{\beta}\right)^2 + 2\left(\sum_{w\in V\setminus \{u,v\}} \frac{L(u,w)}{\beta}\cdot \frac{L(v,w)}{\beta}\right)^2\right] \nonumber \\
    &= \mathbb{E}\left[ 8 \frac{L(u,v)^2}{\beta^2}\right] + \mathbb{E}\left[2 \sum_{w,w' \in V \setminus \{u,v\}} \frac{L(u,w)}{\beta}\cdot \frac{L(v,w)}{\beta}\cdot\frac{L(u,w')}{\beta}\cdot \frac{L(v,w')}{\beta}  \right]. \label{eq:second moment 4 terms}
\end{align}
For the first term, we have
\begin{align*}
    \mathbb{E}[L(u,v)^2]=\sum_{f,f'\in Y}\mathbb{E}[ L(f,u,v)L(f',u,v)].
\end{align*}
By \ref{prop:Independence-T-T-T}, we have $\mathbb{E}[ L(f,u,v)L(f',u,v)]=0$ unless $S_f=S_{f'}$, and therefore by \ref{prop:Multiplicity}, we have for every fixed $f$ that $\sum_{f'\in Y}\mathbb{E}[ L(f,u,v)L(f',u,v)]=O(\lambda)$.
Summing over $O(D)$ choices for $f$ by \ref{prop:Def-codegree-2} then gives
\begin{align*}
    \mathbb{E}\left[\frac{L(u,v)^2}{\beta^2}\right]=O\left( \frac{\lambda D}{\beta^2}\right).
\end{align*}
Turning to the second term in (\ref{eq:second moment 4 terms}), we want to bound
        \begin{align}
        \label{eqn:Ebsquared-General}
            \sum_{e\in X\cup Y} \sum_{u\neq v \in e} \sum_{w,w'\in V\setminus \{u,v\}} \sum_{f_1,f_2,f_3,f_4\in Y} \mathbb{E}[L(f_1,u,w)L(f_2,v,w)L(f_3,u,w')L(f_4,v,w')] 
        \end{align}
        For a given expectation term not to vanish, by \ref{prop:Independence-T-T-T}, we must have for every $i \in [4]$ that $S_{f_i} \subseteq \bigcup_{j \neq i} S_{f_j}$.
        Let $Q$ be the number of tuples $(u,v,w,w',f_1,f_2,f_3,f_4,e)$ such that $f_1, f_2,f_3,f_4 \in Y$ satisfy $S_{f_i} \subseteq \bigcup_{j \neq i} S_{f_j}$ for every $i \in [4]$, $e \in X \cup Y$ and $u,v,w,w' \in V$ satisfy that $u,v,w$ are distinct, $u,v,w'$ are distinct, $u,w \in f_1$, $v,w \in f_2$, $u,w' \in f_3$, $v,w' \in f_4$ and $u,v \in e$.
        By \Cref{lem:Counting-Error-Term}, we have $Q=O(m \Lambda \lambda^2 \Delta D+m D^3 \lambda)$, and therefore there are at most $O(m \Lambda \lambda^2 \Delta D+m D^3 \lambda)$ terms which do not vanish in \eqref{eqn:Ebsquared-General}.
        Since each of those terms is $O(1)$, we get
        \begin{align*}
            \sum_{e\in X\cup Y} \sum_{u\neq v \in e} \mathbb{E}[\sum_{w,w' \in V \setminus \{u,v\}} \frac{L(u,w)}{\beta}&\cdot \frac{L(v,w)}{\beta}\cdot\frac{L(u,w')}{\beta}\cdot \frac{L(v,w')}{\beta} ] = O\left(\frac{m \Lambda \lambda^2 \Delta D+m D^3 \lambda}{\beta^4}\right).
        \end{align*}
        It follows that 
        \begin{align}
        \label{eq:Main-contribution-General-4}
        \sum_{e\in X\cup Y} \sum_{u\neq v \in e} \mathbb{E} [\langle x_u,x_v\rangle^2] =  
        O\left(\frac{m\lambda D}{\beta^2}+\frac{m \Lambda \lambda^2 \Delta D+m D^3 \lambda}{\beta^4}\right).
        \end{align}
        Plugging in \eqref{eq:Main-contribution-General-3} and \eqref{eq:Main-contribution-General-4} into \eqref{eq:Surplus-Random-General-Hypergraph}, we obtain
        \begin{align*}
            \mathbb{E}[\surp(H)] = \Omega\left(\frac{m \lambda}{\beta \log m}\right) - O\left( \frac{m\Lambda \lambda D}{\beta^2}+ \frac{m \Lambda \lambda^2 \Delta D+m D^3 \lambda}{\beta^4}\right)-O(m^{-3}).
        \end{align*}
    As $\beta \geq (\log m)^{(k+1)/2}\sqrt{\Delta \lambda}$, we have $\frac{m \Lambda \lambda^2 \Delta D}{\beta^4} = O(\frac{m\Lambda \lambda D}{\beta^2})$, and therefore we can ignore the $O(\frac{m \Lambda \lambda^2 \Delta D}{\beta^4})$ term.
    We can also ignore the $O(m^{-3})$ term using that $\mathbb{E}[\surp(H)]= \Omega(1)$.
    Therefore,
    \begin{align*}
            \mathbb{E}[\surp(H)] = \Omega\left(\frac{m \lambda}{\beta \log m}\right) - O\left( \frac{m\Lambda \lambda D}{\beta^2}+ \frac{m D^3 \lambda}{\beta^4}\right),
        \end{align*}
    as wanted. 
\end{proof}

The following lemma will be used to verify that condition \ref{prop:Concentration-SDP-Edges-General} is satisfied in our applications of Theorem \ref{thm:Key-Random-General}.

\begin{lemma}
\label{lem:Concentration-Modulos-Multiplicity}
    Let $b,k$ be positive integers, and let $H$ be a $k$-uniform multihypergraph on a vertex set of size at most $n$.
    Suppose that we give independently to each vertex a colour uniformly at random in $\mathbb{Z}_b$.
    Fix $a \in \mathbb{Z}_b$, and let $X_a$ be the random variable counting the number of hyperedges in which the sum of the colours of the vertices is equal to $a$ modulo $b$.
    Then for every constant $C_1$, there exists a constant $C_2=C_2(k,b,C_1)$, such that for $n$ sufficiently large, we have with probability at least $1-n^{-C_1}$ that
    \begin{align*}
        |X_a-\mathbb{E}[X_a]| \leq C_2 (\log n)^{k/2} \cdot (\sum_{e \in E(H)} \lambda_e^2)^{1/2},
    \end{align*}
    where $\lambda_e$ is the multiplicity of the hyperedge $e$.
\end{lemma}

\begin{proof}
    We prove this lemma by induction on $k$. 
    For $k=1$, the result follows by a standard application of Azuma's inequality (see \Cref{lem:azuma}).
    Suppose now that the result is true for $k-1$.
    We may assume that $|V(H)|=n$ by adding, if necessary, isolated vertices.
    We take an arbitrary ordering $v_1, \dots, v_n$ of the vertices of $H$, and reveal their colour one by one.
    We construct a martingale $(M_j)_{0 \leq j \leq n}$ as follows.
    First, we set $M_0=0$.
    Before we reveal the colour of vertex $v_j$, we let $H_j$ be the hypergraph on vertex set $\{v_1,\dots,v_{j-1}\}$ and edge set consisting of $e \setminus v_j$ for every hyperedge $e \in E(H)$ with $j = \max\{ i : v_i \in e \}$.
    We consider the following process: if there exists a colour $\ell\in \mathbb{Z}_b$ such that in $H_{j}$, the number $X^j_\ell$ of hyperedges on which the sum of the colours is congruent to $\ell$ modulo $b$ satisfies $|X^j_\ell-\frac{e(H_{j})}{b}| > C'_2 (\log n)^{(k-1)/2} \cdot (\sum_{e\in E(H): v_j \in e} \lambda_e^2)^{1/2}$, where $C'_2=C(k-1,b,C_1+2)$, we say we terminate with failure, and set $M_n=\dots=M_j$.
    Otherwise, we set $M_{j+1}=M_j + X^j_z - \frac{e(H_{j})}{b}$, where $z$ is the congruence class of $a-c(v_j)$, and $c(v_j)$ is the colour assigned to $v_j$.
    We then move on to reveal the colour of $v_{j+1}$.
    We remark that, by induction hypothesis, we terminate with failure at step $j$ with probability at most $n^{-C_1-2}$.
    Therefore, by union bound, the probability of terminating with failure is at most $n^{-C_1-1}$.
    We also note that $(M_j)_{0 \leq j \leq n}$ is a martingale, whose $j$-th increment is bounded from above by $C'_2 (\log n)^{(k-1)/2} (\sum_{e\in E(H): v_j \in e} \lambda_e^2)^{1/2}$.
    By Azuma's inequality (see \Cref{lem:azuma}), we get 
    \begin{align*}
        \mathbb{P}(|M_n| >  C (\log n)^{k/2} (\sum_{e\in E(H)} \lambda_e^2)^{1/2} ) &\leq 2 \exp\left( -\frac{(C (\log n)^{k/2} (\sum_{e\in E(H)} \lambda_e^2)^{1/2})^2}{2\sum_{j=1}^{n} (C'_2 (\log n)^{(k-1)/2} \cdot (\sum_{e\in E(H): v_j \in e} \lambda_e^2)^{1/2})^2} \right) \\
        &= 2 \exp\left( -\frac{C^2 \log n \sum_{e\in E(H)} \lambda_e^2 }{2(C'_2)^2\sum_{j=1}^{n} \sum_{e\in E(H): v_j \in e} \lambda_e^2} \right).
    \end{align*}
    Since $\sum_{j=1}^{n} \sum_{e\in E(H): v_j \in e} \lambda_e^2 = k \sum_{e\in E(H)} \lambda_e^2$, it follows that for $C$ sufficiently large
    \begin{align}
    \label{eq:Bound-tail-Mn}
        \mathbb{P}(|M_n| >  C (\log n)^{k/2} (\sum_{e\in E(H)} \lambda_e^2)^{1/2} ) \leq 2 \exp\left(-\frac{C^2 \log n}{2k(C'_2)^2}\right) \leq n^{-C_1-1}.
    \end{align}
    Since $M_n= X_a-\mathbb{E}[X_a]$ unless we terminate with failure, we have that $M_n= X_a-\mathbb{E}[X_a]$ holds with probability at least $1-n^{-C_1-1}$. 
    This together with \eqref{eq:Bound-tail-Mn} give the desired conclusion.
\end{proof}

The next lemma and its corollary will be used to verify that conditions \ref{prop:Bounded-Full-Codegree} and \ref{prop:Multiplicity} are satisfied in our applications of Theorem \ref{thm:Key-Random-General}.

\begin{lemma} \label{lem:bounding (k-1)-codegrees}
    Let $\Gamma=\binom{k}{2}m^{\frac{1}{\lfloor k/2\rfloor}}$ and let $G$ be a $k$-uniform hypergraph with $m$ edges. Then there exist an ordering $e^1,\dots,e^m$ of the edges of $G$ and some $u^i,v^i\in e^i$ (with $u^i\neq v^i$) such that for each $1\leq i\leq m-1$, the number of $j>i$ for which $e^j$ contains $e^i\setminus \{u^i,v^i\}$ is at most $\Gamma$.
\end{lemma}

\begin{proof}
    It suffices to prove that there exist some $e\in E(G)$ and distinct $u,v\in e$ with the property that
    the number of $f\in E(G)\setminus \{e\}$ containing $e\setminus \{u,v\}$ is at most $\Gamma$.
    Indeed, if such $e$ and $u,v$ exist, then we can just set $e^i=e$, $u^i=u$, $v^i=v$ and continue with the hypergraph $G-e$.

    Now assume that such $e$ and $u,v$ do not exist. Let $f^0=\{x_1,x_2,\dots,x_k\}$ be an arbitrary hyperedge in $G$. By assumption, there are more than $\Gamma$ edges $f^1\in E(G)$ containing $f^0\setminus \{x_1,x_2\}$. For each such $f^1$, by assumption there are more than $\Gamma$ edges $f^2\in E(G)$ containing $f^1\setminus \{x_3,x_4\}$. Continue in this manner, defining, for each choice of $f^i$, more than $\Gamma$ different $f^{i+1}$s, all in $E(G)$. Observe that, for a given $f^0$, the number of different choices for the sequence $(f^1,f^2,\dots,f^{\lfloor k/2\rfloor})$ is greater than $\Gamma^{\lfloor k/2\rfloor}$. Moreover, for any given $f^0$ and $f^{\lfloor k/2\rfloor}$, the number of possibilities for $(f^1,\dots,f^{\lfloor k/2\rfloor})$ is at most $\binom{k}{2}^{\lfloor k/2 \rfloor}$ since $f^{i+1}\setminus (f^i\setminus \{x_{2i+1},x_{2i+2}\})\subset f^{\lfloor k/2\rfloor}$ for each $i$. Since $f^{\lfloor k/2\rfloor}\in E(G)$, it follows that $\Gamma^{\lfloor k/2\rfloor}<\binom{k}{2}^{\lfloor k/2\rfloor} m$, which is a contradiction.
\end{proof}

From \Cref{lem:bounding (k-1)-codegrees}, we deduce the following corollary.

\begin{Corollary}
\label{cor:Set-Gain-Bound-codegrees}
    Let $H$ be a $k$-uniform multihypergraph with $m$ edges.
    Then there exists a sub-multihypergraph $H'$ of $H$ with at least $\Omega \left(\frac{m}{\log m}\right)$ edges such that there are some positive integers $\lambda, \Lambda$ with $\Lambda = O(m^{\frac{1}{\lfloor k/2\rfloor}})$ and a subset $S(e)\subset e$ of size $k-2$ for every $e\in E(H')$ such that
    \begin{itemize}
        \item the multiplicity of every edge $e \in E(H')$ is at least $\lambda/2$ and at most $\lambda$,
        \item for every edge $e\in E(H')$, the number of edges $f \in E(H')$ with $S(e)=S(f)$ is at most $\lambda \Lambda$.
    \end{itemize}
\end{Corollary}

\begin{proof}
    As the multiplicity of every edge is between $1$ and $m$, a simple dyadic pigeon holing argument shows that there exists a sub-multihypergraph $H^{*}$ of $H$ with at least $\Omega \left(\frac{m}{\log m}\right)$ edges such that the multiplicity of every edge $e \in E(H^*)$ is at least $\lambda/2$ and at most $\lambda$, for some integer $\lambda$.
    We now consider the hypergraph $\tilde{H}$ obtained by keeping only one copy of every edge of $H^{*}$.
    For $M=|E(\tilde{H})|$ we choose an ordering $e^1,e^2,\dots,e^M$ of $E(\tilde{H})$ and vertices $u^i,v^i\in e^i$ for every $i \in [M]$ according to Lemma~\ref{lem:bounding (k-1)-codegrees}, and let $S(e^i)=e^i\setminus \{u^i,v^i\}$.
    If for some $i<j$, we have $S(e^i)=S(e^j)$, then in particular $e^i\setminus \{u^i,v^i\}=S(e^i)\subset e^j$. 
    By construction, the number of such $e^j$ is at most $\Gamma$, for any given $e^i$, for some $\Gamma = O(M^{\frac{1}{\lfloor k/2\rfloor}})$. 
    Therefore we have at most $M \Gamma$ pairs $(e,f) \in E(\tilde{H})^2$ such that $S(e)=S(f)$.
    Let $R$ be the set of edges $e \in E(\tilde{H})$ such that there are at least $4\Gamma$ edges $f \in E(\tilde{H})$ with $S(e)=S(f)$.
    Then $2\Gamma|R| \leq M \Gamma$, giving $|R| \leq M/2$.
    Letting $\Lambda = 4\Gamma = O(m^{\frac{1}{\lfloor k/2\rfloor}})$ and $H'$ be the sub-multihypergraph of $H^*$ consisting of those edges (with their multiplicity in $H^*$) that are not in $R$, we have that $|E(H')| =\Omega \left(\frac{m}{\log m}\right)$ and for every edge $e \in E(H')$, there are at most $\Lambda \lambda $ edges $f \in E(H')$ such that $S(e)=S(f)$.
\end{proof}

We now turn to the proof of Theorem \ref{thm:Hypergraph-Excess}. We distinguish between two cases: $r\leq k-q$ for some large constant $q=q(\eps)$, and $r>k-q$.

\subsection{The case of small $r$}

We first deal with the case where $r$ is not too large in \Cref{thm:Hypergraph-Excess}.

\begin{theorem}
    \label{thm:Hypergraph-Excess-r-small}
    For every $\eps > 0$, there exist some positive integer $q$, such that for every $2 \leq r \leq k-q$, every $k$-uniform hypergraph with $m$ edges has an $r$-cut of surplus $\Omega(m^{2/3-\eps})$. 
\end{theorem}

\begin{proof}
    First of all, by \Cref{lem:Reduction-Main-Small-r-Lem}, it suffices to show that for every $\eps > 0$, there exists some positive integer $q$ such that every mixed $k$-multigraph $H$ with maximum degree at most $\Delta= \Theta(m^{2/3-\eps})$, maximum codegree $D=O(m^{1/3-2\eps})$, $O(m)$ hyperedges in total, and $\Omega(m)$ hyperedges of size $q+2$ has a $2$-cut of surplus $\Omega(m^{2/3-\eps})$.
    More precisely, we will show that every sufficiently large even integer $q$ will be suitable. Clearly, we may assume that $m$ is sufficiently large. \\ 
    We apply \Cref{cor:Set-Gain-Bound-codegrees} to the subhypergraph of $H$ consisting of the edges of size exactly $q+2$ and obtain a set $B \subseteq E(H)$ of hyperedges such that $|B| = \Omega \left(\frac{m}{\log m}\right)$ and there exist positive integers $\lambda, \Lambda$ with $\Lambda = O(m^{\frac{1}{\lfloor (q+2)/2\rfloor}})$ and a subset $S(e)\subset e$ of size $q$ for every $e\in B$ such that
    \begin{itemize}
        \item the multiplicity of every edge $e \in B$ is at least $\lambda/2$ and at most $\lambda$,
        \item for every edge $e\in B$, the number of edges $f \in B$ with $S(e)=S(f)$ is at most $\lambda \Lambda$.
    \end{itemize}
    We now take a random partition $W \cup W^c$ of $V(H)$ by putting each vertex with probability $1/2$ in $W$, and in $W^c$ otherwise, independently for each vertex. 
    We say that a hyperedge $e\in E(H)$ is \emph{good} if $e\in B$ and $e \cap W^c =S(e)$, and $e$ is \emph{bad} otherwise.
    Note that the expected number of good hyperedges is $\Omega(m/\log m)$, and therefore we may fix an outcome of $W$ such that the number of good hyperedges is $\Omega(m/\log m)$.
    \\
    Let $\sigma:W^c\rightarrow \{1,2\}$ be chosen uniformly at random.
    We now construct the following auxiliary multihypergraph $H_{\sigma}$ on vertex set $W$: for every edge $e \in E(H)$, if all vertices of $e$ are in $W$ then we add two copies of $e$, otherwise, if $|\sigma(e \setminus W)| = 1$ and $|e \cap W| > 0$, we add one copy of $e \cap W$. Let $R=E(H_{\sigma})$.
    By \Cref{lem:Reduction-Reveal-Outside-W}, we have
    \begin{align}
    \label{eq:Surplus-Ineq-Hsigma-2}
        \surp(H) \geq \mathbb{E}_{\sigma}[\surp(H_{\sigma})]/2.
    \end{align}
    We now claim that $H_{\sigma}=(W,R)$ is a random $(m,2\Delta,2D,\Lambda,\lambda,k)$-multihypergraph of general type. 
     Indeed, it suffices to verify that all conditions in Definition \ref{defn:random general} are satisfied. 
     \begin{enumerate}[label=(\roman*)]
     \item We let $Y= \{e \cap W : e\in E(H) \text{ is good} \}$, and we let $X=\{e \cap W : e \in E(H) \text{ is bad and } e\cap W\neq \emptyset\} $, where, if $e \subset W$, the element $e \cap W=e$ appears with multiplicity $2$. 
    For every hyperedge $e \in E(H)$ with representative $e_W :=e\cap W \in X \cup Y$, we let $S_{e_W} = e \cap W^c$, and for every vertex $a \in W^c$, we let $Z_a=\sigma(a)$. Then the event $\{ e_W \in R \}$ is a function of $\{ Z_i : i \in S_{e_W} \}$: more precisely, we have $\{ e_W \in R \}$ if and only if all the $Z_i$ for $i \in S_{e_W}$ have the same value.
    \item We have that $|X \cup Y|\leq 2e(H) = O(m)$, and by our choice of $W$, we have $|Y| = \Omega(m/ \log m)$.
    \item We let $T$ consist of all $e \in  Y$ for which $\sum_{v \in S_e} \sigma(v)$ is congruent to $0$ modulo $2$. 
    \item Immediate from the definition of $T$.
    \item Immediate from the definition of $T$.
    \item If $S_f \nsubseteq S_e$, then we have $\mathbb{P}(e\in R, f\in T)=\mathbb{P}(e\in R)\mathbb{P}(f\in T)$. If $S_f \subseteq S_e$, then we let $A=S_e\setminus S_f$.
    For $i=1,2$, we let $U^i_A$ be the event that $\sigma(v)=i$ for every $v \in A$, and $U^i_{S_f}$ be the event that $\sigma(v)=i$ for every $v \in S_f$.
    Then 
    \begin{align*}
        \mathbb{P}(e \in R , f \in T) = \mathbb{P}(U^1_A \cap U^1_{S_f} \cap \{f \in T\}) + \mathbb{P}(U^2_A \cap U^2_{S_f}  \cap  \{f \in T\}).
    \end{align*}
    As for $i=1,2$, $U^i_{S_f} \subseteq \{ f \in T \} $ (as $q$ is even) and $U^i_A$ is independent from $U^i_{S_f}$, we have
    \begin{align*}
        \mathbb{P}(e \in R , f \in T) = \mathbb{P}(U^1_A)\mathbb{P}(U^1_{S_f}) + \mathbb{P}(U^2_A)\mathbb{P}(U^2_{S_f}) &= 2^{-|A|}2^{-|S_f|}+2^{-|A|}2^{-|S_f|} \\
        &=  2^{-|A|-|S_f|+1} \geq  \mathbb{P}(e \in R)\mathbb{P}( f \in T).
    \end{align*}    
    \item As $\{ e \in R \} \subseteq \{ e \in T \} = \{ f \in T\}$  (where the inclusion is due to $q$ being even), we are left to prove that $\mathbb{P}(e \in R) \left(1-\mathbb{P}(e \in T)\right) = \Omega(1)$.
    Since it is easy to check that $\mathbb{P}(e \in T) = 1 - \Omega(1)$ and $ \mathbb{P}(e \in R)= \Omega(1)$, we get the desired property.
    \item It is immediate from the definition of $Y$ that $|e|=2$ and $|S_e|=q$ for each $e\in Y$.
    \item Immediate.
    \item Immediate.
    \item Immediate from the definition of $Y$.
    \item Immediate from the definition of $Y$.
    \item Let $u,v$ be distinct vertices in $W$.
    We let $H_{u,v}$ be the $q$-uniform multihypergraph on vertex set $W^c$, where we add the hyperedge $S_e$ for every $e \in Y$ with vertex set $\{u,v\}$.
    Note that $e_g(u,v)=|E(H_{u,v})|$ and that the multiplicity of every hyperedge in $H_{u,v}$ is at most $\lambda$ by the previous item.
    Note further that $\sigma$ induces a uniformly random colouring of the vertex set of $H_{u,v}$ with colours in $\mathbb{Z}_2$, and that $B(u,v)$ corresponds to the number of hyperedges in $H_{u,v}$ in which the sum of the colours of the vertices is equal to $0$ in $\mathbb{Z}_2$. Ignoring isolated vertices, the number of vertices in $H_{u,v}$ is $O(m)$.
    Thus, by applying \Cref{lem:Concentration-Modulos-Multiplicity} for $H_{u,v}$ with $a=0$, $b=2$, $k=q$ and $C_1=13$, there exists a constant $C=C(q,b,C_1)$ such that, with probability at least $1-\Omega(m^{-13})$, we have
    \begin{align*}
        |B(u,v)-\mathbb{E}[B(u,v)]| \leq C (\log O(m))^{q/2} \lambda^{1/2}e_g(u,v)^{1/2} \leq C (\log m)^{k/2} \lambda^{1/2}e_g(u,v)^{1/2}.
    \end{align*}
    Taking a union bound over all distinct $u,v \in W$ (for which there are at most $O(m^2)$ choices, ignoring isolated vertices in $H$), we obtain the desired result.
     \end{enumerate}
We now apply \Cref{thm:Key-Random-General} and get
\begin{align*}
            \mathbb{E_{\sigma}}[\surp(H_{\sigma})] = \Omega\left(\frac{m \lambda}{\beta \log m}\right) - O\left( \frac{m\Lambda \lambda D}{\beta^2}+\frac{m D^3 \lambda}{\beta^4}\right),
    \end{align*}
    where $\beta=(\log m)^{(k+1)/2}\sqrt{2\Delta \lambda}$.
    Recalling that $\Delta=\Theta(m^{2/3-\eps})$, $D=O(m^{1/3-2\eps})$ and $\Lambda=O(m^{\frac{1}{\lfloor (q+2)/2\rfloor}})$, we get that
    \begin{align*}
        \frac{m\Lambda \lambda D}{\beta^2}+\frac{m D^3 \lambda}{\beta^4} &= \frac{m \lambda}{\beta \log m}\left( \frac{ \Lambda D \log m}{\beta} + \frac{ D^3 \log m}{\beta^3}\right) \\
        &=\frac{m \lambda}{\beta \log m} \cdot O\left( \frac{\Lambda D}{\sqrt{\Delta}}+ \frac{D^3}{\Delta^{3/2}} \right) \\
        &=\frac{m \lambda}{\beta \log m} \cdot O\left(m^{\frac{1}{\lfloor (q+2)/2\rfloor}-3\eps/2}+ m^{-9\eps/2}  \right).
    \end{align*}
    For $q$ sufficiently large, we have $m^{\frac{1}{\lfloor (q+2)/2\rfloor}-3\eps/2}+ m^{-9\eps/2}=o(1)$, which shows that
    \begin{align*}
        \frac{m\Lambda \lambda D}{\beta^2}+\frac{m D^3 \lambda}{\beta^4} = o\left(\frac{m \lambda}{\beta \log m}\right),
    \end{align*}
    and therefore
    \begin{align*}
        \mathbb{E_{\sigma}}[\surp(H_{\sigma})] = \Omega\left(\frac{m \lambda}{\beta \log m}\right) = \Omega\left( m^{2/3} \right).
    \end{align*}
    Together with \eqref{eq:Surplus-Ineq-Hsigma-2}, we obtain the desired conclusion. 
\end{proof}

\subsection{The case of large $r$}

We now deal with the case of large $r$ in \Cref{thm:Hypergraph-Excess}, by showing the following. 

\begin{theorem}
    \label{thm:Hypergraph-Excess-r-large}
     For every $\eps > 0$, and every positive integer $q$, there exists an integer $k_0$, such that for every $k \geq k_0$ and $k-q \leq r \leq k$, every $k$-uniform hypergraph with $m$ edges has an $r$-cut of surplus $\Omega(m^{2/3-\eps})$.
\end{theorem}

In the proof, we will use the following lemma, whose proof will be given in Section \ref{sec:modulo lemma}.

\begin{restatable}{lemma}{modulolemma}
\label{lem:Correlation-Modulo-Constraint}
    For every nonnegative integer $t$, there exists $r_0 \in \mathbb{N}$ such that for every prime $r \geq r_0$, the following holds.
    We let $f : [r+t] \rightarrow [r]$ be a uniformly random surjective function.
    Then the probability that $\sum_{i=1}^r f(i)$ is congruent to $r(r+1)/2$ modulo $r$ is greater than $1/r$.
\end{restatable}

Note that Lemma \ref{lem:Correlation-Modulo-Constraint} is trivial for $t=0$. When $r=k$, we only need this special case of the lemma, so the reader who is not interested in the case $k-q\leq r\leq k-1$ may skip Section \ref{sec:modulo lemma}. \\

We first prove the following lemma which ensures that some auxiliary hypergraph arising in the proof of \Cref{thm:Hypergraph-Excess-r-large} is random of general type.

\begin{lemma}
\label{lem:Check-Random-General-Type}
    Let $q$ be a fixed positive integer.
    There exists $p_0$ such that the following holds.
    Let  $p \geq p_0$ be prime and let $\kappa$ be some integer such that $p+2\leq \kappa \leq p + q +2$.
    Let $H=(V,E)$ be a mixed $\kappa$-multigraph such that $|E| \geq m$, where $m$ is sufficiently large.
    Let $W \subseteq W' \subseteq V$, and let $W^c=V \setminus W$.
    Suppose that there exist a subset $\mathcal{B}$ of $E$ such that all vertices of every edge of $\mathcal{B}$ are in $W'$, and some reals $\Delta, D, \Lambda, \lambda \geq 1$ such that
    \begin{itemize}
        \item For distinct $u,v \in W'$, $\deg_H(v) \leq \Delta$ and $\deg_H(u,v) \leq D$. 
    \item  For every $e \in \mathcal{B}$, the number of edges $f\in \mathcal{B}$ with $e \setminus W=f \setminus W$ is at most $\Lambda\lambda$.
    \item The multiplicity of each $e\in \mathcal{B}$ is at least $\lambda/2$ and at most $\lambda$. 
    \item $|\mathcal{B}| = \Omega(\frac{m}{\log m})$ and for every edge $e \in \mathcal{B}$, we have $|e|=p+2$ and $|e \cap W|=2$.
    \end{itemize}
    Let $\pi : W^c \rightarrow \{ 1, \dots, p+2 \}$ be a random function such that  $\pi(v)$ is sampled uniformly at random from $\{3, \dots, p+2 \}$ if $v \in W' \setminus W$, and uniformly at random from $\{1, \dots, p+2 \}$ if $v \in V \setminus W'$, independently for each vertex $v \in W^c$.
    Let $H_{\pi}$ be the multihypergraph with vertex set $W$ and edge set $R:=\{ e \cap W: e \in E(H), \{3, \dots, p+2\} \subseteq \pi(e \cap W^c) , \{1, \dots, p+2\} \not \subset  \pi(e \cap W^c) \} $, where two copies of $e\cap W$ are added to $R$ if $\pi(e \cap W^c) \cap \{1, \dots, p+2\} =  \{3, \dots, p+2\}$.
    Then $H_{\pi}$ is a random $(m,2\Delta,2D,\Lambda,2\lambda,\kappa)$-multihypergraph of general type.
\end{lemma}

\begin{proof}
    It suffices to verify that all conditions in Definition \ref{defn:random general} are satisfied. 
     \begin{enumerate}[label=(\roman*)]
        \item We let $Y= \{e \cap W : e \in \mathcal{B}  \}$, each appearing with multiplicity two, and we let $X=\{e \cap W : e \in E \setminus \mathcal{B} \} $, where each element also appears with multiplicity $2$. 
        For every hyperedge $e \in E(H)$ with representatives $e^1_W, e^2_W =e\cap W \in X \cup Y$, we let $S_{e^1_W}, S_{e^2_W} = e \setminus W$, and for every vertex $a \in W^c$, we let $Z_a=\pi(a)$. 
        Then we let $ e^1_W \in R=E(H_{\pi}) $ if and only if the set $\{\pi(v): v \in S_{e^1_W}\}$ contains the set $\{ 3, \dots, p+2\}$, but does not contain the set $\{ 1, \dots, p+2\}$, and we let $ e^2_W \in R=E(H_{\pi}) $ if and only if the set $\{\pi(v): v \in S_{e^2_W}\}$ is equal to the set $\{ 3, \dots, p+2\}$ (possibly with multiplicity).    
        (Note that if $e^1_W$ and $e^2_W$ are the representatives of a hyperedge $e \in Y$, then the events $ e^1_W \in R $ and $ e^2_W \in R $ coincide because $S_{e_W^1}=S_{e_W^2}\subset W'$ in this case.) Observe that this defines the same random set $R$ that was given in the statement of the lemma.
        \item $|X \cup Y|\leq 2e(H)=O(m)$ and, by assumption, we have $|Y| = 2|\mathcal{B}|= \Omega(m / \log m)$. 
        \item We let $T$ consist of all $e \in  Y$ for which $\sum_{v \in S_e} \pi(v)$ is congruent to $\sum_{i=3}^{p+2} i = \frac{p(p+1)}{2}$ modulo $p$. 
        \item Immediate from the definition of $T$.
        \item Immediate from the definition of $T$.
        \item If $S_f \nsubseteq S_e$, then we have $\mathbb{P}(e\in R,f\in T)=\mathbb{P}(e\in R)\mathbb{P}(f\in T)$. Assume therefore that $S_f \subseteq S_e$.
        Since $f \in Y$, we have that $|S_f|=p$, and therefore $|S_e\setminus S_f| \leq \kappa  -p \leq q+2$. 
        Moreover, for $f \in Y$, since all vertices of $f$ are in $W'$, we have that each $\pi(v)$ for $v \in S_f$ is uniformly random in $\{3, \dots, p+2\}$, independently of the other vertices, and therefore $\mathbb{P}(f \in T) = 1/p$.
        
        We first deal with the case where $e \in X\cup Y$ is a first representative of an edge $e' \in E(H)$ in the definition of $X$ and $Y$, i.e. $e \in R$ if and only if $\{\pi(v): v \in S_e\}$ contains $\{3, \dots, p+2\}$ but does not contain $\{1, \dots, p+2\}$.
        We claim that 
        \begin{align}
        \label{eq:Conditioning-Stars}
            \mathbb{P}(f\in T&,e\in R,\pi^{-1}(1)\cap S_e=A_1,\pi^{-1}(2)\cap S_e=A_2) \nonumber \\
            &\geq \mathbb{P}(f \in T)\mathbb{P}(e \in R,\pi^{-1}(1)\cap S_e=A_1,\pi^{-1}(2)\cap S_e=A_2)
        \end{align}
        for all disjoint pairs $A_1,A_2\subseteq S_e\setminus W'\subseteq S_e\setminus S_f$.
        Summing over all such $A_1$ and $A_2$ then gives the desired inequality $\mathbb{P}(f\in T, e\in R)\geq \mathbb{P}(f\in T)\mathbb{P}(e\in R)$, so it suffices to establish \eqref{eq:Conditioning-Stars}.
        Note that we may assume that at least one of $A_1$ and $A_2$ is empty, as otherwise both sides of \eqref{eq:Conditioning-Stars} are $0$.

        Now observe that under these assumptions on $A_1$ and $A_2$, the event $$\{e\in R,\pi^{-1}(1)\cap S_e=A_1,\pi^{-1}(2)\cap S_e=A_2\}$$ is equal to the event 
        \begin{align*}
            \mathcal{E}_{A_1,A_2}:=\{\pi(S_e\setminus (A_1\cup A_2))=\{3,\dots,p+2\},\pi^{-1}(1)\cap S_e=A_1,\pi^{-1}(2)\cap S_e=A_2\},
        \end{align*}
        so it suffices to prove that $\mathbb{P}(f\in T,\mathcal{E}_{A_1,A_2})\geq \mathbb{P}(f\in T)\mathbb{P}(\mathcal{E}_{A_1,A_2})$, or, equivalently, that
        $$\mathbb{P}(f\in T \hspace{1mm} | \hspace{1mm} \mathcal{E}_{A_1,A_2})\geq \mathbb{P}(f\in T).$$

         Clearly, $\mathbb{P}(f\in T \hspace{1mm} | \hspace{1mm} \mathcal{E}_{A_1,A_2})$ is equal to the probability that $\sum_{v\in S_f} \pi(v)\equiv p(p+1)/2$ modulo $p$, conditional on $\pi$ being surjective from $S_e\setminus (A_1\cup A_2)$ to $\{3,\dots,p+2\}$.
        By Lemma \ref{lem:Correlation-Modulo-Constraint} (applied with $r=p$ and $t=|S_e\setminus (S_f\cup A_1\cup A_2)|$), there exists some $p_0=p_0(t)$ such that for all $p\geq p_0$, the conditional probability above is at least $1/p$. Thus, since $\mathbb{P}(f\in T)=1/p$, we indeed have $\mathbb{P}(f\in T \hspace{1mm} | \hspace{1mm} \mathcal{E}_{A_1,A_2})\geq \mathbb{P}(f\in T)$.
        Note that $p_0$ only depends on $|S_e\setminus (S_f\cup A_1\cup A_2)|$, which is bounded from above by $q+2$. Therefore, there is $p_0=p_0(q)$ such that for all $p \geq p_0$, the inequality \eqref{eq:Conditioning-Stars} holds simultaneously for all $A_1,A_2$.
        
        The case where $e \in X\cup Y$ is a second representative of an edge $e' \in E(H)$ in the definition of $X$ and $Y$ (i.e. where $e \in R$ if and only if $\{\pi(v): v \in S_e\}$ is equal to $\{3, \dots, p+2\}$) can be treated in a similar way.
        \item As $\{ e \in R \} \subseteq \{ e \in T \} = \{ f \in T\}$ (since $3 + \dots + (p+2) \equiv \frac{p(p+1)}{2}$ modulo $p$), we are left to prove that $\mathbb{P}(e \in R) \left(1-\mathbb{P}(e \in T)\right) = \Omega(1)$.
        Since it is easy to check that $\mathbb{P}(e \in T) = 1 - \Omega(1)$ and $ \mathbb{P}(e \in R)= \Omega(1)$, we get the desired property.
        \item It is immediate from the definition of $Y$ and the assumptions of the lemma that we have $|e|=2$ and $|S_e|=p$ for each $e \in Y$.
        \item Immediate from the degree condition in the lemma.
        \item Immediate from the codegree condition in the lemma, using the assumption that edges in $\mathcal{B}$ are subsets of $W'$.
        \item Immediate from the assumptions of the lemma.
        \item Immediate from the assumptions of the lemma.
        \item Let $u,v$ be distinct vertices in $W$.
    We let $H_{u,v}$ be the $p$-uniform multihypergraph on vertex set $W'\setminus W$, where we add the hyperedge $S_e$ for every $e \in Y$ with vertex set $\{u,v\}$.
    Note that $e_g(u,v)=|E(H_{u,v})|$ and that the multiplicity of every hyperedge in $H_{u,v}$ is at most $2\lambda$ by the previous item.
    Note that $\pi$ induces a uniformly random colouring of the vertex set of $H_{u,v}$ with colours $\{3,4,\dots,p+2\}$, and that $B(u,v)$ corresponds to the number of hyperedges in $H_{u,v}$ in which the sum of the colours of the vertices is congruent to $p(p+1)/2$ modulo $p$. Ignoring isolated vertices, the number of vertices in $H_{u,v}$ is $O(m)$.
    Thus, by applying \Cref{lem:Concentration-Modulos-Multiplicity} for $H_{u,v}$ with $a=p(p+1)/2$, $b=p$, $k=p$ and $C_1=13$, there exists a constant $C=C(p,b,C_1) > 0$ such that, with probability at least $1-\Omega(m^{-13})$, we have 
    \begin{align*}
        |B(u,v)-\mathbb{E}[B(u,v)]| \leq C (\log O(m))^{p/2} \lambda^{1/2}e_g(u,v)^{1/2}\leq C (\log m)^{\kappa/2} \lambda^{1/2}e_g(u,v)^{1/2}.
    \end{align*}
    Taking a union bound over all distinct $u,v \in W$ (for which there are $O(m^2)$ choices, ignoring isolated vertices in $H$) gives the desired result.
     \end{enumerate}
     This finishes the proof.
\end{proof}

We are now ready to prove Theorem \ref{thm:Hypergraph-Excess-r-large}.

\begin{proof}[Proof of Theorem \ref{thm:Hypergraph-Excess-r-large}]
    Let $\eps>0$, let $q$ be a positive integer and let $k-q\leq r\leq k$, where $k$ is sufficiently large as a function of $\eps$ and $q$. Let $H$ be a $k$-uniform hypergraph with $m$ edges.
    By \Cref{lem:Dichotomy-Set-U}, $H$ has an $r$-cut of surplus $\Omega(m^{2/3-\eps})$ unless there exists $U \subseteq V(H)$ such that
    \begin{itemize}
    \item $H[U]$ has $\Omega(m)$ edges,
    \item for every distinct $u,v \in U$, $\deg_H(v) \leq \Delta = \Theta(m^{2/3-\eps})$ and $\deg_H(u,v) \leq D = O(m^{1/3-2\eps})$. 
    \end{itemize}
    We now let $p$ be the largest prime such that $p \leq r-3$.
    Note that $p \geq (r-3)/2 \geq k/4$ by Bertrand's postulate.
    
    Our strategy for the rest of the proof is relatively simple, although the analysis itself is rather technical.
    We first take a uniform random assignment of colours $\{ 1 , \dots, r\}$ for $V(H)$ and reveal each colour class $i$ with $i \geq p+3$.
    By \Cref{lem:Reduction-r-cut-p-cut} this reduces the problem to finding a $(p+2)$-cut of large surplus (in expectation) in a naturally defined auxiliary hypergraph. 
    In this auxiliary hypergraph, we further reveal the colour class of every vertex either outside of $U$, or belonging to a colour class $i \geq 3$.
    By \Cref{lem:Reduction-Key-r-large} this further reduces the problem to finding a $2$-cut of large surplus in another naturally defined auxiliary hypergraph.
    With a suitable choice of the set $\mathcal{B}$, \Cref{lem:Check-Random-General-Type} shows that this hypergraph is a random multihypergraph of general type as defined in \Cref{defn:random general}. 
    Hence, we can use \Cref{thm:Key-Random-General} to prove a suitable lower bound on the expected surplus of this auxiliary hypergraph (of course, the desired $2$-cut in this hypergraph is not uniformly random, so here we deviate from the random assignment we considered at the beginning of this paragraph).
    We now proceed with the detailed proof.
    
    We take a random $\eta:V(H)\rightarrow \{p+3,\dots,r,*\}$, where, for each $p+3\leq i\leq r$, $\eta(v)=i$ with probability $1/r$, and $\eta(v)=*$ with probability $\frac{p+2}{r}$, independently for all $v\in V(H)$.
    Let $H_{\eta}$ be the hypergraph on vertex set $\eta^{-1}(*)$ and edge set
    $\{ e \cap \eta^{-1}(*) : e \in E(H), \forall j \in [p+3,r], |e \cap \eta^{-1}(j)| \geq 1 \}$.
    Then, by \Cref{lem:Reduction-r-cut-p-cut}, we have that
    \begin{align}
    \label{eq:Comparison-r-1-surplus}
        \surp_r(H) \geq \mathbb{E}_{\eta}[\surp_{p+2}(H_{\eta})].
    \end{align} 
    Our goal is now to give a suitable lower bound for $\mathbb{E}_{\eta}[\surp_{p+2}(H_{\eta})]$. \\
    
    We let $V_{\eta}=\eta^{-1}(*)$, $U_{\eta} = U \cap V_{\eta}$ and $U^c_{\eta}=U^c \cap V_{\eta}$.
    Let $F$ be the hypergraph consisting of the edges of $H_{\eta}[U_{\eta}]$ of size exactly $p+2$, and note that (as $H[U]$ has $\Omega(m)$ edges) $f=e(F)$ satisfies $\mathbb{E}[f]\geq \gamma m$ for some constant $\gamma=\gamma(k) > 0$.
    Applying \Cref{cor:Set-Gain-Bound-codegrees} to $F$, we obtain that there exist a set $B \subseteq E(F)$ of hyperedges such that $|B| = \Omega \left(\frac{f}{\log f}\right)$, positive integers $\lambda, \Lambda \geq 1$ with $\Lambda = O(m^{\frac{1}{\lfloor (p+2)/2\rfloor}})=O(m^{\frac{1}{\lfloor k/4\rfloor}})$ and a subset $S(e)\subset e$ of size $p$ for every $e\in B$ such that
    \begin{itemize}
        \item the multiplicity of every edge $e \in B$ is at least $\lambda/2$ and at most $\lambda$,
        \item for every edge $e\in B$, the number of edges $f \in B$ with $S(e)=S(f)$ is at most $\lambda \Lambda$.
    \end{itemize}
     Let $\omega:V_{\eta}\rightarrow [p+2]$ be a uniform random $(p+2)$-cut of $V_{\eta}$, and we let $P\omega$ be the function from $V_{\eta}$ to $\{ *,1, \dots, p+2\}$ such that
    \begin{align*}
        P\omega(v)= \begin{cases}
            \omega(v) & \text{if } v \in U_{\eta}^c \text{ or }  \omega(v) \in \{3, \dots, p+2\}, \\
           * & \text{otherwise}.
        \end{cases}
    \end{align*}
    Let $H_{\eta,P\omega}$ be the multihypergraph on vertex set $(P\omega)^{-1}(*)$ and edge set $R:=\{ e \cap (P\omega)^{-1}(*): e \in E(H_{\eta}), \{3, \dots, p+2\} \subseteq P\omega(e) , \{1, \dots, p+2\} \not \subset  P\omega(e) \} $, where two copies of $e\cap (P\omega)^{-1}(*)$ are added to $R$ if $P\omega(e) \cap \{1, \dots, p+2\} =  \{3, \dots, p+2\}$.
    
    Then, by \Cref{lem:Reduction-Key-r-large}, we have that
    \begin{align}
    \label{eq:Comparison-r-1-surplus-2}
    \surp_{p+2}(H_{\eta}) \geq \mathbb{E}_{\omega}[\surp(H_{\eta,P\omega})]/2
    \end{align}
    Therefore, our goal now is to prove a suitable lower bound for $\mathbb{E}_{\eta,\omega}[\surp(H_{\eta,P\omega})]$. 
    Let $W=(P\omega)^{-1}(*)$, let $\mathcal{B}=\{ e \in B : e \cap W=e \setminus S(e) \}$, and note that $\mathbb{E}_{\omega}[|\mathcal{B}|] \geq \delta |B|$ for some constant $\delta=\delta(k) > 0$.
    We now condition on $W=(P\omega)^{-1}(*)$, and let $\sigma$ have the same distribution as $P\omega$ conditioned on $W=(P\omega)^{-1}(*)$. 
    Note that the distribution of $\sigma$ is simply that if $v \in W$, then $\sigma(v)=*$, and for every vertex $v \in W^c$, if $v \in U_{\eta}$, then $\sigma(v)$ is uniform random in $\{ 3, \dots, p+2\}$, and if $v \in U^c_{\eta}$, then $\sigma(v)$ is uniform random in $\{ 1, \dots, p+2\}$, independently for each vertex $v \in W^c$.

    Note that every edge of $H_{\eta}$ is the representative of an edge of $H$, from which we have removed at least $r-p-2$ vertices (vertices mapped to $p+3,\dots,r$ by $\eta$), so that every edge of $H_{\eta}$ has size at most $k-r+p+2 \leq p+q+2$.
    If $f \geq \gamma m/10$ and $|\mathcal{B}| \geq \delta |B|/10$, then $|\mathcal{B}|=\Omega(m/\log m)$, and therefore by \Cref{lem:Check-Random-General-Type} (applied with $H=H_{\eta}$ and $W'=U_{\eta}$), conditional on some $\eta$ with $f\geq \gamma m/10$ and some $W$ with $\mathcal{B}\geq \delta |B|/10$, we have that $H_{\eta,\sigma}$ is a random $(m,2\Delta,2D,\Lambda,2\lambda,\kappa)$-multihypergraph of general type for $\kappa=p+q+2$.
    We now apply \Cref{thm:Key-Random-General} and get that
    \begin{align*}
            \mathbb{E_{\eta,\omega}}[\surp(H_{\eta, P\omega}) \hspace{1mm} \big| \hspace{1mm} |\mathcal{B}| \geq \delta |B|/10, f \geq \gamma m/10 ] = \Omega\left(\frac{m \lambda}{\beta \log m}\right) - O\left( \frac{m\Lambda \lambda D}{\beta^2}+\frac{m D^3 \lambda}{\beta^4}\right),
    \end{align*}
    where $\beta=(\log m)^{(\kappa+1)/2}\sqrt{4\Delta \lambda}$.
    Recalling that $\Delta=\Theta(m^{2/3-\eps})$, $D=O(m^{1/3-2\eps})$ and $\Lambda=O(m^{\frac{1}{\lfloor k/4\rfloor}})$, we get that
    \begin{align*}
        \frac{m\Lambda \lambda D}{\beta^2}+\frac{m D^3 \lambda}{\beta^4} &= \frac{m \lambda}{\beta \log m}\left( \frac{ \Lambda D \log m}{\beta} + \frac{ D^3 \log m}{\beta^3}\right) \\
        &=\frac{m \lambda}{\beta \log m} \cdot O\left( \frac{\Lambda D}{\sqrt{\Delta}}+ \frac{D^3}{\Delta^{3/2}} \right) \\
        &=\frac{m \lambda}{\beta \log m} \cdot O\left(m^{\frac{1}{\lfloor k/4\rfloor}-3\eps/2}+ m^{-9\eps/2}  \right).
    \end{align*}
    For $k$ sufficiently large, we have $m^{\frac{1}{\lfloor k/4\rfloor}-3\eps/2}+ m^{-9\eps/2}=o(1)$, which shows that
    \begin{align*}
        \frac{m\Lambda \lambda D}{\beta^2}+\frac{m D^3 \lambda}{\beta^4} = o\left(\frac{m \lambda}{\beta \log m}\right),
    \end{align*}
    and therefore
    \begin{align*}
        \mathbb{E_{\eta,\omega}}[\surp(H_{\eta, P\omega}) \hspace{1mm} \big| \hspace{1mm} |\mathcal{B}| \geq \delta |B|/10, f \geq \gamma m/10 ] = \Omega\left(\frac{m \lambda}{\beta \log m}\right) = \Omega(m^{2/3}).
    \end{align*}
    This implies that
    \begin{align*}
        \mathbb{E_{\eta,\omega}}[\surp(H_{\eta, P\omega})] =\Omega(  \mathbb{P}(|\mathcal{B}| \geq \delta |B|/10, f \geq \gamma m/10 ) \cdot m^{2/3}) = \Omega(m^{2/3}).
    \end{align*}
    Using \eqref{eq:Comparison-r-1-surplus} and \eqref{eq:Comparison-r-1-surplus-2}, we obtain $\surp_r(H)=\Omega(m^{2/3})$.
\end{proof}

\subsection{The proof of Lemma \ref{lem:Correlation-Modulo-Constraint}} \label{sec:modulo lemma}

We first recall the following classical estimate on Stirling numbers of the second kind (see for instance \cite{hsu1948note}).

\begin{lemma}
\label{lem:Estimate-Stirling}
    Let $t$ be a fixed nonnegative integer. Then, as $r$ tends to infinity we have
    \begin{align*}
        S(r+t,r) \sim \frac{r^{2t}}{2^t t!}.
    \end{align*}
\end{lemma}

We will also need the following lemma. Here and in the rest of this section, we will think of~$t$ as a constant, and so the $O(\cdot)$ terms are allowed to hide a dependence on $t$.

\begin{lemma}
\label{lem:Proba-f-normal}
    Fix a nonnegative integer $t$. Let $f : [r+t] \rightarrow [r]$ be a uniformly random surjective function.
    Then the probability that there exists some $i \in [r]$ such that $|f^{-1}(i)| \geq 3$ is $O(r^{-1})$.
\end{lemma}

\begin{proof}
    By a union bound it suffices to show that the probability that $|f^{-1}(1)| \geq 3$ is $O(r^{-2})$.
    By a simple counting argument, and applying \Cref{lem:Estimate-Stirling}, we have 
    \begin{align*}
        \mathbb{P}(|f^{-1}(1)|=s)=\frac{\binom{r+t}{s}(r-1)!S(r+t-s,r-1)}{r!S(r+t,r)} = O\left(\frac{r^s r^{2(t-s+1)}}{r r^{2t}}\right) = O\left( r^{1-s}\right).
    \end{align*}
    Summing over $s \geq 3$ finishes the proof.
\end{proof}

Given a prime $r$, and integers $b_1, \dots, b_{\ell}$, we denote by $N_{\mathbb{F}_r}(b_1, \dots, b_{\ell};0)$ the number of solutions of the equation $\sum_{i=1}^{\ell} b_ix_i \equiv 0$ modulo $r$, where all the $x_i$ for $i \in [\ell]$ are distinct elements of $\{0,1,\dots,r-1\}$.
The next result is a consequence of Theorem 1.5 in \cite{li2020distinct}, and gives us a formula for $N_{\mathbb{F}_r}(b_1, \dots, b_{\ell};0)$. Here and below we use $(r)_{\ell}$ to denote $\ell!\binom{r}{\ell}$.

\begin{theorem}
\label{thm:Counting-Distinct-Solutions}
    Let $r$ be prime, and let $b_1, \dots, b_{\ell}$ be integers.
    Let $\zeta(b_1, \dots, b_{\ell};\ell,i)$ be the number of permutations in $S_{\ell}$ composed of exactly $i$ cycles, such that for every such cycle $\mathcal{C}$, we have $\sum_{j \in \mathcal{C}} b_j \equiv 0$ modulo $r$. Then
    \begin{align*}
        N_{\mathbb{F}_r}(b_1, \dots, b_{\ell};0)= \frac{(r)_{\ell}}{r}+ \frac{r-1}{r}\sum_{i=1}^{\ell} (-1)^{\ell-i} \zeta(b_1, \dots, b_{\ell};\ell,i) r^i.
    \end{align*}
\end{theorem}

We now introduce some definitions.
Fix a nonnegative integer $t$ and a (large) prime $r$ for the rest of this section.
The \emph{profile} of a surjective function $f : [r+t] \rightarrow [r]$ is the multiset $P_f=\{ a_1, \dots, a_r\}$, where $a_i=|f^{-1} (i) \cap [r]|-1$ for every $i \in [r]$.
We say that a multiset $P$ is a profile if there exists a surjective function $f : [r+t] \rightarrow [r]$ such that $P=P_f$.
For a profile $P=\{a_1, \dots, a_r\}$, we let $p(P)$ be the probability that $\sum_{i=1}^r a_i \lambda_i=0$ in $\mathbb{F}_r$ for $\lambda=(\lambda_1, \dots, \lambda_r)$ a uniform random permutation of $(0,1, \dots, r-1)$.

The reason we are interested in $p(P)$ is the following simple observation.

\begin{lemma} \label{lem:profile p and sum}
    For any profile $P$, $p(P)$ is the probability that $\sum_{k=1}^r f(k)$ is congruent to $r(r+1)/2$ modulo $r$ for a uniformly random surjective function $f : [r+t] \rightarrow [r]$, conditionally on $P_f=P$.
\end{lemma}

\begin{proof}
     Let $P=\{b_1,\dots,b_r\}$. Conditional on $P_f=P$, the sum $\sum_{k=1}^r f(k)=\sum_{i=1}^r |f^{-1}(i)\cap [r]|\cdot i$ is distributed as $\sum_{i=1}^r (b_i+1)\lambda_i$, where $(\lambda_1,\dots,\lambda_r)$ is a uniform random permutation of $(0,1,\dots,r-1)$, since the tuple $(|f^{-1}(1)\cap [r]|-1,\dots,|f^{-1}(r)\cap [r]|-1)$ is a uniform random permutation of $(b_1,\dots,b_r)$. But $\sum_{i=1}^r (b_i+1)\lambda_i \equiv r(r+1)/2+\sum_{i=1}^r b_i \lambda_i$ modulo~$r$, from which the result follows.
\end{proof}

We will estimate $p(P)$ using the following result.

\begin{lemma}
\label{lem:Relate-Profile-Counting-Solutions}
    Let $P=\{b_1, \dots, b_r\}$ be a profile, and let $1\leq \ell \leq r$ be such that $b_i=0$ for all $i>\ell$.
    Then we have
    \begin{align*}
        p(P) = \frac{1}{r}+ \frac{r-1}{r\cdot (r)_{\ell}}\sum_{i=1}^{\ell} (-1)^{\ell-i} \zeta(b_1, \dots, b_{\ell};\ell,i) r^i.
    \end{align*}
\end{lemma}

\begin{proof}
    Take a uniform random permutation $\lambda=(\lambda_1, \dots, \lambda_r)$ of $(0,1, \dots, r-1)$.
    Then $p(P)$ is the probability that $\sum_{i=1}^{\ell} b_i \lambda_i \equiv 0$ modulo $r$.
    The number of ordered tuples $(x_1, \dots, x_{\ell})$ consisting of distinct elements of $\{0,1, \dots, r-1\}$ such that $\sum_{i=1}^{\ell} b_i x_i \equiv 0$ modulo $r$ is equal to $N_{\mathbb{F}_r}(b_1, \dots, b_{\ell};0)$ by definition.
    The number of ways to extend such a fixed tuple of distinct elements $(x_1, \dots, x_{\ell})$ into a permutation $(x_1,\dots,x_r)$ of  $(0,1, \dots, r-1)$ is $(r-\ell)!$.
    Thus we have that
    \begin{align*}
        p(P)= \frac{N_{\mathbb{F}_r}(b_1, \dots, b_{\ell};0)(r-\ell)!}{r!} = \frac{N_{\mathbb{F}_r}(b_1, \dots, b_{\ell};0)}{(r)_{\ell}}.
    \end{align*}
    Applying \Cref{thm:Counting-Distinct-Solutions}, we get
    \begin{align*}
        p(P) = \frac{1}{r}+ \frac{r-1}{r\cdot (r)_{\ell}}\sum_{i=1}^{\ell} (-1)^{\ell-i} \zeta(b_1, \dots, b_{\ell};\ell,i) r^i,
    \end{align*}
    as desired.
\end{proof}

For $0 \leq j \leq t$, we let $P^j$ be the profile that has exactly $j$ elements equal to $1$ and $j$ elements equal to $-1$, and the rest of the elements are $0$.
Our next result gives an asymptotic formula for $p(P^j)$.

\begin{lemma}
\label{lem:Proba-Modulo-Good-Profile}
    For every $0 \leq j \leq t$, we have
    \begin{align*}
        p(P^j)-\frac{1}{r}= j! (-1)^j r^{-j} + O(r^{-j-1}).
    \end{align*}
\end{lemma}

\begin{proof}
    We have $p(P^0)=1$, so the statement is true for $j=0$. Assume now that $j\geq 1$. We apply \Cref{lem:Relate-Profile-Counting-Solutions} with $\ell=2j$ and get that
    \begin{align*}
        p(P^j)-\frac{1}{r}= \frac{r-1}{r \cdot(r)_{2j}}\sum_{i=1}^{2j} (-1)^{2j-i} \zeta(b_1, \dots, b_{2j};2j,i) r^i,
    \end{align*}
    where $b_i=1$ for $i\in [j]$ and $b_i=-1$ for $i\in [j+1,2j]$.
    Now note that $\zeta(b_1, \dots, b_{2j};2j,i) =0$ for $i \geq j+1$ (as otherwise one of the cycles would consist of exactly one non-zero element), and that $\zeta(b_1, \dots, b_{2j};2j,j) = j!$ (since each cycle in the permutation must consist of exactly one element $b_i$ for $i \in [j]$ and exactly one element $b_i$ for $i \in [j+1, 2j]$).
    Therefore, we obtain 
    \begin{align*}
        p(P^j)-\frac{1}{r}= j! (-1)^j r^{-j} + O(r^{-j-1}),
    \end{align*}
    as wanted.
\end{proof}

We now prove an upper bound on $|p(P)-\frac{1}{r}|$ for any profile $P$. 

\begin{lemma}
\label{lem:Proba-Modulo-Bad-Profile}
    For any profile $P$ with $j$ negative elements, we have
    \begin{align*}
        \left|p(P)-\frac{1}{r}\right|= O(r^{-j}).
    \end{align*}
\end{lemma}

\begin{proof}
    Let $P=\{b_1,\dots,b_r\}$, where $b_i<0$ for $i\in [j]$, $b_i>0$ for $i\in [j+1,j+y]$ and $b_i=0$ for $i>j+y$.
    By \Cref{lem:Relate-Profile-Counting-Solutions} with $\ell=j+y$, we have
\begin{align*}
     \left|p(P)-\frac{1}{r}\right| = O(r^{x-(j+y)}),
\end{align*}
where $x$ is the maximal number of sets that partition $\{b_1,b_2,\dots,b_{j+y}\}$, such that for every such set $\mathcal{C}$ in the partition, we have $\sum_{b_i \in \mathcal{C} } b_i \equiv 0$ modulo $r$. Since $r$ is sufficiently large compared to $t$, we must have $\sum_{b_i \in \mathcal{C} } b_i=0$ for all such $\mathcal{C}$.
Hence, every such $\mathcal{C}$ in the partition must contain at least one positive element, and therefore we have $x \leq y$, which implies that
\begin{align*}
     \left|p(P)-\frac{1}{r}\right| = O(r^{-j}),
\end{align*}
as wanted.
\end{proof}

We slightly abuse notation and for a function $g :[r] \rightarrow [r] $, we will say that the profile of $g$ is the multiset $\{ a_1, \dots, a_r \}$ where $a_i=|g^{-1} (i)|-1$ for every $i \in [r]$.
For two nonnegative integers $m,n$, we let $\Surj_{m,n}$ be the number of surjections from $[m]$ to $[n]$.
We now prove the following result. 

\begin{lemma}
\label{lem:Proba-Good-Profile}
    Fix $0 \leq j \leq t$.
    The probability that a uniform random surjection $f : [r+t] \rightarrow [r]$ has profile $P^j$ is
    \begin{align*}
        (1+o(1))\binom{t}{j}2^{t-j} \frac{t!}{j!}r^{j-t}
    \end{align*}
    as $r\rightarrow \infty$.
\end{lemma}

\begin{proof}
    We first count the number of functions $g : [r] \rightarrow [r]$ with profile $P^j$. 
    We claim that this is $\binom{r}{j} \Surj_{r,r-j} (1-O(r^{-1}))$. Indeed we first choose $Z:=g([r])$ to be a subset of $[r]$ of size $r-j$, for which we have $\binom{r}{j}$ possibilities. 
    Then we choose the map $g' : [r] \rightarrow Z$ such that $g'$ is surjective and there is no $i \in Z$ with $|g'^{-1}(i)| \geq 3$. 
    By \Cref{lem:Proba-f-normal}, the number of possibilities is $ \Surj_{r,r-j} (1-O(r^{-1}))$. \\

    We now count the number of ways we can extend such a function $g$ to a surjective function $f : [r+t] \rightarrow [r]$.
    There are exactly $j$ elements in $[r]$ which do not have an image through $g$. Let $X=\{ x_1, \dots, x_j \}$ be the set of such elements.
    Therefore, we are counting the number of functions $h : [t] \rightarrow [r] $ such that $X \subseteq h([t])$.
    We first choose the size of $h^{-1}(X)$, which we denote by $\ell$.
    Therefore we have that the number of such functions $h$ is
    \begin{align*}
        \sum_{\ell=j}^t \binom{t}{\ell} \Surj_{\ell,j} (r-j)^{t-\ell}=\binom{t}{j} j! r^{t-j} + O(r^{t-j-1}),
    \end{align*}
    as we have $\binom{t}{\ell}$ choices for $H:= h^{-1}(X)$, $\Surj_{\ell,j}$ choices for the map $h|_H : H\rightarrow X$, and $r-j$ choices for $h(a)$ for each $a \notin H$. \\
    
    Putting the above two paragraphs together, we find that the number of surjections $f : [r+t] \rightarrow [r]$ that have profile $P^j$ is
    \begin{align*}
        \binom{r}{j} \Surj_{r,r-j} \binom{t}{j} j! r^{t-j} (1+o(1)).
    \end{align*}
    Dividing by the number of surjections from $[r+t]$ to $[r]$, we obtain by \Cref{lem:Estimate-Stirling} that the desired probability is 
    \begin{align*}
        (1+o(1))\frac{\binom{r}{j} \Surj_{r,r-j} \binom{t}{j} j! r^{t-j}}{\Surj_{r+t,r}} &= (1+o(1))\frac{\binom{r}{j} (r-j)! S(r,r-j) \binom{t}{j} j! r^{t-j}}{r! S(r+t,r)} \\
        &= (1+o(1)) \frac{ \frac{r^{2j}}{2^j j!} \binom{t}{j} r^{t-j}}{\frac{r^{2t}}{2^t t!}} \\
        &= (1+o(1)) \binom{t}{j} 2^{t-j} \frac{t!}{j!} r^{j-t},
    \end{align*}
    as desired.
\end{proof}

We say that a profile is \emph{good} if it contains exactly $j$ elements $1$ and $j$ elements $-1$, and the rest of the elements are $0$, for some integer $j$, and \emph{bad} otherwise.
In other words, the only profiles which are not bad are the profiles $P_j$ for $0 \leq j \leq t$.
The following lemma shows that profiles with exactly $j$ negative elements that are not $P_j$ are unlikely to occur.

\begin{lemma}
\label{lem:Proba-Bad-Profile}
    Fix $0 \leq j \leq t$, and let $P$ be a bad profile with exactly $j$ negative elements.
    The probability that a uniform random surjection $f : [r+t] \rightarrow [r]$ has profile $P$ is $O(r^{j-t-1})$.
\end{lemma}

\begin{proof}
    The proof is similar to the proof of \Cref{lem:Proba-Good-Profile}.
    We first claim that there are at most $O(\binom{r}{j}\Surj_{r,r-j} r^{-1})$ functions $g : [r] \rightarrow [r]$ with profile $P$.
    Indeed we first choose $Z:=g([r])$ to be a subset of $[r]$ of size $r-j$, for which we have $\binom{r}{j}$ possibilities. 
    Then choose the map $g' : [r] \rightarrow Z$, and note that $g'$ is surjective and there is some $i \in Z$ such that $|g'^{-1}(i)| \geq 3$ because $P$ is a bad profile. 
    By \Cref{lem:Proba-f-normal}, the number of possibilities is $ O(\Surj_{r,r-j} r^{-1})$, thus giving in total $O(\binom{r}{j}\Surj_{r,r-j} r^{-1})$ choices for the function $g : [r] \rightarrow [r]$. \\
    
    Then the second part of the proof is completely identical to the one of \Cref{lem:Proba-Good-Profile}, giving the desired bound.
\end{proof}

We are now ready to prove Lemma \ref{lem:Correlation-Modulo-Constraint}, which we recall for the reader's convenience.

\modulolemma*

\begin{proof}
    By Lemma \ref{lem:profile p and sum}, the probability of interest can be written as
    \begin{align*}
        \sum_P \mathbb{P}(P_f=P)p(P),
    \end{align*}
    where $f$ is a uniform random surjection $f : [r+t] \rightarrow [r]$, and the sum is over every possible profile $P$.
    Therefore it is equivalent to show that
    \begin{align*}
        \sum_P \mathbb{P}(P_f=P)\left(p(P)-\frac{1}{r}\right) > 0.
    \end{align*}
    We now split the sum of interest into two:
    \begin{align}
    \label{eq:Split-Profile-good-bad}
        \sum_P \mathbb{P}(P_f=P)\left(p(P)-\frac{1}{r}\right) = \sum_{P \text{ good }} \mathbb{P}(P_f=P)\left(p(P)-\frac{1}{r}\right)+ \sum_{P \text{ bad }} \mathbb{P}(P_f=P)\left(p(P)-\frac{1}{r}\right) .
    \end{align}
    We will show that the first term gives a main positive contribution, while the second term is of lower order.
    By \Cref{lem:Proba-Good-Profile,lem:Proba-Modulo-Good-Profile}, we have 
    \begin{align}
    \label{eq:Main-term-Modular-Lemma}
        \sum_{P \text{ good }} \mathbb{P}(P_f=P)\left(p(P)-\frac{1}{r}\right) &= \sum_{j=0}^t \mathbb{P}(P_f=P^j)\left(p(P^j)-\frac{1}{r}\right)  \nonumber \\
        &= \sum_{j=0}^t (1+o(1))\binom{t}{j}2^{t-j} \frac{ t!}{j!}r^{j-t} \cdot j! (-1)^j r^{-j} \nonumber \\
        &= (1+o(1))t! r^{-t}\sum_{j=0}^t \binom{t}{j}2^{t-j}  (-1)^j \nonumber \\
        &= (1+o(1)) t! r^{-t}.
    \end{align}
    We now estimate the second term.
    We let $P^j_{bad}$ be the set of bad profiles with exactly $j$ negative elements.
    By \Cref{lem:Proba-Bad-Profile,lem:Proba-Modulo-Bad-Profile}, we have
    \begin{align*}
        \sum_{P \text{ bad }} \mathbb{P}(P_f=P)\left|p(P)-\frac{1}{r}\right| &= \sum_{j=0}^{t} \sum_{P \in P^j_{bad}} \mathbb{P}(P_f=P)\left|p(P)-\frac{1}{r}\right| \\
        &= O\left(\sum_{j=0}^{t} \sum_{P \in P^j_{bad}} r^{j-t-1} \cdot r^{-j} \right) 
    \end{align*}
    Since the number of profiles is bounded as a function of $t$ (a profile has at most $2t$ non-zero elements, each of which is at most $t$ in absolute value), we obtain
    \begin{align}
    \label{eq:Error-term-Modular-Lemma}
        \sum_{P \text{ bad }} \mathbb{P}(P_f=P)\left|p(P)-\frac{1}{r}\right| = O(r^{-t-1}).
    \end{align}
    Plugging \eqref{eq:Main-term-Modular-Lemma} and \eqref{eq:Error-term-Modular-Lemma} into \eqref{eq:Split-Profile-good-bad}, we get 
        \begin{align*}
        \sum_P \mathbb{P}(P_f=P)\left(p(P)-\frac{1}{r}\right) = (1+o(1))t!r^{-t} > 0,
    \end{align*}
    which finishes the proof.
\end{proof}

\section{A new construction for $2$-cut in $4$-graphs} \label{sec:(4,2)-construction}

In this short section we prove Proposition \ref{prop:Alternative-Construction-4-2}. We recall the statement for the reader's convenience.

\blockconstruction*

\begin{proof}
    Clearly, $e(G)=\Theta(ts^4)$. Since $t=\Theta(n^2/s^2)$, we get $e(G)=\Theta(n^2s^2)$.

    We now turn to upper bounding $\surp(G)$. Consider a cut $(A,B)$ of $V$. There exist $x_1,\dots,x_t$ such that $|A\cap S_i|=s/2+x_i$ and $|B\cap S_i|=s/2-x_i$.

    Now the surplus of this cut is
    \begin{align}
        \surp_G(A,B)=&\sum_{i=1}^t \left(\frac{1}{8}\binom{s}{4}-\binom{s/2+x_i}{4}-\binom{s/2-x_i}{4}\right) \nonumber \\
        &=\sum_{i=1}^t \left(\frac{1}{32}s^3-\frac{1}{8}s^2x_i^2-\frac{11}{64}s^2+\frac{3}{4}sx_i^2+\frac{7}{32}s-\frac{1}{12}x_i^4-\frac{11}{12}x_i^2\right) \nonumber \\
        &\leq \sum_{i=1}^t \left(\frac{1}{32}s^3-\frac{11}{64}s^2+\frac{7}{32}s-\left(\frac{1}{8}s^2-\frac{3}{4}s+\frac{11}{12}\right)x_i^2\right). \label{eqn:expanded}
    \end{align}

    On the other hand,  in $S_i$, the difference between the number of pairs of vertices cut by $(A,B)$ and the number of pairs of vertices not cut by $(A,B)$ is
    $$\binom{s}{2}-2\left(\binom{s/2+x_i}{2}+\binom{s/2-x_i}{2}\right)=s/2-2x_i^2.$$
    In $V$, the number of pairs cut by $(A,B)$ is at most an $O(n)$ additive term greater than the number of pairs not cut by $(A,B)$. Since each pair of vertices in $V$ belongs to precisely one block $S_i$, we obtain
    $$\sum_{i=1}^t (s/2-2x_i^2)\leq O(n),$$
    and hence
    $$\sum_{i=1}^t -x_i^2\leq O(n)-\sum_{i=1}^t s/4.$$
    Plugging this into (\ref{eqn:expanded}) and using $\frac{1}{8}s^2-\frac{3}{4}s+\frac{11}{12}\geq 0$ (which holds for $s\geq 5$; note that the proposition is trivial for $s=4$), we get
    $$\surp_G(A,B)\leq O(s^2n)+\sum_{i=1}^t \left(\frac{1}{32}s^3-\frac{11}{64}s^2+\frac{7}{32}s-\left(\frac{1}{8}s^2-\frac{3}{4}s+\frac{11}{12}\right)\frac{s}{4}\right)\leq O(s^2n)+O(ts^2).$$
    Using $t=\Theta(n^2/s^2)$, we conclude that $\surp_G(A,B)\leq O(s^2n)+O(n^2)$, as desired.
\end{proof}

\section{Concluding remarks} \label{sec:concluding remarks}

\subsection*{MaxCut for hypergraphs}

In this paper we proved that a $k$-uniform hypergraph with $m$ edges has an $r$-cut of surplus $\Omega(m^{2/3-\eps_k})$, where $\eps_k\rightarrow 0$ as $k\rightarrow \infty$. The way our proof is written\footnote{Our proof gives a linear dependence in the most important cases $r=2$ and $r=k$. In the general case, to get a polynomial bound, one can quantify the dependence of the implicit constant on $t$ in Section \ref{sec:modulo lemma}.} gives an $\eps_k$ that depends polynomially on $k$. On the other hand, we believe that our methods can be pushed further to give an $\eps_k$ that is exponentially small in $k$. We omit the details.

Our method can also be used (with suitable modifications) to give polynomially improved lower bounds for all values of $(r,k)$ with the sole exception of $(r,k)=(3,3)$ (where we can match the best known bound). For $r=k$, our proof gives the asymptotically tight $m^{2/3}/\textrm{polylog}(m)$ lower bound under the very mild assumption on the hypergraph that no two hyperedges intersect in at least $k-1$ vertices. The same holds for $r=2$ if $k$ is even. When $r=2$ and $k$ is odd, we get the asymptotically tight $m^{2/3}/\textrm{polylog}(m)$ lower bound under the assumption that no two hyperedges intersect in at least $k-2$ vertices.

\subsection*{Bisection width and positive discrepancy}

An equipartition of a finite set is a partition into two parts, whose sizes differ by at most one. A \emph{bisection} of a hypergraph is an equipartition of its vertex set, together with all edges that intersect each part in at least one vertex. The \emph{bisection width} of a hypergraph $G$, denoted as $\textrm{bw}(G)$, is the smallest number of edges in a bisection of $G$. In other words, $\textrm{bw}(G)$ is the smallest $2$-cut of $G$ over all balanced bipartitions.

This parameter has been extensively studied in both Theoretical Computer Science (see, e.g. \cite{feige2002polylogarithmic}) and Extremal and Probabilistic Combinatorics. A celebrated result of Alon \cite{alon1997edge} states that if $G$ is an $n$-vertex $d$-regular graph with $d=O(n^{1/9})$, then $\textrm{bw}(G)\leq e(G)(\frac{1}{2}-\Omega(\frac{1}{\sqrt{d}}))$. While the bound here is tight (e.g. for random $d$-regular graphs, as proven by Bollob\'as \cite{bollobas1988isoperimetric}), the assumption on $d$ being small is not sharp. Indeed, recently the same bound $\textrm{bw}(G)\leq e(G)(\frac{1}{2}-\Omega(\frac{1}{\sqrt{d}}))$ was established by R\"aty, Sudakov and Tomon \cite{raty2023positive} under the weaker assumption $d=O(n^{2/3})$. They also showed that this bound no longer holds for $d=\omega(n^{2/3})$.

Recently, R\"aty and Tomon \cite{raty2024bisection} obtained a hypergraph analogue of Alon's result by showing that if $G$ is an $n$-vertex $d$-regular $k$-uniform hypergraph with $d=o(n^{1/2})$, then $\textrm{bw}(G)\leq e(G)(1-\frac{1}{2^{k-1}}-\Omega(\frac{1}{\sqrt{d}}))$. While this bound is tight for the range of parameters the result applies to, the assumption that $d$ is small is not optimal. In upcoming work \cite{bisectionwidth}, we adapt the methods we developed in this paper to prove the above (tight) upper bound for the bisection width of hypergraphs under a nearly optimal condition on $d$. Furthermore, for linear hypergraphs, we obtain the optimal condition without any error. Our results reveal a difference between the graph and the hypergraph case that was not demonstrated in previous works.

Finally, we obtain analogous improvements for the closely related notion of positive discrepancy. For details about this hypergraph parameter and its interesting connections to eigenvalues of the adjacency tensor, we refer the reader to \cite{raty2024bisection} and to our upcoming paper \cite{bisectionwidth}.

\appendix

\section{Proof of \Cref{lem:Reduction-r-cut-p-cut}}
\label{Appendix-lem:Reduction-r-cut-p-cut}

\begin{proof}[Proof of \Cref{lem:Reduction-r-cut-p-cut}]
    Let $\omega: V\rightarrow \{1,\dots,r\}$ be a uniform random $r$-cut of $V$.
    We define the ``information restriction'' operator $P$ such that for any $r$-cut $\omega : V \rightarrow \{ 1, \dots, r \}$, $P\omega$ is the function from $V$ to $\{ *, q+1,\dots,r\}$ such that 
    \begin{align*}
        P\omega(v)= \begin{cases}
            \omega(v) & \text{if } \omega(v) \in \{q+1,\dots, r\}, \\
           * & \text{otherwise.}
        \end{cases}
    \end{align*}
    Then it is clear that $P\omega$ has the same distribution as $\sigma$, and that a $q$-cut of $H_{\sigma}$ corresponds to an $r$-cut $\omega$ of $H$ of the same size, with $P\omega=\sigma$.
    Let $Z$ be the random variable associated to the size of the cut $\omega$ in $H$. 
    It is straightforward to see that $\mathbb{E}[Z | P{\omega}]$ is the average size of a uniform $q$-cut of $H_{P\omega}$.
    Therefore, by definition, $H_{P\omega}$ has a $q$-cut of size $\mathbb{E}[Z | P{\omega}]+\surp_{q}(H_{P\omega})$ and the result follows by taking expectation over $\omega$.
\end{proof}

\section{Proof of \Cref{lem:Reduction-Key-r-large}}
\label{Appendix-lem:Reduction-Key-r-large}

\begin{proof}[Proof of \Cref{lem:Reduction-Key-r-large}]
    Let $W$ be a subset of $U$ and consider any $\sigma : W^c \rightarrow \{1,\dots, r\}$ such that $\sigma(U \cap W^c) \subseteq \{3, \dots, r\}$. We treat $W$ and $\sigma$ as fixed for the moment, though ultimately we will let $W=(P\omega)^{-1}(*)$ and let $\sigma$ be the restriction of $P\omega$ to $W^c$. 
    With a slight abuse of notation, we let $H_{\sigma}$ be the multihypergraph on vertex set $W$ and edge set $\{ e \cap W: e \in E, \{3, \dots, r\} \subseteq \sigma(e \setminus W) , \{1, \dots, r\} \not \subset  \sigma(e \setminus W) \} $, where two copies of $e\cap W$ are added to $H_{\sigma}$ if $\sigma(e \setminus W) \cap \{1, \dots, r\} =  \{3, \dots, r\}$.
    We let $\tilde{H}_{\sigma} \subseteq H$ consist of the set of edges $e$ of $H$ such that $|\sigma(e \setminus W) \cap \{1,2\}|=1 $, $\{3, \dots, r\} \subseteq \sigma(e \setminus W)$ and $|e \cap W| > 0$. Let $H_{\sigma}^*$ be the multihypergraph on vertex set $W$ with edge set $\{ e \cap W: e \in E , \sigma(e \setminus W) \cap \{1, \dots, r\} =  \{3, \dots, r\} \} $.

    Let $\phi : W \rightarrow \{1,2\}$.
    Let $z_1$ be the size of $\phi$ as a $2$-cut of $H_{\sigma}^*$, and let $z_2$ be the number of edges in $\tilde{H}_{\sigma}$ that have both a vertex mapped to $1$ and a vertex mapped to $2$ by $\phi$.
    Then $z=2z_1+z_2$ is the size of $\phi$ as a $2$-cut of $H_{\sigma}$ since every edge $e\in H_{\sigma}^*$ was included in $H_{\sigma}$ with multiplicity two.
    Let $N_{\sigma}^{cut}$ be the number of edges $e \in E(H)$ such that $|\sigma(e \setminus W) | = r$, i.e. the number of edges of $H$ that can already be seen as cut by $\sigma$ alone. Define $\psi_{\mathrm{swap}}(\phi)$ to be the mapping $\psi_{\mathrm{swap}}(\phi) : W \rightarrow \{1,2\}$ satisfying $\psi_{\mathrm{swap}}(\phi) (a)=3-\phi(a)$ for every vertex $a \in W$.
    Let $\Phi : V \rightarrow \{1, \dots, r\}$ be the $r$-cut equal to $\phi$ on $W$ and to $\sigma$ on $W^c$, and let $\Phi' : V \rightarrow \{1,\dots, r\}$ be the $r$-cut equal to $\psi_{\mathrm{swap}}(\phi)$ on $W$ and to $\sigma$ on $W^c$.
    Note that for every edge $e \in E(\tilde{H}_{\sigma})$ whose corresponding edge in $H_{\sigma}$ is not already cut by $\phi$, the edge $e$ is cut by exactly one of $\Phi$ and $\Phi'$.
    Therefore, the average size of $\Phi$ and $\Phi'$ as $r$-cuts of $H$ is
    \begin{align}
    \label{eq:Averaging-2-cut-Reduction-2}
        z_1+ z_2+(e(\tilde{H}_{\sigma}) -z_2)/2 + N_{\sigma}^{cut} = z/2 + e(\tilde{H}_{\sigma})/2 + N_{\sigma}^{cut}.
    \end{align}
    (Indeed, the total contribution of the edges $e\in E(H), \{ 3, \dots, r\} \subseteq \sigma(e \setminus W)$ with $ |\sigma(e\setminus W) \cap \{1,2\}|=0$ is $z_1$, the total contribution of the edges with $|\sigma(e\setminus W) \cap \{1,2\}|=1$ is $z_2+(e(\tilde{H}_{\sigma}) -z_2)/2$, and the total contribution of edges with $|\sigma(e\setminus W) \cap \{1,2\}|=2$ is $N_{\sigma}^{cut}$.)

    Now let $\omega : V \rightarrow \{1,\dots, r\}$ be a uniformly random $r$-cut, as in the statement of the lemma. Let $Z$ be the random variable associated to the size of $\omega$ as an $r$-cut of $H$, and let $Z^{part}$ be the random variable associated to the size of $\omega$ restricted to $W:=(P\omega)^{-1}(*)$ as a $2$-cut of $H_{P\omega}$.
    Let $P\omega_{rest}$ be the restriction of $P\omega$ to $W^c$.
    It follows from \eqref{eq:Averaging-2-cut-Reduction-2} by taking conditional expectation on the outcome $P\omega_{rest}=\sigma$ that
    \begin{align}
    \label{eq:Averaging-2-cut-2-Reduction-2}
        \mathbb{E}_{\omega}[Z | P\omega_{rest}=\sigma]=\mathbb{E}_{\omega}[Z^{part} | P\omega_{rest}=\sigma]/2 + e(\tilde{H}_{\sigma})/2 + N_{\sigma}^{cut}.
    \end{align}
    It is clear that, conditional on $P\omega_{rest}=\sigma$, the restriction of $\omega$ to $W$ is a uniformly random $2$-cut of $H_{\sigma}$, and thus $\mathbb{E}_{\omega}[Z^{part} | P\omega_{rest}=\sigma]$ is the average size of a $2$-cut of $H_{\sigma}$.
    Hence, $H_{\sigma}$ has a $2$-cut of size $\surp(H_{\sigma})+\mathbb{E}_{\omega}[Z^{part} | P\omega_{rest}=\sigma]$, and so it follows by \eqref{eq:Averaging-2-cut-Reduction-2} that $H$ has an $r$-cut with surplus
    \begin{align*}
        (\surp(H_{\sigma})+\mathbb{E}_{\omega}[Z^{part} | P\omega_{rest}=\sigma])/2 + e(\tilde{H}_{\sigma})/2 + N_{\sigma}^{cut} - \mathbb{E}_{\omega}[Z],
    \end{align*}
    which, by \eqref{eq:Averaging-2-cut-2-Reduction-2}, is equal to
    \begin{align*}
        \surp(H_{\sigma})/2 + \mathbb{E}_{\omega}[Z | P\omega_{rest}=\sigma] - \mathbb{E}_{\omega}[Z].
    \end{align*}
    Thus, $H$ has an $r$-cut with surplus
    $$\surp(H_{P\omega})/2 + \mathbb{E}_{\omega}[Z | P\omega] - \mathbb{E}_{\omega}[Z].$$
    Taking expectation gives the desired result.
\end{proof}

\section{Proof of \Cref{lem:Reduction-Main-Small-r-Lem}}
\label{Appendix-Reduction-2-Cut}

\begin{proof}[Proof of \Cref{lem:Reduction-Main-Small-r-Lem}]
    Let $H=(V,E)$ be an $m$-edge $k$-multigraph and let $2\leq r\leq k-q$.
    By applying \Cref{lem:Dichotomy-Set-U}, $H$ has an $r$-cut of surplus $\Omega(m^{2/3-\eps})$ (and so we are done) unless there exists a vertex subset $U \subseteq V$ such that
\begin{itemize}
    \item $H[U]$ has $\Omega(m)$ edges,
    \item for every distinct $u,v \in U$, $\deg_H(v) \leq m^{2/3-\eps}$ and $\deg_H(u,v) \leq  m^{1/3-2\eps}$. 
\end{itemize}
    Let $\omega: V\rightarrow [r]$ be a uniform random $r$-cut of $V$, and let $P\omega$ be the function from $V$ to $\{ *,1, \dots, r\}$ such that
    \begin{align*}
        P\omega(v)= \begin{cases}
            \omega(v) & \text{if } v \in U^c \text{ or }  \omega(v) \in \{3, \dots, r\}, \\
           * & \text{otherwise}.
        \end{cases}
    \end{align*}
    Let $H_{P\omega}$ be the graph on vertex set $(P\omega)^{-1}(*)$ and edge set $\{ e \cap (P\omega)^{-1}(*): e \in E, \{3, \dots, r\} \subseteq P\omega(e) , \{1, \dots, r\} \not \subset  P\omega(e) \} $, where an edge $e$ is added twice if $P\omega(e) \cap \{1, \dots, r\} =  \{3, \dots, r\}$.
    By \Cref{lem:Reduction-Key-r-large}, we have 
    \begin{align}
    \label{eq:Reduction-Main-Small-r-Lem-1}
        \surp_{r}(H) = \Omega(\mathbb{E}_{\omega}[\surp(H_{P\omega})]). 
    \end{align}
    It is clear by the choice of $U$ that $H_{P\omega}$ has maximum degree $O(m^{2/3-\eps})$, maximum codegree $O(m^{1/3-2\eps})$ and $O(m)$ hyperedges in total. 
    For every edge $e \in E(H[U])$, we have that the probability that $|e \cap (P\omega)^{-1}(*)|=q+2$ and $P\omega(e) \cap \{1, \dots, r\}=\{3, \dots, r\}$ is $\Omega(1)$.
    Indeed, letting $v_1, \dots, v_k$ be the vertices that belong to $e$, the above event occurs if $\omega(v_i)=i+2$ for $1 \leq i \leq r-2$, $\omega(v_i)=*$ for $r-1 \leq i \leq r+q$, and $\omega(v_i)=r$ if $r+q+1 \leq i \leq k$.
    Therefore, letting $e_{q+2}(H_{P\omega})$ be the number of edges of $H_{P\omega}$ of size $q+2$, we have $\mathbb{E}_{\omega}[e_{q+2}(H_{P\omega})] = \Omega(m)$.
    As $e_{q+2}(H_{P\omega})=O(m)$ with probability $1$, there exists a constant $\delta=\delta(k) >0$ such that $\mathbb{P}( e_{q+2}(H_{P\omega}) \geq \delta m) \geq \delta$.
    By assumption of the lemma, we have
    \begin{align*}
        \surp(H_{P\omega}) \geq \Omega(m^{2/3-\eps} \ind_{e_{q+2}(H_{P\omega}) \geq \delta m}),
    \end{align*}
    and by taking expectation over $\omega$ and plugging this in \eqref{eq:Reduction-Main-Small-r-Lem-1}, we obtain
    \begin{align*}
        \surp_{r}(H) = \Omega(\mathbb{E}_{\omega}[\surp(H_{P\omega})]) = \Omega(m^{2/3-\eps} \mathbb{P}(e_{q+2}(H_{P\omega}) \geq \delta m)) = \Omega(m^{2/3-\eps}),
    \end{align*}
    as desired.
\end{proof}

\section{Proof of \Cref{lem:Large-3-cut-Linear-Hypergraph}}
\label{Appendix-Auxiliary-Large-3-cut-Linear-Hypergraph}
 
\begin{proof}[Proof of \Cref{lem:Large-3-cut-Linear-Hypergraph}]
    Let $G$ be the underlying multi-graph of $H$. 
    For every hyperedge $\{u,v,w\}$ in $E(H)$, we colour the copy of $\{u,v\} \in E(G)$ which comes from the hyperedge $\{u,v,w\}$ with colour $w$. 
    Let $S_0 \subseteq V(H)$ be such that $H[S_0]$ has $m$ edges, and each vertex in $S_0$ has degree at most $\Delta$ in $H$.
    If $B$ is a random subset of $S_0$ containing each vertex independently with probability $1/3$, then the expected number of edges in $H[S_0]$ with exactly one vertex in $B$ is $4m/9$.
    Therefore there exists a set $Q \subseteq S_0$ such that at least $4m/9$ edges of $H[S_0]$ contain exactly one vertex of $Q$. Let $S = S_0\setminus Q$. \\
    We now sample the vertices of $H$ independently with probability $p = 1/3$, and we let $T$ be the set of sampled
    vertices. 
    We first condition on $T \setminus Q$, in other words, we reveal $T$ outside of $Q$, and we treat $T \cap Q$ as the source of randomness. 
    We introduce the following notation.
    \begin{itemize}
        \item $V = S \setminus T$,
        \item $G_0$ is the graph on vertex set $V$ which consists of the edges of $G[V]$ whose colour is in $Q$,
        \item $G^{*}_0$ is the subgraph of $G_0$ in which we only keep the edges whose colour is in $T \cap Q$,
        \item $G_1$ is the graph on vertex set $V$ which consists of the edges of $G[V]$ whose colour is in $V(H) \setminus Q$,
        \item $G^{*}_1$ is the subgraph of $G_1$ in which we only keep the edges whose colour is in $T \setminus Q$,
        \item $G^{*}$ is the subgraph of $G[T^c]$ in which we only keep the edges whose colour is in $T$ (where $T^c=V(H) \setminus T$),
        \item $\tilde{G} = G^*[V]=G_0^*\cup G_1^*$.
    \end{itemize}
    Now we have that $\surp(G^*) \geq \surp(\tilde{G})$ by \Cref{lem:Surplus-Ineq-Induced-Subgraph}.
    Conditionally on $T \setminus Q$, we have that $\tilde{G}$ is a random $(m,2\Delta,2)$-multihypergraph of linear type, provided that $f=e(G_0)$ satisfies $f \geq m/10$.
    Indeed, it suffices to verify that all conditions in Definition \ref{defn:random linear} are satisfied.
    \begin{enumerate}[label=(\roman*)]
        \item Let $X=E(G_0)\cup E(G_1)$ and let $R=E(G_0^*)\cup E(G_1^*)$. For every hyperedge $e \in E(G_0)$ of colour $w$, we let $S_e = \{ w\}$; we then have that $e \in E(G^{*}_0)$ if and only if $Z_w =1$, where $Z_w$ is the indicator of the event $w \in T$.
        For every hyperedge $e \in E(G_1)$, we let $S_e = \emptyset$, and whether $e \in  E(G^{*}_1)$ is already known from $T \setminus Q$.
        As $\tilde{G} = G^{*}_0 \cup G^{*}_1$, this concludes this item.
        \item This item is trivial.
        \item Letting $Y$ be the set of hyperedges $e \in E(G_0)$, then we have $|Y| = f \geq m/10$, and for every $e \in Y$, we have $|e| = 2$ and $\mathbb{P}(e \in R)=1/3$.
        \item This item is trivial.
        \item This item is trivial.
    \end{enumerate}
    Thus, by \Cref{lem:Key-Random-Linear}, we have
    \begin{align*}
        \mathbb{E}_{T}[\surp(G^*) | T \setminus Q] \geq \mathbb{E}_{T}[\surp(\tilde{G}) | T \setminus Q] = \Omega\left( \frac{m \ind_{f \geq m/10}}{\sqrt{\Delta}}\right).
    \end{align*}
    Taking expectation, we get
    \begin{align*}
        \mathbb{E}_T[\surp(G^*)] = \Omega\left( \frac{m \cdot \mathbb{P}(f \geq m/10)}{\sqrt{\Delta}}\right).
    \end{align*}
    Recall that $f$ is the number of edges in $G[S \setminus T]$ whose colour is in $Q$.
    Therefore, every edge of $G[S]$ whose colour is in $Q$ is counted with probability $(1-p)^2 = 4/9$, and thus $\mathbb{E}_T[f] \geq 16m/81$.
    Since $f=O(m)$ almost surely, it follows that $\mathbb{P}(f \geq m/10) = \Omega(1)$, and plugging this in the previous display gives 
    \begin{align*}
        \mathbb{E}_T[\surp(G^*)] = \Omega\left( \frac{m}{\sqrt{\Delta}}\right).
    \end{align*}
    Let $W_1,W_2$ be a partition of $T^c$ such that $e_{G^*}(W_1,W_2) = \frac{1}{2}e(G^*) + \surp(G^*)$.
    Since $e_H(T,W_1,W_2)=e_{G^*}(W_1,W_2)$, we obtain
    \begin{align*}
         \mathbb{E}_T[e_H(T,W_1,W_2)]=\mathbb{E}_T[e_{G^*}(W_1,W_2)]=\mathbb{E}_T\left[\frac{1}{2}e(G^*) + \surp(G^*)\right]=\frac{2}{9}e(H)+\Omega\left( \frac{m}{\sqrt{\Delta}}\right),
    \end{align*}
    so $H$ has a $3$-cut of surplus $\Omega(\frac{m}{\sqrt{\Delta}})$.
\end{proof}

\section{Proof of \Cref{lem:Large-2-cut-Linear-Hypergraph}}
\label{Appendix-Proof-Auxiliary-Large-2-cut-Linear-Hypergraph}

\begin{proof}[Proof of \Cref{lem:Large-2-cut-Linear-Hypergraph}]
    Let $S$ be a subset of $V(H)$ such that there are at least $m$ edges $e\in E(H)$ with $|e| \geq 4$ and $|e \cap S| \geq 2$, and each vertex in $S$ has degree at most $\Delta$ in $H$.
    We start by taking a random subset $W$ of $S$ by selecting each vertex with probability $1/2$ independently, and let $W^c = V(H) \setminus W$.
    For every hyperedge $e \in E(H)$ such that $|e| \geq 4$ and $|e \cap S| \geq 2$, the event $\{|e \cap W| \geq 2, |e \cap W^c| \geq 2\}$ happens with probability $\Omega(1)$.
    Therefore we may fix for the rest of the proof an outcome of $W$ such that the number of hyperedges $e\in E(H)$ satisfying $|e \cap W| \geq 2$ and $|e \cap W^c| \geq 2$ is $\Omega(m)$.
    We put each vertex $v \notin W$ in part $1$ or part $2$ of the cut uniformly at random and independently, and we let $\sigma$ be this assignment.
    We now construct the following auxiliary multihypergraph $H_{\sigma}$ on vertex set $W$: for every edge $e \in E(H)$, if all vertices of $e$ are in $W$ then we add two copies of $e$, otherwise, if $|\sigma(e \setminus W)| = 1$ and $|e \cap W| > 0$, we add one copy of $e \cap W$.
    By \Cref{lem:Reduction-Reveal-Outside-W}, we have
    \begin{align}
    \label{eq:Surplus-Ineq-Hsigma}
        \surp(H) \geq \mathbb{E}_{\sigma}[\surp(H_{\sigma})]/2
    \end{align}
    
    We note that the graph $H_{\sigma}$ is a random $(m,2\Delta,k)$-multihypergraph of linear type.
    Indeed, it suffices to verify that all conditions in Definition \ref{defn:random linear} are satisfied.
    \begin{enumerate}[label=(\roman*)]
        \item We let $X= \{e \cap W : e \in E(H)\}$, where, if $e \subset W$, the element $e \cap W=e$ appears with multiplicity $2$.  
        For every hyperedge $e \in E(H)$ with representative $e_W =e\cap W \in X$, we set $S_{e_W} = e \setminus W$, and for every vertex $a \in W^c$, we let $Z_a$ be the random variable associated to which part of the cut $a$ has been assigned by $\sigma$. Then it is clear that the event $\{ e_W \in E(H_{\sigma}) \}$ is a function of $\{ Z_i : i \in S_{e_W} \}$. 
        \item For two edges $e,f \in X$, we have $|S_e \cap S_f| \leq 1$ since $H$ is a linear hypergraph. It is then clear that we have $\mathbb{P}(e \in R,f \in R)=\mathbb{P}(e \in R)\mathbb{P}(f \in R) $. 
        \item Let $Y = \{ e \cap W : e \in E(H), |e \cap W| \geq 2,  |e \cap W^c| \geq 2 \}$. We have by choice of $W$ that $|Y| = \Omega(m)$, and it is easy to check that for every $e \in Y$, we have $|e| \geq 2$ and $\Omega(1)=\mathbb{P}(e \in E(H_{\sigma}))=1-\Omega(1)$.
        \item This item is trivial.
        \item This item is trivial.
    \end{enumerate}
    By \Cref{lem:Key-Random-Linear}, we have
    \begin{align*}
        \mathbb{E_{\sigma}}[\surp(H_{\sigma})] = \Omega\left( \frac{m}{\sqrt{\Delta}}\right),
    \end{align*}
    which, together with \eqref{eq:Surplus-Ineq-Hsigma}, finishes the proof.
\end{proof}

\end{document}